\documentclass[12pt,english]{article}
\usepackage[T1]{fontenc}
\usepackage[latin1]{inputenc}
\usepackage{verbatim}
\usepackage{amsmath}
\usepackage{amssymb}
\usepackage{esint}

\makeatletter
\usepackage{amsfonts}

\usepackage{amscd}

\newtheorem{theorem}{Theorem}[section]\newtheorem{lemma}[theorem]{Lemma}\newtheorem{proposition}[theorem]{Proposition}\newtheorem{example}[theorem]{Example}\newtheorem{definition}[theorem]{Definition}\newtheorem{corollary}[theorem]{Corollary}\newtheorem{remark}[theorem]{Remark}\newtheorem{problem}[theorem]{Problem}

\newenvironment{proof}{\bf Proof. \rm}{$\Box$}

\newcommand{\be}{\begin{equation}}\newcommand{\ee}{\end{equation}}

\newcommand{\norm}[1]{\Vert #1
\Vert}\newcommand{\cA}{\mathcal{A}}\newcommand{\cB}{\mathcal{B}}\newcommand{\cE}{\mathcal{E}}\newcommand{\cM}{\mathcal{M}}\newcommand{\cF}{\mathcal{F}}\newcommand{\cL}{
\mathcal{L}}

\newcounter{parno}[paragraph]
\renewcommand{\theparno}{\thesection.\arabic{parno}}\newcommand{\p}{\refstepcounter{parno}\noindent{\textbf{\theparno}}\ }

\newcounter{subparno}[parno]
\renewcommand{\thesubparno}{\thesection.\arabic{parno}.\arabic{subparno}}\newcommand{\subp}{\refstepcounter{subparno}\noindent{\textbf{\thesubparno}}\ }

\makeatother

\usepackage{babel}

\makeatother

\usepackage{babel}

\makeatother

\usepackage{babel}

\begin{document}

\title{Representations of the Hardy Algebra: Absolute Continuity, Intertwiners,
and Superharmonic Operators}

\author{Paul S. Muhly%
\thanks{Supported in part by a grant from the U.S.-Israel Binational Science
Foundation.%
}\\
 Department of Mathematics\\
 University of Iowa\\
 Iowa City, IA 52242\\
 e-mail: muhly@math.uiowa.edu \and Baruch Solel%
\thanks{Supported in part by the U.S.-Israel Binational Science Foundation
and by the Lowengart Research Fund.%
}\\
 Department of Mathematics\\
 Technion\\
 32000 Haifa, Israel\\
 e-mail: mabaruch@techunix.technion.ac.il}
\maketitle
\begin{abstract}
Suppose $\mathcal{T}_{+}(E)$ is the tensor algebra of a $W^{*}$-correspondence
$E$ and $H^{\infty}(E)$ is the associated Hardy algebra. We investigate
the problem of extending completely contractive representations of
$\mathcal{T}_{+}(E)$ on a Hilbert space to ultra-weakly continuous
completely contractive representations of $H^{\infty}(E)$ on the
same Hilbert space. Our work extends the classical Sz.-Nagy - Foia\c{s}
functional calculus and more recent work by Davidson, Li and Pitts
on the representation theory of Popescu's noncommutative disc algebra.
\end{abstract}

\section{Introduction\label{sec:Introduction}}

Suppose $\rho$ is a contractive representation of the disc algebra
$A(\mathbb{D})$ on a Hilbert space $H$, i.e., suppose $\Vert\rho(f)\Vert_{B(H)}\leq\Vert f\Vert_{\infty}$,
where $\Vert\cdot\Vert_{B(H)}$ is the operator norm on the space
of bounded operators on $H$, $B(H)$, and where $\Vert\cdot\Vert_{\infty}$
is the sup norm on $A(\mathbb{D})$, taken over $\overline{\mathbb{D}}$.
(By the maximum modulus principle, the supremum needs only to be evaluated
over the circle, $\mathbb{T}$.) Then $\rho$ is completely determined
by its value at the identity function $z$ in $A(\mathbb{D})$, $T:=\rho(z)$.
Of course $T$ is a contraction operator in $B(H)$. On the other
hand, given a contraction operator $T$ in $B(H)$, then von Neumann's
inequality guarantees that there is a unique contractive representation
$\rho$ of $A(\mathbb{D})$ in $B(H)$ such that $T=\rho(z)$. A natural
question arises: When does $\rho$ extend to a representation of $H^{\infty}(\mathbb{T})$
in $B(H)$ that is continuous with respect to the weak-$*$ topology
on $H^{\infty}(\mathbb{T})$ and the weak-$*$ topology on $B(H)$?
(We follow the convention of calling the weak-$*$ topology on $B(H)$
the ultra-weak topology.) Thanks to the Sz.-Nagy - Foia\c{s} functional
calculus \cite{SzNF70}, a neat succinct answer may be given in terms
of $T$, viz., $\rho$ admits such an extension if and only if the
unitary part of $T$ is absolutely continuous. In a bit more detail,
recall that an arbitrary contraction operator $T$ on a Hilbert space
$H$ decomposes uniquely into the direct sum $T=T_{cnu}\oplus U$,
where $T_{cnu}$ is completely non unitary, meaning that there are
no invariant subspaces for $T_{cnu}$ on which $T_{cnu}$ acts as
a unitary operator, and where $U$ is unitary. Thus the answer to
the question is: $\rho$ extends if and only if the spectral measure
for $U$ is absolutely continuous with respect to Lebesgue measure
on $\mathbb{T}$. The assertion that extension is possible when $U$
is absolutely continuous is \cite[Theorem III.2.1]{SzNF70}. The assertion
that if the extension is possible, then $U$ is absolutely continuous
is essentially \cite[Theorem III.2.3]{SzNF70}. Note in particular
that since the eigenspace of any eigenvalue for $T$ of modulus one
must reduce $T$, it follows that when $H$ is finite dimensional
$\rho$ extends to $H^{\infty}(\mathbb{T})$ as a weak-$*$ continuous
representation if and only if the spectral radius of $T$ is strictly
less than one. And when $\dim(H)=1$, we recover the well-known fact
that a character of $A(\mathbb{D})$ extends to a weak-$*$ continuous
character of $H^{\infty}(\mathbb{T})$ if and only if it comes from
a point in $\mathbb{D}$.

We were drawn to thinking about this perspective on the Sz.-Nagy -
Foia\c{s} functional calculus by recent work we have done in the theory
of tensor and Hardy algebras. Suppose $M$ is a $W^{*}$-algebra and
that $E$ is a $W^{*}$-correspondence over $M$ in the sense of \cite{MSHardy}.
Then, in a fashion that will be discussed more thoroughly in the next
section, one can form both the tensor algebra of $E$, $\mathcal{T}_{+}(E)$,
and its ultra-weak closure, the Hardy algebra of $E$, $H^{\infty}(E)$.
If $M=\mathbb{C}=E$, then $\mathcal{T}_{+}(E)=A(\mathbb{D})$ and
$H^{\infty}(E)=H^{\infty}(\mathbb{T})$. Every completely contractive
representation $\rho:\mathcal{T}_{+}(E)\rightarrow B(H)$ of $\mathcal{T}_{+}(E)$
on a Hilbert space $H$ with the property that $\rho$ restricted
to (the copy of) $M$ in $\mathcal{T}_{+}(E)$ is a normal representation
of $M$ on $H$, that we denote by $\sigma$, is determined uniquely
by a contraction operator $\widetilde{T}:E\otimes_{\sigma}H\to H$
satisfying the intertwining equation \begin{equation}
\widetilde{T}\sigma^{E}\circ\varphi(\cdot)=\sigma(\cdot)\widetilde{T},\label{eq:intertwining}\end{equation}
where $\varphi$ gives the left action of $M$ on $E$ and where $\sigma^{E}$
is the induced representation of $\mathcal{L}(E)$ on $E\otimes_{\sigma}H$
defined by the formula $\sigma^{E}(a)=a\otimes I_{H}$, $a\in\mathcal{L}(E)$.
And conversely, once $\sigma$ is fixed, each contraction $\widetilde{T}$
satisfying this equation determines a completely contractive representation
of $\mathcal{T}_{+}(E)$. We write $\rho=T\times\sigma$. The question
we wanted to address, and which we will discuss here, is:
\begin{quote}
\emph{What conditions must $\widetilde{T}$ satisfy so that $T\times\sigma$
extends from $\mathcal{T}_{+}(E)$ to an ultra-weakly continuous representation
of $H^{\infty}(E)$?}
\end{quote}
It is easy to see that if $\Vert\widetilde{T}\Vert<1$, then $T\times\sigma$
extends from $\mathcal{T}_{+}(E)$ to an ultra-weakly continuous representation
of $H^{\infty}(E)$ \cite[Corollary 2.14]{MSHardy}. Thus, the question
is really about operators $\widetilde{T}$ that have norm equal to
one. With quite a bit more work we showed in \cite[Theorem 7.3]{MSHardy}
that if $\widetilde{T}$ is \emph{completely non-coisometric}, meaning
that there is no subspace of $H$ that is invariant under $T\times\sigma(\mathcal{T}_{+}(E))^{*}$
to which $\widetilde{T}^{*}$ restricts yielding an isometry mapping
to $E\otimes_{\sigma}H$, then $T\times\sigma$ extends to an ultra-weakly
continuous representation of $H^{\infty}(E)$. Thus it looks like
we are well on the way to generalizing the theorems of Sz.-Nagy and
Foia\c{s} that we cited above. ``All we need is a good generalization
of the notion of a completely non-unitary contraction and a good generalization
of an absolutely continuous unitary operator.'' It turns out, however,
that things are not this simple. While a natural generalization of
a unitary operator is a representation $T\times\sigma$, where $\widetilde{T}$
is a Hilbert space isomorphism, it is not quite so clear what it means
for $\widetilde{T}$ to be absolutely continuous. There is no evident
notion of spectral measure for $\widetilde{T}$ in this case. Further,
in the Sz.-Nagy - Foia\c{s} theory, it is important to know about
the minimal unitary extension of the minimal isometric dilation of
the contraction $T$, i.e., it is important to know about the minimal
unitary dilation of $T$. However, it turns out in the theory we are
describing, while there is always a unique (up to unitary equivalence)
minimal isometric dilation of $\widetilde{T}$ there may be many ``unitary''
extensions of the isometric dilation. Straightforward definitions
and results do not appear to exist.

We are not the first to ponder our basic question. We have received
a lot of inspiration from two important papers: \cite{DKP} and \cite{DLP}.
In \cite{DKP}, Davidson, Katsoulis and Pitts weren't directly involved
with this question, but they clearly were influenced by it. They considered
the situation where $M=\mathbb{C}$ and $E=\mathbb{C}^{d}$ for a
suitable $d$. (When $d=\infty$, we view $\mathbb{C}^{d}$ as $\ell^{2}(\mathbb{N})$.)
The tensor algebra, $\mathcal{T}_{+}(\mathbb{C}^{d})$, in this case
is the norm-closed algebra generated by the creation operators on
the full Fock space $\mathcal{F}(\mathbb{C}^{d})$. We fix an orthonormal
basis $\{e_{i}\}_{i=1}^{d}$ for $\mathbb{C}^{d}$ and let $L_{i}$
be the creation operator of tensoring with $e_{i}$. Thus $L_{i}\eta=e_{i}\otimes\eta$
for all $\eta\in\mathcal{F}(\mathbb{C}^{d})$. Then $\mathcal{T}_{+}(\mathbb{C}^{d})$
is generated by the $L_{i}$ and coincides with Popescu's \emph{noncommutative
disc algebra}, denoted $\mathfrak{A}_{d}$. The weakly closed algebra
generated by the $L_{i}$ is denoted $\mathcal{L}_{d}$ and is called
the \emph{noncommutative analytic Toeplitz algebra}. It turns out
that $\mathcal{L}_{d}$ is also the ultra-weak closure of $\mathcal{T}_{+}(\mathbb{C}^{d})$,
and so $\mathcal{L}_{d}$ coincides with $H^{\infty}(\mathbb{C}^{d})$.
In their setting $\mathcal{T}_{+}(\mathbb{C}^{d})$ and $H^{\infty}(\mathbb{C}^{d})$
are concretely defined operator algebras since $\mathcal{F}(\mathbb{C}^{d})$
is a Hilbert space. We single out this special representation of $\mathcal{T}_{+}(\mathbb{C}^{d})$
and $H^{\infty}(\mathbb{C}^{d})$ with the notation $\lambda$ - for
left regular representation, which it is, if $\mathcal{F}(\mathbb{C}^{d})$
is identified with the $\ell^{2}$-space of the free semigroup on
$d$ generators through a choice of basis. If $\rho=T\times\sigma$
is completely contractive representation of $\mathcal{T}_{+}(\mathbb{C}^{d})$
on $H$, then $\sigma$ must be a multiple of the identity representation
of $\mathbb{C}$, namely the Hilbert space dimension of $H$. Also,
$\widetilde{T}$ is simply the row $d$-tuple of operators, $(T_{1},T_{2},\cdots,T_{d}),$
where $T_{i}=\rho(L_{i})$. As an operator from $\mathbb{C}^{d}\otimes H$
to $H$, $\widetilde{T}$ has norm at most $1$; that is $(T_{1},T_{2},\cdots,T_{d})$
a \emph{row contraction}. The equation \eqref{eq:intertwining} is
automatic in this case. Davidson, Katsoulis and Pitts assume in their
work that $(T_{1},T_{2},\cdots,T_{d})$ is a row isometry, i.e., that
$\widetilde{T}$ is an isometry. They are interested in how $(T_{1},T_{2},\cdots,T_{d})$
relates to $(L_{1},L_{2},\cdots,L_{d})$. For this purpose they let
$\mathcal{S}$ be the weakly closed subalgebra of $B(H)$ that is
generated by $\{T_{1},T_{2},\cdots,T_{d}\}$ and the identity, and
they let $\mathcal{S}_{0}$ be the weakly closed \emph{ideal} in $\mathcal{S}$
generated by $\{T_{1},T_{2},\cdots,T_{d}\}$. Their principal result
is \cite[Theorem 2.6]{DKP}, which they call \emph{The Structure Theorem}.
It asserts that if $N$ denotes the von Neumann algebra generated
by $\mathcal{S}$, and if $p$ is the largest projection in $N$ such
that $p\mathcal{S}p$ is self-adjoint, then
\begin{enumerate}
\item $Np=\bigcap_{k\geq1}\mathcal{S}_{0}^{k}$;
\item $\mathcal{S}p=Np$, so in particular, $p\mathcal{S}p=pNp$;
\item $p^{\perp}H$ is invariant under $\mathcal{S}$ and $\mathcal{S}=Np+p^{\perp}\mathcal{S}p^{\perp}$;
and
\item assuming $p\neq I$, $p^{\perp}\mathcal{S}p^{\perp}$ is completely
isometrically isomorphic and ultra-weakly homeomorphic to $\mathcal{L}_{d}$.
\end{enumerate}
Since $\mathcal{L}_{d}=H^{\infty}(\mathbb{C}^{d})$, it follows that
if $p=0$, then the representation of $\mathcal{T}_{+}(\mathbb{C}^{d})$
determined by the tuple $(T_{1},T_{2},\cdots,T_{d})$ extends to $H^{\infty}(\mathbb{C}^{d})$
as an weak-$*$ continuous representation of $H^{\infty}(\mathbb{C}^{d})$.
If $p\neq0$, then the representation may still extend to $H^{\infty}(\mathbb{C}^{d})$,
but the matter becomes more subtle.

As the authors of \cite{DKP} observe, this decomposition, is suggestive
of certain aspects of absolute continuity in the setting of a single
isometry. This point is taken up in \cite{DLP}, where Davidson, Li
and Pitts say that a vector $x$ in the Hilbert space of $(T_{1},T_{2},\cdots,T_{d})$,
$H$, is \emph{absolutely continuous} if the vector functional on
$\mathcal{T}_{+}(\mathbb{C}^{d})$ it determines can be represented
by a vector functional on $\mathcal{L}_{d}$, i.e., if there are vectors
$\xi,\eta\in\mathcal{F}(\mathbb{C}^{d})$ such that $(\rho(a)x,x)=(\lambda(a)\xi,\eta)$
for all $a\in\mathcal{T}_{+}(\mathbb{C}^{d})$. The collection of
all such vectors $x$ is denoted $\mathcal{V}_{ac}(\rho)$. This set
is, in fact, a closed subspace of $H$, and the representation $\rho$
extends to $H^{\infty}(\mathbb{C}^{d})$ as an ultraweakly continuous
representation if and only if $\mathcal{V}_{ac}(\rho)=H$. One of
their main results is \cite[Theorem 3.4]{DLP}, which implies that
$\mathcal{V}_{ac}(\rho)=H$ if and only if the structure projection
for the representation determined by $(L_{1}\oplus T_{1},L_{2}\oplus T_{2},\cdots,L_{d}\oplus T_{d})$
acting on $\mathcal{F}(\mathbb{C}^{d})\oplus H$ is zero.

A central role is played in \cite{DLP} by the operators that intertwine
$\lambda$ and $\rho$, i.e., operators $X:\mathcal{F}(\mathbb{C}^{d})\to H$
that satisfy the equation $\rho(a)X=X\lambda(a)$, for all $a\in\mathcal{T}_{+}(\mathbb{C}^{d})$.
Theorem 2.7 of \cite{DLP} shows that $\mathcal{V}_{ac}(\rho)$ is
the union of the ranges of the $X$'s that intertwine $\lambda$ and
$\rho$. In an aside \cite[Remark 2.12]{DLP}, the authors note that
Popescu \cite[Theorem 3.8]{gP03} has shown that if $X$ is an intertwiner
then $XX^{*}$ is a nonnegative operator on $H$ that satisfies the
two conditions:\begin{equation}
\Phi(XX^{*})\leq XX^{*}\label{eq:Popescu_Intertwiner1}\end{equation}
and\begin{equation}
\Phi^{k}(XX^{*})\to0\label{eq:Popescu_Intertwiner2}\end{equation}
in the strong operator topology, where $\Phi(Q):=\widetilde{T}Q\widetilde{T}^{*}$,
and conversely every nonnegative operator $Q$ on $H$ that satisfies
\eqref{eq:Popescu_Intertwiner1} and \eqref{eq:Popescu_Intertwiner2}
can be factored as $Q=XX^{*}$, where $X$ is an intertwiner.

After contemplating this connection between \cite{DLP} and \cite{gP03},
we realized that there is a very tight connection among all the various
constructs we have discussed and that they all can be generalized
to our setting of tensor and Hardy algebras associated to a $W^{*}$-correspondence.
This is what we do here. In the next section, we draw together a number
of facts that we will use in the sequel. Most are known from the literature.
In  Section \eqref{sec:Absolute-Continuity} we develop the notion
of absolute continuity first for isometric representations of $\mathcal{T}_{+}(E)$,
i.e. for representations $\rho=T\times\sigma$ where $\widetilde{T}$
is an isometry. In Section \ref{sec:CCReps and CPMaps}, we study
absolute continuity in the context of an arbitrary completely contractive
representation $\rho$. Here we show that $\rho$ extends from $\mathcal{T}_{+}(E)$
to an ultraweakly continuous completely contractive representation
of $H^{\infty}(E)$ if and only if $\rho$ is absolutely continuous,
i.e., if and only if $\mathcal{V}_{ac}(\rho)=H$. It turns out that
the absolutely continuous subspace $\mathcal{V}_{ac}(\rho)$ is really
an artifact of the completely positive map attached to $\widetilde{T}$.
This fact, coupled to our work in \cite{MS02}, which shows that every
completely positive map on a von Neumann algebra gives rise to a $W^{*}$-correspondence
and a representation of it, enables us to formulate a notion of absolute
continuity for an arbitrary completely positive map. This formulation
is made in Section \ref{sec:Further-Corollaries}, where other corollaries
of Sections \ref{sec:Absolute-Continuity} and \ref{sec:CCReps and CPMaps}
are drawn. Section \ref{sec:Induced-Representations} is something
of an interlude, where we deal with an issue that does not arise in
\cite{DKP,DLP}. The analogue of the representation $\lambda$ in
our theory is a representation of $\mathcal{T}_{+}(E)$ that is induced
from a representation of $M$ in the sense of Rieffel \cite{mR74a,mR74b}.
Owing to the possibility that the center of $M$ is non-trivial, an
induced representation need not be faithful. This fact creates a number
of technical problems for us with which we deal in Section \ref{sec:Induced-Representations}.
The final section, Section \ref{sec:The-Structure-Theorem}, is devoted
to our generalization of the Structure Theorem of Davidson, Katsoulis
and Pitts \cite[Theorem 2.6]{DKP} and to its connection with the
notion of absolute continuity.

\section{Background and Preliminaries\label{sec:Preliminaries}}

It will be helpful to have at our disposal a number of facts developed
in the literature. Our presentation is only a survey, and a little
discontinuous. Certainly, it is not comprehensive, but we have given
labels to paragraphs for easy reference in the body of the paper.

\p \label{par:Wstar-algebras} Throughout this paper, $M$ will denote
a fixed $W^{*}$-algebra. We do not preclude the possibility that
$M$ may be finite dimensional. Indeed, as we have indicated, the
situation when $M=\mathbb{C}$ can be very interesting. However, we
want to think of $M$ abstractly, as a $C^{*}$-algebra that is a
dual space, without regard to any Hilbert space on which $M$ might
be represented. We will reserve the term {}``von Neumann algebra''
for a concretely represented $W^{*}$-algebra. The weak-$*$ topology
on a $W^{*}$-algebra or on any of its weak-$*$ closed subspaces
will be referred to as the \emph{ultra-weak} topology. If $S$ is
a subset of a $W^{*}$-algebra, we shall write $\overline{S}^{u-w}$
for its ultra-weak closure.

To eliminate unnecessary technicalities we shall always assume $M$
is $\sigma$-finite in the sense that every family of mutually orthogonal
projections in $M$ is countable. Alternatively, to say $M$ is $\sigma$-finite
is to say that $M$ has a faithful normal representation on a separable
Hilbert space. So, unless explicitly indicated otherwise, every Hilbert
space we consider will be assumed to be separable.


\p \label{par:WstarCorsOverWstarAlgs} In addition, $E$ will denote
a $W^{*}$-correspondence over $M$ in the sense of \cite{MSHardy}.
This means first that $E$ is a (right) Hilbert $C^{*}$-module over
$M$ that is self-dual in the sense that each (right) module map $\Phi$
from $E$ into $M$ is induced by a vector in $E$, i.e., there is
an $\eta\in E$ such that $\Phi(\xi)=\langle\eta,\xi\rangle$, for
all $\xi\in E$. Our basic reference for Hilbert $C^{*}$- and $W^{*}$-
modules is \cite{MT05}. It is shown in \cite[Proposition 3.3.4]{MT05}
that when $E$ is a self-dual Hilbert module over a $W^{*}$-algebra
$M$, then $E$ must be a dual space. In fact, it may be viewed as
an ultra-weakly closed subspace of a $W^{*}$-algebra. Further, every
continuous module map on $E$ is adjointable \cite[Corollary 3.3.2]{MT05}
and the algebra $\mathcal{L}(E)$ consisting of all continuous module
maps on $E$ is a $W^{*}$-algebra \cite[Proposition 3.3.4]{MT05}.
To say that $E$ is a $W^{*}$-correspondence over $M$ is to say,
then, that $E$ is a self-dual Hilbert module over $M$ and that there
is an ultra-weakly continuous $*$-representation $\varphi:M\to\mathcal{L}(E)$
such that $E$ becomes a bimodule over $M$ where the left action
of $M$ is determined by $\varphi$, $a\cdot\xi=\varphi(a)\xi$. We
shall assume that $E$ is \emph{essential} or \emph{non-degenerate}
as a left $M$-module. This is the same as assuming that $\varphi$
is unital. We also shall assume that our $W^{*}$-correspondences
are countably generated as self-dual Hilbert modules over their coefficient
algebras. This is equivalent to assuming that $\mathcal{L}(E)$ is
$\sigma$-finite.


\p \label{par:Intertwining} In this paper, we will be studying objects
of various kinds, algebras, $*$-algebras, modules, etc. and we will
be considering various types of linear maps of such objects to spaces
of bounded operators on Hilbert spaces. We will be especially interested
in spaces of intertwining operators between such maps. For this purpose,
we introduce the following notation. Suppose $\mathcal{X}$ is an
object of one of the various kinds we are considering in this paper,
e.g., an algebra, a bimodule, etc., and suppose that for $i=1,2$,
$\rho_{i}:\mathcal{X}\to B(H_{i})$ is a map of $\mathcal{X}$ to
bounded linear operators on the Hilbert space $H_{i}$. Then we write
$\mathcal{I}(\rho_{1},\rho_{2})$ for the space of all bounded linear
operators $X:H_{1}\to H_{2}$ such that $X\rho_{1}(\xi)=\rho_{2}(\xi)X$
for all $\xi\in\mathcal{X}$. That is, $\mathcal{I}(\rho_{1},\rho_{2})$
is the \emph{intertwining space} or the space of all \emph{intertwiners}
of $\rho_{1}$ and $\rho_{2}$. If $\mathcal{X}$ is a $C^{*}$-algebra
and $\rho_{1}$ and $\rho_{2}$ are $C^{*}$-representations, then
$\mathcal{I}(\rho_{1},\rho_{2})$ has the structure of a Hilbert $W^{*}$-module
over the commutant of $\rho_{1}(\mathcal{X})$, $\rho_{1}(\mathcal{X})'$.
The $\rho_{1}(\mathcal{X})'$-valued inner product on $\mathcal{I}(\rho_{1},\rho_{2})$
is given by the formula $\langle X,Y\rangle:=X^{*}Y$ and the right
action of $\rho(E)'$ on $\mathcal{I}(\rho_{1},\rho_{2})$ is given
by the formula $X\cdot a:=Xa$, $X\in\mathcal{I}(\rho_{1},\rho_{2})$,
$a\in\rho(\mathcal{X})'$. Of course, the product, $Xa$, is just
the composition of operators. It is quite clear that $\mathcal{I}(\rho_{1},\rho_{2})$
with these operations is a Hilbert $C^{*}$-module over $\rho(\mathcal{X})'$,
but what makes it a self-dual Hilbert module is the fact that it is
ultra-weakly closed in $B(H_{1},H_{2})$. See \cite[Theorem 3.5.1]{MT05}.

If $\mathcal{X}$ is a $C^{*}$-algebra and $\rho_{1}$ and $\rho_{2}$
are two $C^{*}$-representations, then it is clear that if $\mathcal{I}(\rho_{1},\rho_{2})^{*}$
denotes the space $\{X^{*}\mid X\in\mathcal{I}(\rho_{1},\rho_{2})\}$,
then $\mathcal{I}(\rho_{1},\rho_{2})^{*}=\mathcal{I}(\rho_{2},\rho_{1})$.
Consequently, either $\mathcal{I}(\rho_{1},\rho_{2})$ and $\mathcal{I}(\rho_{2},\rho_{1})$
are both nonzero or both are zero (in the latter case, $\rho_{1}$
and $\rho_{2}$ are called \emph{disjoint}.) However, there are situations
that we shall encounter (see Remark \ref{rem:Nonselfadjoint_vageries}.)
where $\mathcal{X}$ is an operator algebra and the $\rho_{i}$ are
completely contractive representations of $\mathcal{X}$ with the
property $\mathcal{I}(\rho_{1},\rho_{2})$ is nonzero, but $\mathcal{I}(\rho_{2},\rho_{1})=\{0\}$.
Thus, in the non-self-adjoint setting, one has to take extra care
when manipulating intertwining spaces.


\p \label{par:WstarCorsBetwWstarAlgs} The concept of a $W^{*}$-correspondence
over a $W^{*}$-algebra is really a special case of the very useful
more general notion of a $W^{*}$-correspondence from one $W^{*}$-algebra
to another. Specifically, if $M_{1}$ and $M_{2}$ are $W^{*}$-algebras,
a \emph{$W^{*}$-correspondence from $M_{1}$ to $M_{2}$ }is a self-dual
Hilbert module $F$ over $M_{2}$ endowed with a normal representation
$\varphi:M_{1}\to\mathcal{L}(F)$. We always assume that $\varphi$
is unital. In particular, every normal representation of $M$ on a
Hilbert space $H$ makes $H$ a $W^{*}$-correspondence from $M$
to $\mathbb{C}$, and conversely, a $W^{*}$-correspondence from $M$
to $\mathbb{C}$ is just a normal representation of $M$ on a Hilbert
space. Also, observe that if $\rho_{1}$ and $\rho_{2}$ are two representations
of a $C^{*}$-algebra $A$ on Hilbert spaces $H_{1}$ and $H_{2}$,
respectively, then $\mathcal{I}(\rho_{1},\rho_{2})$ is a $W^{*}$-correspondence
from $\rho_{2}(A)'$ to $\rho_{1}(A)'$. We know already that $\mathcal{I}(\rho_{1},\rho_{2})$
is a self-dual Hilbert $C^{*}$-module over $\rho_{1}(A)'$. The left
action of $\rho_{2}(A)'$ on $\mathcal{I}(\rho_{1},\rho_{2})$ is
given by composition of maps: $b\cdot X:=bX$ for all $b\in\rho_{2}(A)'$
and all $X\in\mathcal{I}(\rho_{1},\rho_{2})$.

Correspondences can be {}``composed'' through the process of balanced
tensor product, but a little care must be taken. That is, if $F$
is a $W^{*}$-correspondence from $M_{1}$ to $M_{2}$ and if $G$
is a $W^{*}$-correspondence from $M_{2}$ to $M_{3}$ then one can
form their balanced $C^{*}$-tensor product, balanced over $M_{2}$,
but in general it won't be a $W^{*}$-correspondence from $M_{1}$
to $M_{3}$. So for us, the tensor product $F\otimes_{M_{2}}G$, balanced
over $M_{2}$, will be the unique self-dual completion of the balanced
$C^{*}$-tensor product of $F$ and $G$. (See \cite[Theorem 3.2.1]{MT05}.)
It will be a $W^{*}$-correspondence from $M_{1}$ to $M_{3}$. As
an example, one can easily check that $\mathcal{I}(\rho_{2},\rho_{3})\otimes_{\rho_{2}(A)'}\mathcal{I}(\rho_{1},\rho_{2})$
is naturally isomorphic $\mathcal{I}(\rho_{1},\rho_{3})$, where the
$\rho_{i}$ are all $C^{*}$-representations of a $C^{*}$-algebra
$A$; the isomorphism sends $X\otimes Y$ to $XY$, where $XY$ is
just the ordinary product of $X$ and $Y$.


\p \label{par:InducedReps} As a special case of the composition
of $W^{*}$-correspondences, we find the notion of \emph{induced representation}
in the sense of Rieffel \cite{mR74a,mR74b}. If $F$ is a self-dual
Hilbert $C^{*}$-module over the $W^{*}$-algebra $M$, then \emph{inter
alia} $F$ is a $W^{*}$-correspondence from $\mathcal{L}(F)$ to
$M$. So if $\sigma$ is a normal representation of $M$ on a Hilbert
space $H$, then $\sigma$ makes $H$ a $W^{*}$-correspondence from
$M$ to $\mathbb{C}.$ Thus we get a correspondence $F\otimes_{M}H$
from $\mathcal{L}(F)$ to $\mathbb{C}$. If we want to think of $F\otimes_{M}H$
in terms of representations, then the normal representation of $\mathcal{L}(F)$
associated with $F\otimes_{M}\mathbb{C}$ is denoted $\sigma^{F}$
and is called the \emph{representation of $\mathcal{L}(F)$ induced}
by $\sigma$. Evidently, $\sigma^{F}$ is given by the formula $\sigma^{F}(T)(\xi\otimes h)=(T\xi)\otimes h$.
We will usually write $F\otimes_{M}H$ as $F\otimes_{\sigma}H$. It
is a consequence of \cite[Theorem 6.23]{mR74b} that the commutant
of $\sigma^{F}(\mathcal{L}(F))$ is $I_{F}\otimes\sigma(M)'.$


\p \label{par:SigmaDual} Putting together the structures we have
discussed so far, we come to a central concept to our theory. Returning
to our correspondence $E$ over $M$, let $\sigma$ be a normal representation
of $M$ on a Hilbert space $H$. Form the induced representation $\sigma^{E}$
of $\mathcal{L}(E)$ and form the normal representation of $M$, $\sigma^{E}\circ\varphi$,
which acts on $E\otimes_{\sigma}H$. (Recall that the left action
of $M$ on $E$ is given by the normal representation $\varphi:M\to\mathcal{L}(E)$.)
We define $E^{\sigma}$ to be $\mathcal{I}(\sigma,\sigma^{E}\circ\varphi$)
and we call $E^{\sigma}$ \emph{the $\sigma$-dual of} $E$. This
is a $W^{*}$-correspondence over $\sigma(M)'$. The bimodule actions
are given by the formula \[
a\cdot X\cdot b:=(I_{E}\otimes a)Xb,\qquad a,b\in\sigma(M)',\quad X\in\mathcal{I}(\sigma,\sigma\circ\varphi).\]


\p \label{par:TensorHardy}Along with $E$, we may form the ($W^{*}$-)tensor
powers of $E$, $E^{\otimes n}$. They will be understood to be the
self-dual completions of the $C^{*}$-tensor powers of $E$. Likewise,
the Fock space over $E$, $\mathcal{F}(E)$, will be the self-dual
completion of the Hilbert $C^{*}$-module direct sum of the $E^{\otimes n}$:
\[
\mathcal{F}(E)=M\oplus E\oplus E^{\otimes2}\oplus E^{\otimes3}\oplus\cdots\]
 We view $\mathcal{F}(E)$ as a $W^{*}$-correspondence over $M$,
where the left and right actions of $M$ are the obvious ones, i.e.,
the diagonal actions, and we shall write $\varphi_{\infty}$ for the
left diagonal action of $M$.

For $\xi\in E$, we shall write $T_{\xi}$ for the so-called \emph{creation
operator} on $\mathcal{F}(E)$ defined by the formula $T_{\xi}\eta=\xi\otimes\eta$,
$\eta\in\mathcal{F}(E)$. It is easy to see that $T_{\xi}$ is in
$\mathcal{L}(\mathcal{F}(E))$ with norm $\Vert\xi\Vert$, and that
$T_{\xi}^{*}$ annihilates $M$, as a summand of $\mathcal{F}(E)$,
while on elements of the form $\xi\otimes\eta$, $\xi\in E,$ $\eta\in\mathcal{F}(E)$,
it is given by the formula \[
T_{\xi}^{*}(\zeta\otimes\eta):=\varphi_{\infty}(\langle\xi,\zeta\rangle)\eta.\]

\begin{definition}\label{def:Tensor and Hardy Algs}If $E$ is a
$W^{*}$-correspondence over a $W^{*}$-algebra $M$, then \emph{the
tensor algebra} of $E$, denoted $\mathcal{T}_{+}(E)$, is defined
to be the norm-closed subalgebra of $\mathcal{L}(\mathcal{F}(E))$
generated by $\varphi_{\infty}(M)$ and $\{T_{\xi}\mid\xi\in E\}$.
The Hardy algebra of $E$, denoted $H^{\infty}(E)$, is defined to
be the ultra-weak closure in $\mathcal{L}(\mathcal{F}(E))$ of $\mathcal{T}_{+}(E)$.
\end{definition}


\p \label{par:CCReps} Plenty of examples are given in \cite{MSHardy}
and discussed in detail there. More will be given below, but we now
want to describe some properties of the representation theory of $\mathcal{T}_{+}(E)$
and $H^{\infty}(E)$ that we shall use. Details for what we describe
are presented in Section 2 of \cite{MSHardy}. If $\rho$ is a completely
contractive representation of $\mathcal{T}_{+}(E)$ on a Hilbert space
$H$, then $\sigma:=\rho\circ\varphi_{\infty}$ is a $C^{*}$-representation
of $M$ on $H$. We shall consider only those completely contractive
representations of $\mathcal{T}_{+}(E)$ with the property that $\rho\circ\varphi_{\infty}$
is an ultra-weakly continuous representation of $M$. This is not
a significant restriction. In particular, it is not a restriction
at all, if $H$ is assumed to be separable, since every $C^{*}$-representation
of a $\sigma$-finite $W^{*}$-algebra on a separable Hilbert space
is automatically ultra-weakly continuous \cite[Theorem V.5.1]{mT79}.

In addition to the representation $\sigma$ of $M$, $\rho$ defines
a bimodule map $T$ from $E$ to $B(H)$ by the formula\[
T(\xi):=\rho(T_{\xi}).\]
 To say that $T(\cdot)$ is a bimodule map means simply that $T(\varphi(a)\xi b)=\sigma(a)T(\xi)\sigma(b)$
for all $a,b\in M$ and for all $\xi\in E$. The assumption that $\rho$
is completely contractive guarantees that $T$ is completely contractive
with respect to the unique operator space structure on $E$ that arises
from viewing $E$ as a corner of its linking algebra. On the other
hand, the complete contractivity of $T$ is equivalent to the assertion
that the linear map $\widetilde{T}$ defined initially on the algebraic
tensor product $E\otimes H$ to $H$ by the formula \begin{equation}
\widetilde{T}(\xi\otimes h)=T(\xi)h\label{eq:Bimodule-contraction}\end{equation}
 has norm at most one and extends to a contraction, mapping $E\otimes_{\sigma}H$
to $H$, that satisfies the equation \eqref{eq:intertwining} by \cite[Lemma 2.16]{MS02}.
That lemma, coupled with \cite[Theorem 2.8]{MSHardy}, also guarantees,
conversely, that if $\widetilde{T}$ is a contraction from $E\otimes_{\sigma}H$
to $H$ satisfying equation \eqref{eq:intertwining}, then the equation
\eqref{eq:Bimodule-contraction} defines a completely contractive
bimodule map that together with $\sigma$ can be extended to a completely
contractive representation of $\mathcal{T}_{+}(E)$ on $H$. For these
reasons we call the pair $(T,\sigma)$ (or the pair $(\widetilde{T},\sigma)$)
a \emph{completely contractive covariant representation} of $(E,M)$
and we call the representation $\rho$ the \emph{integrated form}
of $(T,\sigma)$ and write $\rho=T\times\sigma$. From equation \eqref{eq:intertwining}
we see that $\widetilde{T}^{*}$ lies in the space we have denoted
$E^{\sigma}$. So, if we write $\mathbb{D}(E^{\sigma})$ for the open
unit ball in $E^{\sigma}$ and $\overline{\mathbb{D}(E^{\sigma})}$
for its norm closure, then all the completely contractive representations
$\rho$ of $\mathcal{T}_{+}(E)$ such that $\rho\circ\varphi_{\infty}=\sigma$
are parametrized bijectively by $\overline{\mathbb{D}(E^{\sigma*})}=\overline{\mathbb{D}(E^{\sigma})^{*}}=\overline{\mathbb{D}(E^{\sigma})}^{*}$.


\subp \label{subpar:RowCont} In the special case when $(E,M)$ is
$(\mathbb{C}^{d},\mathbb{C}),$ a representation $\sigma$ of $\mathbb{C}$
on a Hilbert space $H$ is quite simple; it does the only thing it
can: $\sigma(c)h=ch$, $h\in H$, and $c\in\mathbb{C}$. In this setting,
$E\otimes_{\sigma}H$ is just the direct sum of $d$ copies of $H$
and $\widetilde{T}$ is simply a $d$-tuple of operators $(T_{1},T_{2},\cdots,T_{d})$
such that $\Vert\sum_{i}T_{i}T_{i}^{*}\Vert\leq1$, i.e. $\widetilde{T}$
is a row contraction. The map $T$, then, is given by the formula
$T(\xi)=\sum\xi_{i}T_{i}$, where $\xi=(\xi_{1},\xi_{2},\cdots,\xi_{d})^{\top}\in\mathbb{C}^{d}$.
The space $E^{\sigma}$ is column space over $B(H)$, $\mathbf{C}_{d}(B(H))$
and $\mathbb{D}(E^{\sigma})$ is simply the unit ball in $\mathbf{C}_{d}(B(H))$.


\subp \label{subpar:TensorPowers} If $(T,\sigma)$ is a completely
contractive covariant representation of $(E,M)$ on a Hilbert space
$H$, then while it is not possible to form the powers of $\widetilde{T}$,
which maps $E\otimes_{\sigma}H$ to $H$, we can form the \emph{generalized}
powers of $\widetilde{T}$, which map $E^{\otimes n}\otimes H$ to
$H$, inductively as follows: Since $M\otimes_{\sigma}H$ is isomorphic
to $H$ via the map $a\otimes h\to\sigma(a)h$, $\widetilde{T_{0}}$
is just the identity map. Of course, $\widetilde{T_{1}}=\widetilde{T}$.
For $n>0$, $\widetilde{T_{n+1}}:=\widetilde{T_{1}}(I_{E}\otimes\widetilde{T_{n}})$,
mapping $E^{\otimes{n+1}}\otimes H$ to $H$. This sequence of maps
satisfies a semigroup-like property \begin{equation}
\widetilde{T_{n+m}}=\widetilde{T_{m}}(I_{E^{\otimes m}}\otimes\widetilde{T_{n}})=\widetilde{T_{n}}(I_{E^{\otimes n}}\otimes\widetilde{T_{m}}),\label{generalized powers}\end{equation}
 where we identify $E^{\otimes m}\otimes(E^{\otimes n}\otimes_{\sigma}H)$
and $E^{\otimes n}\otimes(E^{\otimes m}\otimes_{\sigma}H)$ with $E^{\otimes(n+m)}\otimes_{\sigma}H$
\cite[Section 2]{MSWold}. Since the maps $\widetilde{T_{n}}$ are
all contractions, they may be used to promote $(T,\sigma)$ to a completely
contractive covariant representation $(T_{n},\sigma)$ of $E^{\otimes n}$
on $H$, simply by setting \[
T_{n}(\xi_{1},\xi_{2},\cdots,\xi_{n})h:=T(\xi_{1})T(\xi_{2})\cdots T(\xi_{n})h=\widetilde{T_{n}}(\xi_{1}\otimes\xi_{2}\otimes\cdots\otimes h).\]


\subp \label{subpar:CPMaps} If $(T,\sigma)$ is a completely contractive
covariant representation of $(E,M)$ on a Hilbert space $H$, then
$(T,\sigma)$ induces a completely positive map $\Phi_{T}$ on $\sigma(M)'$
defined by the formula\begin{equation}
\Phi_{T}(a):=\widetilde{T}(I_{E}\otimes a)\widetilde{T}^{*}\qquad a\in\sigma(M)'.\label{eq:Phi_T}\end{equation}
 Indeed, $\Phi_{T}$ is clearly a completely positive map from $\sigma(M)'$
into $B(H)$, since $a\to I_{E}\otimes a$ is faithful normal representation
of $\sigma(M)'$ (onto the commutant of $\sigma^{E}(\mathcal{L}(E))$),
as we have noted earlier, and $\widetilde{T}$ is a bounded linear
map from $E\otimes_{\sigma}H$ to $H$. To see that its range is contained
in $\sigma(M)'$, simply note that $\widetilde{T}\sigma^{E}\circ\varphi=\sigma\widetilde{T}$.
So, if $a\in\sigma(M)'$ and if $b\in M$, then $\sigma(b)\Phi_{T}(a)=\sigma(b)\widetilde{T}(I_{E}\otimes a)\widetilde{T}^{*}=\widetilde{T}\sigma^{E}\circ\varphi(b)(I_{E}\otimes a)\widetilde{T}^{*}=\widetilde{T}(I_{E}\otimes a)\sigma^{E}\circ\varphi(b)\widetilde{T}^{*}=\widetilde{T}(I_{E}\otimes a)\widetilde{T}^{*}\sigma(b)=\Phi_{T}(a)\sigma(b)$,
which shows that the range of $\Phi_{T}$ is contained in $\sigma(M)'$.
Furthermore, we see that \begin{equation}
\Phi_{T}^{n}(a)=\widetilde{T_{n}}(I_{E^{\otimes n}}\otimes a)\widetilde{T_{n}}^{*}\label{eq:Phi_T_n}\end{equation}
 for all $n$. This is an immediate application of paragraph \ref{subpar:TensorPowers}
(see \cite[Theorem 3.9]{MS02} for details.)


\p \label{par:GaugeAutos} An important tool used in the analysis
of $\mathcal{T}_{+}(E)$ and $H^{\infty}(E)$ is the {}``spectral
theory of the gauge automorphism group''. What we need is developed
in detail in \cite[Section 2]{MSHardy}. We merely recall the essentials
that we will use. Let $P_{n}$ denote the projection of $\mathcal{F}(E)$
onto $E^{\otimes n}$. Then $P_{n}\in\mathcal{L}(\mathcal{F}(E))$
and the series\[
W_{t}:=\sum_{n=0}^{\infty}e^{int}P_{n}\]
 converges in the ultra-weak topology on $\mathcal{L}(\mathcal{F}(E))$.
The family $\{W_{t}\}_{t\in\mathbb{R}}$ is an ultra-weakly continuous,
$2\pi$-periodic unitary representation of $\mathbb{R}$ in $\mathcal{L}(\mathcal{F}(E))$.
Further, if $\{\gamma_{t}\}_{t\in\mathbb{R}}$ is defined by the formula
$\gamma_{t}=Ad(W_{t})$, then $\{\gamma_{t}\}_{t\in\mathbb{R}}$ is
an ultra-weakly continuous group of $*$-automorphisms of $\mathcal{L}(\mathcal{F}(E))$
that leaves invariant $\mathcal{T}_{+}(E)$ and $H^{\infty}(E)$.
Indeed, the subalgebra of $H^{\infty}(E)$ fixed by $\{\gamma_{t}\}_{t\in\mathbb{R}}$
is $\varphi_{\infty}(M)$ and $\gamma_{t}(T_{\xi})=e^{-it}T_{\xi}$,
$\xi\in E$. Associated with $\{\gamma_{t}\}_{t\in\mathbb{R}}$ we
have the {}``Fourier coefficient operators'' $\{\Phi_{j}\}_{j\in\mathbb{Z}}$
on $\mathcal{L}(\mathcal{F}(E))$, which are defined by the formula

\begin{equation}
\Phi_{j}(a):=\frac{1}{2\pi}\int_{0}^{2\pi}e^{-int}\gamma_{t}(a)\, dt,\qquad a\in\mathcal{L}(\mathcal{F}(E)),\label{FourierOperators}\end{equation}

where the integral converges in the ultra-weak topology. An alternate
formula for $\Phi_{j}$ is \[
\Phi_{j}(a)=\sum_{k\in\mathbb{Z}}P_{k+j}aP_{k}.\]
 Each $\Phi_{j}$ leaves $H^{\infty}(E)$ invariant and, in particular,
$\Phi_{j}(T_{\xi_{1}}T_{\xi_{2}}\cdots T_{\xi_{n}})=T_{\xi_{1}}T_{\xi_{2}}\cdots T_{\xi_{n}}$
if and only if $n=j$ and zero otherwise. Associated with the $\Phi_{j}$
are the {}``arithmetic mean operators'' $\{\Sigma_{k}\}_{k\geq1}$
that are defined by the formula\[
\Sigma_{k}(a):=\sum_{|j|<k}(1-\frac{|j|}{k})\Phi_{j}(a),\]
 $a\in\mathcal{L}(\mathcal{F}(E))$. For $a\in\mathcal{L}(\mathcal{F}(E))$,
$\lim_{k\to\infty}\Sigma_{k}(a)=a$, where the limit is taken in the
ultra-weak topology.


\p \label{par:isoreps} A completely contractive covariant representation
$(T,\sigma)$ of $(E,M)$ on a Hilbert space $H$ is called \emph{isometric}
(resp. \emph{fully coisometric}) in case $\widetilde{T}:E\otimes H\to H$
is an isometry (resp. a coisometry). It is not difficult to see that
a completely contractive covariant representation $(T,\sigma)$ is
isometric if and only if its integrated form $T\times\sigma$ is a
completely isometric representation. Further, this happens if and
only if $T\times\sigma$ is the restriction to $\mathcal{T}_{+}(E)$
of a $C^{*}$-representation of the $C^{*}$-subalgebra $\mathcal{T}(E)$
of $\mathcal{L}(\mathcal{F}(E))$ generated by $\mathcal{T}_{+}(E)$.
This $C^{*}$-algebra is called the \emph{Toeplitz algebra} of $E$.

A special kind of isometric covariant representations are constructed
as follows. Let $\pi_{0}:M\to B(H_{0})$ be a normal representation
of $M$ on the Hilbert space $H_{0}$, and let $H=\mathcal{F}(E)\otimes_{\pi_{0}}H_{0}$.
Set $\sigma=\pi^{\mathcal{F}(E)}\otimes I_{H_{0}}$, and define $S:E\to B(H)$
by the formula $S(\xi)=T_{\xi}\otimes I_{H_{0}}$, $\xi\in E$. Then
it is immediate that $(S,\sigma)$ is an isometric covariant representation
and we say that it is \emph{induced by $\pi_{0}$}. We also will say
$S\times\sigma$ is induced by $\pi_{0}$. In a sense that will become
clear, the representation $\pi_{0}$ should be viewed as a generalization
of the multiplicity of a shift.

An induced isometric covariant representation has the property that
$\widetilde{S_{n}}\widetilde{S_{n}^{*}}\to0$ strongly as $n\to\infty$
because $\widetilde{S_{n}}\widetilde{S_{n}^{*}}$ is the projection
onto $\sum_{k\geq n}E^{\otimes k}\otimes_{\pi_{0}}H_{0}$. In general,
an isometric covariant representation $(S,\sigma)$ and its integrated
form are called \emph{pure} if $\widetilde{S_{n}}\widetilde{S_{n}^{*}}\to0$
strongly as $n\to\infty$.

It is clear that induced covariant representations are analogues of
shifts. Corollary 2.10 of \cite{MSWold} shows that every pure isometric
covariant representation of $(E,M)$ is unitarily equivalent to an
isometric covariant representation that is induced by a normal representation
of $M$. We therefore will usually say simply that a pure isometric
covariant representation \emph{is} induced. In Theorem 2.9 of \cite{MSWold}
we proved a generalization of the Wold decomposition theorem that
asserts that every isometric covariant representation of $(E,M)$
decomposes as the direct sum of an induced isometric covariant representation
of $(E,M)$ and an isometric representation if $(E,M)$ that is both
isometric and fully coisometric.


\p \label{par:UnivIIRs} We will need an analogue of a unilateral
shift of infinite multiplicity.

\begin{lemma}\label{Lemma:Unitary Equiv.}For $i=1,2$, let $\pi_{i}:M\to B(H_{i})$
be normal representation of $M$ on the Hilbert space $H_{i}$, and
let $(S_{i},\sigma_{i})$ be the isometric covariant representation
of $(E,M)$ induced by $\pi_{i}$.
\begin{enumerate}
\item If $X:H_{1}\to H_{2}$ intertwines $\pi_{1}$ and $\pi_{2}$, then
$I_{\mathcal{F}(E)}\otimes X$ intertwines $(S_{1},\sigma_{1})$ and
$(S_{2},\sigma_{2})$. In particular, if $\pi_{1}$ and $\pi_{2}$
are unitarily equivalent, then so are $(S_{1},\sigma_{1})$ and $(S_{2},\sigma_{2})$.
\item If $(S_{1},\sigma_{1})$ and $(S_{2},\sigma_{2})$ are unitarily equivalent,
then so are $\pi_{1}$ and $\pi_{2}$.
\end{enumerate}
\end{lemma}

\begin{proof}The first assertion is a straightforward calculation.
The key point is that $I_{\mathcal{F}(E)}\otimes X$ represents a
bounded operator from $\mathcal{F}(E)\otimes_{\pi_{1}}H_{1}$ to $\mathcal{F}(E)\otimes_{\pi_{2}}H_{2}$
precisely because $X$ intertwines $\pi_{1}$ and $\pi_{2}$. The
second assertion may be seen as follows. Let $U$ be a Hilbert space
isomorphism from $\mathcal{F}(E)\otimes_{\pi_{1}}H_{1}$ to $\mathcal{F}(E)\otimes_{\pi_{2}}H_{2}$
such that $U(S_{1}\times\sigma_{1})=(S_{2}\times\sigma_{2})U$ and
let $H_{0}^{\infty}(E)=\{a\in H^{\infty}(E)\mid\Phi_{0}(a)=0\}$,
where $\Phi_{0}$ is the zero$^{th}$ Fourier coefficient operator
\eqref{FourierOperators}. Then $\overline{H_{0}^{\infty}(E)\mathcal{F}(E)}^{u-w}=\sum_{k\geq1}E^{\otimes k}$.
It follows that for $i=1,2$, $\overline{(S_{i}\times\sigma_{i})(H_{0}^{\infty}(E))(\mathcal{F}(E)\otimes_{\pi_{i}}H_{i})}=(\sum_{k\geq1}E^{\otimes k})\otimes_{\pi_{i}}H_{i}$
and, since $U(S_{1}\times\sigma_{1})=(S_{2}\times\sigma_{2})U$, it
follows also that $U$ carries $(\sum_{k\geq1}E^{\otimes k})\otimes_{\pi_{1}}H_{1}$
onto $(\sum_{k\geq1}E^{\otimes k})\otimes_{\pi_{2}}H_{2}$. Consequently,
$U$ restricts to a Hilbert space isomorphism from $M\otimes_{\pi_{1}}H_{1}\simeq H_{1}$
onto $M\otimes_{\pi_{2}}H_{2}\simeq H_{2}$ which, because $U(S_{1}\times\sigma_{1})=(S_{2}\times\sigma_{2})U$,
must satisfy $U\pi_{1}=\pi_{2}U$. \end{proof}

We shall fix, once and for all, a representation $(S_{0},\sigma_{0})$
that is induced by a faithful normal representation $\pi$ of $M$
that has \emph{infinite multiplicity} . That is, $(S_{0},\sigma_{0})$
acts on a Hilbert space of the form $\cF(E)\otimes_{\pi}K_{0}$, where
$\pi:M\to B(K_{0})$ is an infinite ampliation of a faithful normal
representation of $M$. Then $\sigma_{0}:=\pi^{\cF(E)}\circ\varphi_{\infty}$,
while $S_{0}(\xi):=T_{\xi}\otimes I_{K_{0}}$, $\xi\in E$. The following
proposition will be used in a number of arguments below.

\begin{proposition}\label{Lemma:Universal}The representation $(S_{0},\sigma_{0})$
is unique up to unitary equivalence and every induced isometric covariant
representation of $(E,M)$ is unitarily equivalent (in a natural way)
to a restriction of $(S_{0},\sigma_{0})$ to a subspace of the form
$\cF(E)\otimes_{\pi}\mathfrak{K}$, where $\mathfrak{K}$ is a subspace
of $K_{0}$ that reduces $\pi$.\end{proposition}

\begin{proof}The uniqueness assertion is an immediate consequence
of Lemma \ref{Lemma:Unitary Equiv.} and the structure of isomorphisms
between von Neumann algebras \cite[Theorem I.4.4.3]{DvN}. If $(R,\rho)$
is an induced representation acting, say, on the Hilbert space $H,$
then there is a normal representation $\rho_{0}$ of $M$ on a Hilbert
space $H_{0}$ and a Hilbert space isomorphism $W:H\to\cF(E)\otimes_{\rho_{0}}H_{0}$
such that $WR(\xi)W^{-1}=T_{\xi}\otimes I_{H_{0}}$ and such that
$W\rho(a)W^{-1}=\varphi_{\infty}(a)\otimes I_{H_{0}}$, $\xi\in E$
and $a\in M$. That is, $W(R\times\rho)W^{-1}$ is the induced representation
$\rho_{0}^{\cF(E)}$ of $\cL(\cF(E))$ restricted to $\mathcal{T}_{+}(E)$
acting on $\cF(E)\otimes_{\rho_{0}}H_{0}$. Since $\pi$ is a faithful
normal representation of $M$ on $K_{0}$ with infinite multiplicity,
one knows that $\rho_{0}$ is unitarily equivalent to $\pi_{0}$ restricted
to a reducing subspace $\mathfrak{K}$ of $K_{0}$. It follows that
$\rho_{0}^{\cF(E)}$is unitarily equivalent $\pi_{0}^{\cF(E)}$, which,
in turn, is unitarily equivalent to the restriction of $\pi^{\cF(E)}$
to $\cF(E)\otimes_{\pi}\mathfrak{K}$ . Stringing the unitary equivalences
together completes the proof. \end{proof}

\begin{definition}\label{Definition:Universal}We shall refer to
$(S_{0},\sigma_{0})$ as \emph{the universal induced covariant representation}
of $(E,M)$. \end{definition}By Proposition \ref{Lemma:Universal},
$(S_{0},\sigma_{0})$ does not really depend on the choice of representation
$\pi$ used to define it. It will serve the purpose in our theory
that the unilateral shift of infinite multiplicity serves in the structure
theory of single operators on Hilbert space.


\p \label{par:IsoDilate} A key tool in our theory is the following
result that we proved as \cite[Theorem 2.8]{MSHardy}.

\begin{theorem} Let $(T,\sigma)$ be a completely contractive covariant
representation of $(E,M)$ on a Hilbert space $H$. Then there is
an isometric covariant representation $(V,\rho)$ of $(E,M)$ acting
on a Hilbert space $K$ containing $H$ such that if $P$ denotes
the projection of $K$ onto $H$, then
\begin{enumerate}
\item $P$ commutes with $\rho(M)$ and $\rho(a)P=\sigma(a)P$, $a\in M$,
and
\item for all $\eta\in E$, $V(\eta)^{*}$ leaves $H$ invariant and $PV(\eta)P=T(\eta)P$.
\end{enumerate}
The representation $(V,\rho)$ may be chosen so that the smallest
subspace of $K$ that contains $H$ is all of $K$. When this is done,
$(V,\rho)$ is unique up to unitary equivalence and is called \textbf{\emph{the
minimal isometric dilation}}\emph{ of} $(T,\rho)$. \end{theorem}

There is an explicit matrix form for $(V,\rho)$ similar to the classical
Sch\"{a}ffer matrix for the unitary dilation of a contraction operator.
We will need parts of it, so we present the essentials here. For more
details, see Theorem 3.3 and Corollary 5.21 in \cite{MS98b} and the
proof of Theorem 2.18 in \cite{MS02}, in particular. Let $\Delta=(I-\tilde{T}^{\ast}\tilde{T})^{1/2}$
and let $\mathcal{D}$ be its range. Then $\Delta$ is an operator
on $E\otimes_{\sigma}H$ and commutes with the representation $\sigma^{E}\circ\varphi$
of $M$, by equation (\ref{intertwine}). Write $\sigma_{1}$ for
the restriction of $\sigma^{E}\circ\varphi$ to $\mathcal{D}$. Form
$K=H\oplus\mathcal{F}(E)\otimes_{\sigma_{1}}\mathcal{D}$, where the
tensor product $\mathcal{F}(E)\otimes_{\sigma_{1}}\mathcal{D}$ is
balanced over $\sigma_{1}$. The representation $\rho$ is just $\sigma\oplus\sigma_{1}^{\mathcal{F}(E)}\circ\varphi_{\infty}$.
For $V$, we form $E\otimes_{\rho}K$ and define $\widetilde{V}:E\otimes_{\rho}K\rightarrow K$
matricially as \begin{equation}
\tilde{V}:=\left[\begin{array}{cccccc}
\tilde{T} & 0 & 0 & \cdots\\
\Delta & 0 & 0 &  & \ddots\\
0 & I & 0 & \ddots\\
0 & 0 & I & 0 & \ddots\\
\vdots & 0 & 0 & I & \ddots\\
 &  &  &  & \ddots & \ddots\end{array}\right],\label{Vtilde}\end{equation}
 where the identity operators in this matrix must be interpreted as
the operators that identify $E\otimes_{\sigma_{n+1}}(E^{\otimes n}\otimes_{\sigma_{1}}\mathcal{D})$
with $E^{\otimes(n+1)}\otimes_{\sigma_{1}}\mathcal{D}$, and where,
in turn, $\sigma_{n+1}=\sigma_{1}^{E^{\otimes n}}\circ\varphi^{(n)}$.
(Here $\varphi^{(n)}$ denotes the representation of $M$ in $\mathcal{L}(E^{\otimes n})$
given by the formula $\varphi^{(n)}(a)(\xi_{1}\otimes\xi_{2}\cdots\otimes\xi_{n})=(\varphi(a)\xi_{1})\otimes\xi_{2}\otimes\cdots\xi_{n})$.)
Then it is easily checked that $\tilde{V}$ is an isometry and that
the associated covariant representation $(V,\rho)$ is the minimal
isometric dilation $(T,\sigma)$.

%
{}

\p\label{par:Commutant_Lifting} Another important tool in our analysis
is the \emph{commutant lifting theorem} \cite[Theorem 4.4]{MS98b}\emph{.
}The way we stated this result in \cite{MS98b}, it may be difficult
to recognize relations with the classical theorem. So we begin with
a more ``down to earth'' statement. It is exactly what is proved in
\cite{MS98b}.

\begin{theorem}\label{Theorem:Old_Fashion_CLT}Let $(T,\sigma)$
be a completely contractive representation of $(E,M)$ on a Hilbert
space $H$ and let $(V,\rho)$ be the minimal isometric dilation of
$(T,\sigma)$ acting on the space $K$ containing $H$. If $X$ is
an operator on $H$ that commutes with $(T\times\sigma)(\mathcal{T}_{+}(E))$
then there is an operator $Y$ on $K$ such that the following statements
are satisfied:
\begin{enumerate}
\item $\Vert Y\Vert=\Vert X\Vert$.
\item $Y$ commutes with $(V\times\rho)(\mathcal{T}_{+}(E))$.
\item $Y$ leaves $K\ominus H$ invariant.
\item $PY|_{H}=X$, where $P$ is the projection of $K$ onto $H$.
\end{enumerate}
\end{theorem}

Suppose, now, that for $i=1,2$ $(T_{i},\sigma_{i})$ is a completely
contractive covariant representation of $(E,M)$ on a Hilbert space
$H_{i}$. Then we shall write $\mathcal{I}((T_{1},\sigma_{1}),(T_{2},\sigma_{2}))$
for the set of operators $X:H_{1}\to H_{2}$ such that $XT_{1}(\xi)=T_{2}(\xi)X$
for all $\xi\in E$ and $X\sigma_{1}(a)=\sigma_{2}(a)X$ for all $a\in M$.
That is, $\mathcal{I}((T_{1},\sigma_{1}),(T_{2},\sigma_{2}))=\mathcal{I}(T_{1},T_{2})\cap\mathcal{I}(\sigma_{1},\sigma_{2})$.
Alternatively, $\mathcal{I}((T_{1},\sigma_{1}),(T_{2},\sigma_{2}))=\mathcal{I}(T_{1}\times\sigma_{1},T_{2}\times\sigma_{2})$.

\begin{theorem}\label{Theorem:ComLifting}For $i=1,2$, let $(T_{i},\sigma_{i})$
is a completely contractive covariant representation of $(E,M)$ on
a Hilbert space $H_{i}$, let $(V_{i},\rho_{i})$ be the minimal isometric
dilation of $(T_{i},\sigma_{i})$ acting on the space $K_{i}$, and
let $P_{i}$ be the orthogonal projection of $K_{i}$ onto $H_{i}$.
Further, let $\mathcal{I}((V_{1},\rho_{1}),(V_{2},\rho_{2});\, H_{1},H_{2})=\{X\in\mathcal{I}((V_{1},\rho_{1}),(V_{2},\rho_{2}))\mid XH_{1}^{\perp}\subseteq H_{2}^{\perp}\}$.
Then the map $\Psi$ from $B(K_{1},K_{2})$ to $B(H_{1},H_{2})$ defined
by the formula $\Psi(X)=P_{2}XP_{1}$ maps $\mathcal{I}((V_{1},\rho_{1}),(V_{2},\rho_{2});\, H_{1},H_{2})$
onto $\mathcal{I}((T_{1},\sigma_{1}),(T_{2},\sigma_{2}))$ and for
every $Y\in\mathcal{I}((T_{1},\sigma_{1}),(T_{2},\sigma_{2}))$, there
is an $X\in\mathcal{I}((V_{1},\rho_{1}),(V_{2},\rho_{2});\, H_{1},H_{2})$
with $\Vert X\Vert=\Vert Y\Vert$. \end{theorem}Otherwise stated,
the restriction of $\Psi$ to $\mathcal{I}((V_{1},\rho_{1}),(V_{2},\rho_{2});\, H_{1},H_{2})$
is a complete quotient map from $\mathcal{I}((V_{1},\rho_{1}),(V_{2},\rho_{2});\, H_{1},H_{2})$
onto $\mathcal{I}((T_{1},\sigma_{1}),(T_{2},\sigma_{2}))$. Theorem
\ref{Theorem:ComLifting} is not the statement of Theorem 4.4 of \cite{MS98b},
but it is a consequence of \cite[Theorem 4.4]{MS98b} and the following
standard matrix trick: Write the sum $(T_{1},\sigma_{1})\oplus(T_{2},\sigma_{2})$
acting on $H_{1}\oplus H_{2}$ matricially as $\left[\begin{array}{cc}
(T_{1},\sigma_{1}) & 0\\
0 & (T_{2},\sigma_{2})\end{array}\right]$. Then the minimal isometric dilation of $(T_{1},\sigma_{1})\oplus(T_{2},\sigma_{2})$
is $(V_{1},\rho_{1})\oplus(V_{2},\rho_{2})$. Further, from the matricial
perspective it is clear that the commutant of $\left[\begin{array}{cc}
(T_{1},\sigma_{1}) & 0\\
0 & (T_{2},\sigma_{2})\end{array}\right]$ is\[
\left[\begin{array}{cc}
(T_{1},\sigma_{1})' & \mathcal{I}((T_{2},\sigma_{2}),(T_{1},\sigma_{1}))\\
\mathcal{I}((T_{1},\sigma_{1}),(T_{2},\sigma_{2})) & (T_{2},\sigma_{2})'\end{array}\right]\]
 and similarly for the commutant of $\left[\begin{array}{cc}
(V_{1},\rho_{1}) & 0\\
0 & (V_{2},\rho_{2})\end{array}\right]$. With these observations, it is easy to see how Theorem \ref{Theorem:ComLifting}
follows from \cite[Theorem 4.4]{MS98b}.


\p \label{par:TopPoint} The following lemma will be used at several
points in the text. \begin{lemma}\label{cont} Suppose $\cA$ and
$\cB$ are ultra-weakly closed linear subspaces of $B(H)$ and $\psi:\cA\rightarrow\cB$
is a bounded linear map whose restriction to the unit ball of $\mathcal{A}$
is continuous with respect to the weak operator topologies on $\mathcal{A}$
and $\mathcal{B}$. Then $\psi$ is continuous with respect to the
ultra-weak topologies on $\mathcal{A}$ and $\mathcal{B}$. \end{lemma}
\begin{proof} Let $f$ be in $\cB_{*}$ and consider the linear functional
$f\circ\psi$ (on $\cA$). Since the weak operator topology and the
ultra-weak topology agree on any bounded ball of $\cB$, they agree
on $\psi(\cA_{1})$, where $\cA_{1}$ is the unit ball of $\cA$.
Consequently, $f\circ\psi|\cA_{1}$ is ultra-weakly continuous. It
now follows from \cite[Theorem I.3.1, (ii4) $\Rightarrow$ (ii1)]{DvN}
that $f\circ\psi$ is ultra-weakly continuous. Since $f$ was an arbitrary
ultra-weakly continuous functional, $\psi$ is continuous with respect
to the ultra-weak topologies.\end{proof}

\section{Absolute Continuity and Isometric Representations\label{sec:Absolute-Continuity}}

Throughout this paper we shall use the following standard notation.
If $\xi$ and $\eta$ are vectors in a Hilbert space $H$, then $\omega_{\xi,\eta}$
will denote the functional on $B(H)$ defined by the formula $\omega_{\xi,\eta}(X):=\langle X\xi,\eta\rangle$,
$X\in B(H)$. If $\xi=\eta$, then we shall simply write $\omega_{\xi}$
for $\omega_{\xi,\eta}$.

\begin{definition}\label{abscont} Let $(T,\sigma)$ be a completely
contractive covariant representation of $(E,M)$ on a Hilbert space
$H$.
\begin{enumerate}
\item [(1)] A vector $x\in H$ is said to be \emph{absolutely continuous}
in case the functional $\omega_{x}\circ(T\times\sigma)$ on $\mathcal{T}_{+}(E)$
can be written as $\omega_{\xi,\eta}\circ(S_{0}\times\sigma_{0})$
for suitable vectors $\xi$ and $\eta$ in the Hilbert space $\cF(E)\otimes_{\pi}K_{0}$
of the universal representation $(S_{0},\sigma_{0})$. That is, \[
\langle(T\times\sigma)(a)x,x\rangle=\langle(S_{0}\times\sigma_{0})(a)\xi,\eta\rangle=\langle(a\otimes I)\xi,\eta\rangle,\qquad a\in\mathcal{T}_{+}(E).\]

\item [(2)] We write $\mathcal{V}_{ac}$ (or $\mathcal{V}_{ac}(T,\sigma)$)
for the set of all absolutely continuous vectors in $H$.
\item [(3)] We say that $(T,\sigma)$ and $T\times\sigma$ are \emph{absolutely
continuous} in case $\mathcal{V}_{ac}=H$.
\end{enumerate}
\end{definition}

We will show eventually that $\mathcal{V}_{ac}(T,\sigma)$ is a closed
linear subspace of $H$, as one might expect. However, this will take
a certain amount of preparation and development. At the moment, all
we can say is that $\mathcal{V}_{ac}(T,\sigma)$ is closed under scalar
multiplication.

One of our principal goals is to show that a completely contractive
representation $T\times\sigma$ of $\mathcal{T}_{+}(E)$ extends to
an ultra-weakly continuous completely contractive representation of
$H^{\infty}(E)$ if and only if $(T,\sigma)$ is absolutely continuous.
The following remark suggests that one should think of being absolutely
continuous as being a {}``local'' phenomenon.

\begin{remark}\label{ac} A vector $x\in H$ is absolutely continuous
if and only if the functional $\omega_{x}\circ(T\times\sigma)$ extends
to an ultra-weakly continuous linear functional on $H^{\infty}(E)$.
To see this, first observe that any faithful normal representation
of $M$ induces a faithful representation of $\cL(\cF(E))$ that is
also a homeomorphism with respect the ultra-weak topologies. Consequently,
$S_{0}\times\sigma_{0}$ extends to a faithful representation of $H^{\infty}(E)$
that is a homeomorphism between the ultra-weak topology on $H^{\infty}(E)$
and the ultra-weak topology on $B(\cF(E)\otimes_{\pi}K_{0})$ restricted
to the range of $S_{0}\times\sigma_{0}$. Further, since $\pi$ is
assumed to have infinite multiplicity, every ultra-weakly continuous
linear functional on $(S_{0}\times\sigma_{0})(H^{\infty}(E))$ is
a vector functional, that is, it is of the form $\omega_{\xi,\eta}$
for vectors $\xi$ and $\eta$ in $\cF(E)\otimes_{\pi}K_{0}$. Consequently,
if $\omega_{x}\circ(T\times\sigma)$ extends to an ultra-weakly continuous
linear functional on $H^{\infty}(E)$, it must be representable in
the form $\omega_{\xi,\eta}\circ(S_{0}\times\sigma_{0})$, for suitable
vectors $\xi$ and $\eta$ in $\cF(E)\otimes_{\pi}K_{0}$. Thus, $x$
is an absolutely continuous vector. On the other hand, if $x$ is
an absolutely continuous vector, then $\omega_{x}\circ(T\times\sigma)$
can be written as $\omega_{\xi,\eta}\circ(S_{0}\times\sigma_{0})$,
by definition, and so extends to an ultra-weakly continuous linear
functional on $H^{\infty}(E)$. \end{remark}

To show that a completely contractive representation of $\mathcal{T}_{+}(E)$
is absolutely continuous if and only if it extends to an ultra-weakly
continuous representation of $H^{\infty}(E)$, we begin by proving
this for completely isometric representations. This fact and the technology
used to prove it will be employed in the next subsection to obtain
the general result.

\begin{definition}\label{wandering} Let $(S,\sigma)$ be an isometric
representation of $(E,M)$ on $H$. A vector $x\in H$ is said to
be a \emph{wandering vector} for $(S,\sigma)$, or for $S\times\sigma$,
if for every $k$ and every $\xi\in E^{\otimes k}$, the spaces $\sigma(M)x$
and $S(\xi)\sigma(M)x$ are orthogonal spaces. \end{definition}

\begin{lemma}\label{wand} Let $(S,\sigma)$ be an isometric representation
of $(E,M)$ on the Hilbert space $H$. If $x\in H$ is a wandering
vector for $(S,\sigma)$, then the representation $S\times\sigma$,
restricted to the closed invariant subspace of $H$ generated by $x$
is pure and, therefore, an induced representation. \end{lemma} \begin{proof}
Suppose $x$ is a wandering vector. Then the closed invariant subspace
$\mathcal{M}$ generated by $x$ is the orthogonal sum $\sum_{k=0}^{\infty}\oplus[S(E^{\otimes k})\sigma(M)x]$.
If we write $H_{0}$ for $[\sigma(M)x]$ and if $\pi_{0}$ denotes
the restriction of $\sigma$ to $H_{0}$, then it is easy to check
that $S\times\sigma$, restricted to $\mathcal{M}$, is unitarily
equivalent to the representation induced by $\pi_{0}$. \end{proof}

The proof of the next proposition was inspired by the proof of \cite[Theorem 1.6]{DKP}.

\begin{proposition}\label{X} Let $(S,\sigma)$ be an isometric representation
on the Hilbert space $H$ and suppose $x\in\mathcal{V}_{ac}(S,\sigma)$.
Then
\begin{enumerate}
\item [(i)] $x$ lies in the closure of the subspace of $H\oplus(\mathcal{F}(E)\otimes_{\pi}K_{0})$
generated by the wandering vectors for the representation $(S\times\sigma)\oplus(S_{0}\times\sigma_{0})$.
\item [(ii)] There is an $X\in\mathcal{I}((S_{0},\sigma_{0}),(S,\sigma))$
such that $x$ lies in the range of $X$.
\end{enumerate}
\end{proposition} \begin{proof} Since $x\in H$ is absolutely continuous
for $(S,\sigma)$, we can find vectors $\zeta,\eta$ in $\mathcal{F}(E)\otimes_{\pi}K_{0}$
such that $\langle(S\times\sigma)(a)x,x\rangle=\langle(S_{0}\times\sigma_{0})(a)\zeta,\eta\rangle$
for $a\in\mathcal{T}_{+}(E)$. Write $\tau$ for the representation
$(S\times\sigma)\oplus(S_{0}\times\sigma_{0})$, and then note that
\begin{equation}
\langle\tau(a)(x\oplus\zeta),(x\oplus(-\eta))\rangle=\langle(S\times\sigma)(a)x,x\rangle-\langle(S_{0}\times\sigma_{0})(a)\zeta,\eta\rangle=0\label{orth}\end{equation}
 for all $a\in\mathcal{T}_{+}(E)$. Let $\mathcal{N}$ be the closed,
$\tau(\mathcal{T}_{+}(E))$-invariant subspace generated by $x\oplus\zeta$,
$[\tau(\mathcal{T}_{+}(E))(x\oplus\zeta)]$, and write $(R,\rho)$
for the isometric representation associated with the restriction of
$\tau$ to $\mathcal{N}$. Thus $R(\xi)=(S(\xi)\oplus S_{0}(\xi))|\mathcal{N}=(S(\xi)\oplus(T_{\xi}\otimes I_{K_{0}}))|\mathcal{N}$
and $\rho=(\sigma\oplus\sigma_{0})|\mathcal{N}=(\sigma\oplus(\varphi_{\infty}\otimes I_{K_{0}}))|\mathcal{N}$.

We shall use Corollary 2.10 of \cite{MSWold} to show that $(R,\rho)$
is an induced representation. For that we must show that \[
\bigcap_{k=0}^{\infty}\overline{R(E^{\otimes k})(\mathcal{N})}=\{0\}.\]
 But, since $R(\xi_{k})=S(\xi_{k})\oplus(T_{\xi_{k}}\otimes I_{K})$,
it is immediate that \begin{equation}
\bigcap_{k=0}^{\infty}\overline{R(E^{\otimes k})(\mathcal{N})}\subseteq H\oplus\{0\}.\label{h}\end{equation}
 Since $(R,\rho)$ is an isometric representation, we can apply the
Wold decomposition theorem (\cite[Theorem 2.9]{MSWold}) and write
$(R,\rho)=(R_{1},\rho_{1})\oplus(R_{2},\rho_{2})$ where $(R_{1},\rho_{1})$
is an induced isometric representation and where $(R,\rho_{2})$ is
fully coisometric. This decomposition enables us, then, to write $x\oplus\zeta=\lambda+\mu$,
where $\lambda$ and $\mu$ are the projections of $x\oplus\zeta$
into the induced and the coisometric subspaces, respectively, for
$(R,\rho)$ . (We want to emphasize that we are not claiming $\lambda$
lies in $H$ and $\mu$ lies in $\cF(E)\otimes_{\pi}K_{0}$; i.e.,
the Wold direct sum decomposition may be different from the direct
sum decomposition $H\oplus(\cF(E)\otimes_{\pi}K_{0}).$) It follows
from equation (\ref{h}), however, that $\mu$ is of the form $h\oplus0$
for some $h\in H$ and, thus, $\lambda=(x-h)\oplus\zeta$. Since $\mu$
is orthogonal to $\lambda$, $h$ is orthogonal to $x-h$. Also, it
follows from equation (\ref{orth}) that $h\oplus0$ is orthogonal
to $x\oplus(-\eta)$. Thus $h$ is orthogonal to $x$. Since it is
also orthogonal to $x-h$, $h=0$, implying that $\mu=0$ and, thus,
that the representation $(R,\rho)$ is induced. But that implies that
every vector in $\mathcal{N}$ is in the closure of the subspace spanned
by the wandering vectors of $\tau$. In particular, $x\oplus\zeta$
lies there. Now note that we could have replaced $\zeta$ and $\eta$
by $t\zeta$ and $t^{-1}\eta$ for any $t>0$. We would then find
that $x\oplus t\zeta$ lies in the span of the wandering vectors for
every $t>0$. Letting $t\rightarrow0$, we find that $x$ lies in
that span, completing the proof of (i).

To prove (ii), first let $X_{0}$ be the projection of $H\oplus(\mathcal{F}(E)\otimes_{\pi}K_{0})$
onto $H$, restricted to $\mathcal{N}$. Then $x=X_{0}(x\oplus\zeta)$.
By construction, $X_{0}$ intertwines $R\times\rho$ and $S\times\sigma$.
However, by Lemma \ref{Lemma:Universal} and the fact that $(R,\rho)$
is induced, we find that $R\times\rho$ is unitarily equivalent to
a summand of $S_{0}\times\sigma_{0}$. Taking the equivalence into
account, $X_{0}$ can be exchanged for an $X$ that intertwines $S_{0}\times\sigma_{0}$
and $S\times\sigma$, and has $x$ in its range. \end{proof}

\begin{lemma}\label{intertwine} Suppose $(S,\sigma)$ is an isometric
representations of $(E,M)$ on a Hilbert space $H$ and suppose that
$X:\cF(E)\otimes_{\pi}K_{0}\rightarrow H$ is an element of the intertwining
space $\mathcal{I}((S_{0}\times\sigma_{0}),(S\times\sigma))$. If
$\mathcal{X}$ denotes the closure of the range of $X$, then $\mathcal{X}$
is invariant under $(S,\sigma)$ and the restriction $(R,\rho)$ of
$(S,\sigma)$ to $\mathcal{X}$ is an isometric representation. Also,
$R\times\rho$ admits a unique extension to a representation of $H^{\infty}(E)$
on $\mathcal{X}$ that is ultra-weakly continuous and completely isometric.
Consequently, \[
\mathcal{X}=\mathcal{V}_{ac}(R,\rho)\subseteq\mathcal{V}_{ac}(S,\sigma).\]
 \end{lemma} \begin{proof} The fact that $R\times\rho$ admits such
an extension is proved in \cite[Lemma 7.12]{MSHardy}. It then follows
that for $x\in\mathcal{X}$, the functional $\omega_{x}\circ(R\times\rho)=\omega_{x}\circ(S\times\sigma)$
extends to an ultra-weakly continuous functional on $H^{\infty}(E)$.
By Remark~\ref{ac}, this proves the last statement of the lemma.
\end{proof}

\begin{theorem}\label{subspace}If $(S,\sigma)$ is an isometric
covariant representation of $(E,M)$, then\[
\mathcal{V}_{ac}(S,\sigma)=\bigvee\{Ran(X)\mid X\in\mathcal{I}((S_{0},\sigma_{0}),(S,\sigma))\},\]
 and so in particular $\mathcal{V}_{ac}(S,\sigma)$ is a closed, $\sigma(M)$-invariant
subspace of $H$. \end{theorem} \begin{proof} We already noted that
$\mathcal{V}_{ac}(S,\sigma)$ is closed under scalar multiplication.
To see that it is closed under addition, fix $x,y\in\mathcal{V}_{ac}(S,\sigma)$.
Then, by Proposition~\ref{X}(ii), there are operators $X,Y\in\mathcal{I}((S_{0},\sigma_{0}),(S,\sigma))$
such that $x=X(\xi)$ and $y=Y(\eta)$ for suitable vectors $\xi$
and $\eta$ in $\cF(E)\otimes_{\pi}K_{0}$. Since $\pi$ has infinite
multiplicity, we may assume that the initial spaces of $X$ and $Y$
are orthogonal and, in particular, that $\xi$ and $\eta$ are orthogonal.
It follows, then, that if we set $Z:=X\oplus Y$, then $Z\in\mathcal{I}((S_{0},\sigma_{0}),(S,\sigma))$,
and $Z(\xi+\eta)=x+y$. Lemma~\ref{intertwine} implies that $x+y\in\mathcal{V}_{ac}(S,\sigma)$.
But also, $\overline{Ran(X)}\subseteq\mathcal{V}_{ac}(S,\sigma)$
for every $X\in\mathcal{I}((S_{0},\sigma_{0}),(S,\sigma))$, by Lemma
\ref{intertwine}. Thus, it remains to show that $\mathcal{V}_{ac}(S,\sigma)$
is closed. To this end, suppose $\{x_{n}\}_{n\in\mathbb{N}}\subseteq\mathcal{V}_{ac}(S,\sigma)$
is a sequence that converges to $x$ in $H$. Then the ultra-weakly
continuous linear functionals $\omega_{x_{n}}\circ(S\times\sigma)$
converge in norm to $\omega_{x}\circ(S\times\sigma)$, since, in general
$\norm{\omega_{x}\circ(S\times\sigma)-\omega_{y}\circ(S\times\sigma)}\leq\norm{x-y}$.
But the ultra-weakly continuous linear functionals on $H^{\infty}(E)$
form a norm closed subset of the dual space of $H^{\infty}(E)$. Thus
$\omega_{x}\circ(S\times\sigma)$ is extends to an ultra-weakly continuous
functional on $H^{\infty}(E)$. By Remark~\ref{ac}, $x\in\mathcal{V}_{ac}(S,\sigma)$.
\end{proof}

\begin{corollary}\label{wander} If $(S,\sigma)$ is an isometric
representation of $(E,M)$ on $H$, then \[
\mathcal{V}_{ac}(S,\sigma)=H\cap\overline{span}\{\emph{ \mbox{the wandering vectors of} }\rho:=(S\times\sigma)\oplus(S_{0}\times\sigma_{0})\}\]
 \[
\supseteq\overline{span}\{\mbox{the wandering vectors of }S\times\sigma\}.\]
 \end{corollary}

\begin{proof} If $x$ is a wandering vector for $S\times\sigma$,
then the restriction of $(S,\sigma)$ to the smallest $S\times\sigma$-
invariant subspace $\mathcal{N}$ spanned by $x$ is an induced isometric
representation, by Lemma \ref{wand}. Let $X$ be the inclusion of
$\mathcal{N}$ into $H$. Then it follows from Lemma~\ref{intertwine}
that $x\in\mathcal{V}_{ac}(S,\sigma)$. Now suppose $x\oplus\zeta\in H\oplus(\mathcal{F}(E)\otimes_{\pi}K_{0})$
is a wandering vector for $\rho$. Then the same argument shows that
the functional $\omega_{x\oplus\zeta}\circ\rho$ on $\mathcal{T}_{+}(E)$
extends to an ultra-weakly continuous linear functional on $H^{\infty}(E)$.
The same applies to the functional $\omega_{\zeta}\circ(S_{0}\times\sigma_{0})$.
Thus $\omega_{x}\circ(S_{0}\times\sigma_{0})$ extends to an ultra-weakly
continuous functional on $H^{\infty}(E)$. By Proposition \ref{X}
and Corollary \ref{subspace} $x\in\mathcal{V}_{ac}(S,\sigma)$. \end{proof}

\begin{remark} In general, the closed linear span of the wandering
vectors of $S\times\sigma$ is a proper subspace of $\mathcal{V}_{ac}(S,\sigma)$.
Indeed, it can be zero and yet $\mathcal{V}_{ac}(S,\sigma)$ is the
whole space. Let $S$ be the unitary operator on $L^{2}$ of the upper
half of the unit circle with Lebesgue measure that is given by multiplication
by the independent variable. Then $S$ is an absolutely continuous
unitary operator, but it has no wandering vectors.\end{remark}

\begin{theorem}\label{abscontrep} For an isometric representation
$(S,\sigma)$ of $(E,M)$ on a Hilbert space $H$ the following assertions
are equivalent:
\begin{enumerate}
\item [(1)] $S\times\sigma$ admits an ultra-weakly continuous extension
to a completely isometric representation of $H^{\infty}(E)$ on $H$.
\item [(2)] $(S,\sigma)$ is absolutely continuous (i.e., $\mathcal{V}_{ac}(S,\sigma)=H$).
\item [(3)] $H$ is contained in the closed linear span of the wandering
vectors of $(S,\sigma)\oplus(S_{0},\sigma_{0})$.
\end{enumerate}
\end{theorem} \begin{proof} It is clear that (1) implies (2). The
equivalence of (2) and (3) follows from Corollary~\ref{wander}.
It is left to show that (2) implies (1). So assume (2) holds and for
every $X\in\mathcal{I}((S_{0},\sigma_{0})(S,\sigma))$, let $\overline{Ran(X)}$
be the closure of the range of $X$. It follows from the assumption
that $\mathcal{V}_{ac}(S,\sigma)=H$ and Proposition \ref{X} that
$H$ is spanned by the family of subspaces $\{\overline{Ran(X)}\mid X\in\mathcal{I}((S_{0},\sigma_{0})(S,\sigma))\}$
and, furthermore, the restriction of $S\times\sigma$ to each subspace
$\overline{Ran(X)}$ in this family extends to an ultra-weakly continuous,
completely isometric, representation of $H^{\infty}(E)$ that we shall
denote by $(S\times\sigma)_{X}$. We need to show that these {}``restriction
representations'' can be glued together to form an ultra-weakly continuous
completely contractive extension of $S\times\sigma$.

To this end, fix an operator $a\in H^{\infty}(E)$ and recall that
$\Sigma_{k}(a)$ denotes the $k^{th}$-arithmetic mean of the Taylor
series of $a$. The $\Sigma_{k}(a)$ all lie in $\mathcal{T}_{+}(E)$,
satisfy the inequality $\norm{\Sigma_{k}(a)}\leq\norm{a}$, and converge
to $a$ in the ultra-weak topology on $H^{\infty}(E)$. Since the
sequence $(S\times\sigma)(\Sigma_{k}(a))$ is uniformly bounded in
$B(H)$, it has an ultra-weak limit point in $B(H)$. Any two limit
points must agree on each of the spaces $\overline{Ran(X)}$ since
the restrictions $(S\times\sigma)(\Sigma_{k}(a))|\overline{Ran(X)}$
must converge to $(S\times\sigma)_{X}$. Since the spaces $\overline{Ran(X)}$
span $H$, we see that there is only one limit point $\Theta(a)$
of the sequence $(S\times\sigma)(\Sigma_{k}(a))$. Thus the sequence
$\{(S\times\sigma)(\Sigma_{k}(a))\}_{k\in\mathbb{N}}$ converges ultra-weakly
to $\Theta(a)$. Moreover, for $x\in\overline{Ran(X)}$, $\Theta(a)x=(S\times\sigma)_{X}(a)x$.
The same sort of reasoning shows that $\Theta$, so defined, is a
completely isometric representation of $H^{\infty}(E)$ on $H$ that
extends $S\times\sigma$. It remains to show that $\Theta$ is ultra-weakly
continuous. For this it suffices to show that if $\{a_{\alpha}\}_{\alpha\in A}$
is a \emph{bounded} net $H^{\infty}(E)$ converging ultra-weakly in
$H^{\infty}(E)$ to an element $a\in H^{\infty}(E)$, then $\{\Theta(a_{\alpha})\}_{\alpha\in A}$
converges ultra-weakly to $\Theta(a)$. Since $\Theta$ is continuous,
$\{\Theta(a_{\alpha})\}_{\alpha\in A}$ is a bounded net and so we
need only show that it converges \emph{weakly} to $\Theta(a)$. But
for any $x\in H$, we can find an $X\in\mathcal{I}((S_{0},\sigma_{0})(S,\sigma))$
so that $x\in Ran(X)$ by Lemma \ref{X}. We conclude, then, that
$\omega_{x}\circ\Theta(a_{\alpha})=\omega_{x}\circ(S\times\sigma)_{X}(a_{\alpha})\to\omega_{x}\circ(S\times\sigma)_{X}(a)=\omega_{x}\circ\Theta(a)$.
Thus $\Theta(a_{\alpha})\rightarrow\Theta(a)$ weakly. It follows
(Lemma~\ref{cont}) that $\Theta$ is $\sigma$-weakly continuous
on $H^{\infty}(E)$. This proves that (2) implies (1). \end{proof}

\section{Completely Contractive Representations and Completely Positive Maps\label{sec:CCReps and CPMaps}}

%
{}As we saw in paragraph \ref{subpar:CPMaps}, if $(T,\sigma)$ is a
completely contractive covariant representation of $(E,M)$ on a Hilbert
space $H$, then $(T,\sigma)$ induces a completely positive map $\Phi_{T}$
on $\sigma(M)'$ defined by the formula\begin{equation}
\Phi_{T}(a):=\widetilde{T}(I_{E}\otimes a)\widetilde{T}^{*}\qquad a\in\sigma(M)'.\end{equation}
 One of our goals is to show that the absolutely continuous subspace
$\mathcal{V}_{ac}(T,\sigma)$ can be described completely in terms
of $\Phi_{T}$. We therefore want to begin by showing that given a
contractive, normal, completely positive map $\Phi$ on a $W^{*}$-algebra
$M$ and a normal representation $\rho$ of $M$ on a Hilbert space
$H$, then there is a canonical way to view $\rho\circ\Phi$ as a
$\Phi_{T}$ for a certain $T$ attached to a completely contractive
covariant representation $(T,\sigma)$ of a natural correspondence
over $\rho(M)'$. We will then prove that $\mathcal{V}_{ac}(T,\sigma)$
is an artifact of $\Phi$. For this first step, it will be convenient
for later use to omit the assumption that our completely positive
maps are contractions, in the following theorem.

\begin{theorem}\label{Theorem:Identity Representation} Given a normal
completely positive map $\Phi$ on a $W^{*}$-algebra $M$ and a normal
$*$-representation $\rho$ of $M$ on a Hilbert space $H$, there
is a canonical triple $(E,\eta,\sigma)$, where $E$ is a $W^{*}$-correspondence
over the commutant of $\rho(M)$, $\rho(M)'$, $\sigma$ is a normal
$*$-representation of $\rho(M)'$, and where $\eta$ is an element
of $E^{\sigma}$, such that \begin{equation}
\rho(\Phi(a))={\eta}^{*}(I_{E}\otimes\rho(a)){\eta}\label{eq:Arveson-Stinespring}\end{equation}
 for all $a\in M$. The triple $(E,\eta,\sigma)$ is essentially unique
in the following sense: If $(E_{1},\eta_{1},\sigma_{1})$ is another
triple consisting of a $W^{*}$-correspondence $E_{1}$ over $\rho(M)'$,
a normal $*$-representation $\sigma_{1}$ of $\rho(M)'$ and an element
$\eta_{1}$ of ${E_{1}}^{\sigma_{1}}$ such that $\rho(\Phi(a))={\eta_{1}}^{*}(I_{E_{1}}\otimes\rho(a)){\eta_{1}}$
for all $a\in M$, then the kernel of $\eta_{1}^{*}$ is of the form
$\sigma^{E_{1}}(q_{1})E_{1}\otimes H$ for a projection $q_{1}\in\mathcal{L}(E_{1})$
and there is an adjointable, surjective, bi-module map $W:E_{1}\to E$
such that $\eta_{1}^{*}=\eta^{*}(W\otimes I)$ and such that $W^{*}W=I_{E_{1}}-q_{1}.$
Further, $\sigma_{1}$ differs from $\sigma$ by an automorphism of
$\rho(M)'$, i.e., $\sigma_{1}=\sigma\circ\alpha$ for a suitable
automorphism $\alpha$ of $\rho(M)'$. \end{theorem}

\begin{proof} We present an outline of the existence of $(E,\eta,\sigma)$
since parts of the argument will be useful later. The details may
be found in \cite{MS02}. The uniqueness is proved in \cite[Theorem 2.6]{SSp09}
and we omit those details here. First, recall Stinespring's dilation
theorem \cite{fS55} and Arveson's proof of it \cite{wAr69}. Form
the Stinespring space $M\otimes_{\rho\circ\Phi}H$, which is the completion
of the algebraic tensor product $M\odot H$ in the inner product derived
from the formula\[
\langle a\otimes h,b\otimes k\rangle:=\langle h,\rho\circ\Phi(a^{*}b)k\rangle,\]
 and view $M$ as acting on $M\otimes_{\rho\circ\Phi}H$ through the
Stinespring representation, $\pi$: $\pi(a)(b\otimes h)=ab\otimes h$.
Let $V$ be the map from $H$ to $M\otimes_{\rho\circ\Phi}H$ defined
by the formula $Vh=I\otimes h$. Then the equation \begin{eqnarray}
\langle\pi(a)Vh,Vk\rangle & = & \langle(a\otimes h),1\otimes k\rangle\nonumber \\
 & = & \langle h,\rho\circ\Phi(a^{*})k\rangle\nonumber \\
 & = & \langle\rho\circ\Phi(a)h,k\rangle,\label{eq:Stinespring}\end{eqnarray}
 which is valid for all $a\in M$ and $h,k\in M$, shows that $V$
is bounded, with norm $\Vert\Phi(I)\Vert^{\frac{1}{2}}$, and that
\begin{equation}
V^{*}\pi(a)V=\rho(\Phi(a))\label{eq:Stinespring 2}\end{equation}
 for all $a\in M$.

Then let $E$ be the intertwining space $\mathcal{I}(\rho,\pi)$,
i.e., $\mathcal{I}(\rho,\pi)=\{X\in B(H,M\otimes_{\rho\circ\Phi}H)\mid X\rho(a)=\pi(a)X,\,\mbox{for all }a\in M\}$.
As we noted in paragraph \ref{par:Intertwining}, this space is a
$W^{*}$-correspondence from the commutant of $\pi(M)$, $\pi(M)'$,
to the commutant of $\rho(M)$, $\rho(M)'$. However, the map $a\to I_{M}\otimes a$
is normal representation of $\rho(M)'$ into the commutant of $\pi(M)$,
and so by restriction, $E=\mathcal{I}(\rho,\pi)$ becomes a $W^{*}$-correspondence
over $\rho(M)'$. The bimodule structure is given by the formula\[
a\cdot X\cdot b=(I_{M}\otimes a)Xb,\]
 $a,b\in\rho(M)'$. We let $\sigma$ be the identity representation
of $\rho(M)'$ on $H$ .

To define $\eta\in E^{\sigma}$, we first observe that there is a
Hilbert space isomorphism $U:E\otimes_{\sigma}H\to M\otimes_{\rho\circ\Phi}H$
defined by the formula $U(X\otimes h):=Xh$. The fact that $U$ is
isometric is immediate from the way the $\rho(M)'$-valued inner product
on $E$ is defined. The fact that $U$ is surjective is Lemma 2.10
of \cite{MS02}. Further, a straightforward computation shows that
$U(I_{E}\otimes\rho(a))U^{-1}=\pi(a)$ for all $a\in M$. Indeed,
if $X\in E=\mathcal{I}(\rho,\pi)$ and if $h\in H$, then for $a\in M$,
\[
U(I_{E}\otimes\rho(a))(X\otimes h)=X\rho(a)h=\pi(a)Xh=\pi(a)U(X\otimes h).\]
 That is, $U(I_{E}\otimes\rho(\cdot))=\pi(\cdot)U$. Second, we note
that since $I_{M}\otimes\sigma(a)$ lies in $\pi(M)'$ for all $a\in\rho(M)'$,
a similar calculation shows that $U(\varphi(\cdot)\otimes I_{H})=(I_{M}\otimes\sigma(\cdot))U$.
Finally, observe from the definition of $V$ that $V\sigma(\cdot)=I_{M}\otimes\sigma(\cdot)V$.
Now set $\eta=U^{*}V$. Then $\eta\in E^{\sigma}$ since for all $a\in\rho(M)'$,
$\eta\sigma(a)=U^{*}V\sigma(a)=U^{*}(I_{M}\otimes\sigma(a))V=(\varphi(a)\otimes I_{H})U^{*}V=(\varphi(a)\otimes I_{H})\eta$.
Also, \[
\eta^{*}(I_{E}\otimes\rho(a))\eta=V^{*}U(I_{E}\otimes\rho(a))U^{*}V=V^{*}\pi(a)V=\rho\circ\Phi(a).\]
 \end{proof}

\begin{remark}\label{Rem:cb-norm}The $cb$-norm of any completely
positive map is the norm of its value at the identity. So $\Vert\Phi_{\eta}\Vert_{cb}=\Vert\eta^{*}\eta\Vert_{\sigma(M)'}=\Vert\eta\Vert_{E^{\sigma}}^{2}$.
Consequently, $\Phi_{\eta}$ is contractive and completely positive
if and only if $\eta\in\overline{\mathbb{D}(E^{\sigma})}$. We thus
see that every contractive completely positive map $\Phi$ on a $W^{*}$-algebra
can be realized in terms of a completely contractive covariant representation
of the natural $W^{*}$-correspondence $E$ we just constructed from
it. We call $E$ the \emph{Arveson-Stinespring correspondence} associated
to $\Phi$ (see \cite{MS02}). It depends on a choice of a representation
of $M$, but that will only be emphasized when necessary. The ultra-weakly
continuous, completely contractive covariant representation $(T,\sigma)$
of $(E,\rho(M)')$ such that $\rho\circ\Phi=\Phi_{\eta}$, where $\widetilde{T}=\eta^{*}$
is called the \emph{identity representation}. The advantage of the
identity representation of a completely positive map through equation
(\ref{eq:Arveson-Stinespring}) in Theorem \ref{Theorem:Identity Representation}
over the Stinespring representation, equation (\ref{eq:Stinespring 2}),
is that one can express the powers of $\Phi$ in terms of it as we
discussed in paragraph \ref{subpar:CPMaps}. In this setting, equation
\ref{eq:Phi_T_n} becomes \[
\rho(\Phi^{n}(a))=\widetilde{T_{n}}(I_{E^{\otimes n}}\otimes\rho(a))\widetilde{T_{n}}^{*}.\]
 \end{remark}

\begin{example}\label{Example:MarkovChain1}To illustrate these constructs
in a concrete example, let $M=\ell^{\infty}(\{1,2,\cdots,n\})$ and
let $\sigma$ represent $M$ on the Hilbert space $\mathbb{C}^{n}$
as diagonal matrices. Thus $\sigma(\underline{\varphi})=diag(\varphi_{1},\varphi_{2},\cdots,\varphi_{n})$.
Of course, $\sigma(M)$ is the masa $\mathbf{D}_{n}$ consisting of
all diagonal and so $\sigma(M)'=\mathbf{D}_{n}$, too. Also, let $A=(a_{ij})$
be an $n\times n$ sub-Markov matrix. This means that the $a_{ij}$
are all non-negative, and that for each $i$, $\sum_{j}a_{ij}\leq1$.
Such a matrix determines a completely positive, contractive map $\Phi$
on $\mathbf{D}_{n}$ through the formula\[
\Phi(\underline{d}):=diag(\sum_{j}a_{1j}d_{j},\sum_{j}a_{2j}d_{j},\cdots,\sum_{j}a_{nj}d_{j}),\]
 where $\underline{d}=(d_{1},d_{2},\cdots,d_{n})$. We let $\varepsilon_{i}$
be the diagonal matrix with zeros everywhere but in the $i^{th}$
row and column, where it is a one, and we let $\{e_{i}\}_{i=1}^{n}$
be the standard basis for $\mathbb{C}^{n}$. Then the vectors $\varepsilon_{i}\otimes e_{j}$,
$i,j=1,2,\cdots,n$, span $\sigma(M)'\otimes_{\Phi}\mathbb{C}^{n}$,
and an easy calculation shows that \begin{eqnarray*}
\langle\varepsilon_{i}\otimes e_{j},\varepsilon_{k}\otimes e_{l}\rangle & = & a_{ji}\end{eqnarray*}
 if and only $(i,j)=(k,l)$ and is zero otherwise. It follows that
in $\sigma(M)'\otimes_{\Phi}\mathbb{C}^{n}$, $\varepsilon_{i}\otimes e_{j}$
is nonzero if and only if $(j,i)$ lies in the support of $A$, which
we denote by $G^{1}$. (The reason for the super script is that we
are going to view $G^{1}$ as the edge set of a graph. The vertex
set, $G^{0}$, is $\{1,2,\cdots,n\}$.) The calculation just completed
shows that \[
\{a_{ji}^{-\frac{1}{2}}\varepsilon_{i}\otimes e_{j}\mid(j,i)\in G^{1}\}\]
 is an orthonormal basis for $\sigma(M)'\otimes_{\Phi}\mathbb{C}^{n}=\mathbf{D}_{n}\otimes_{\Phi}\mathbb{C}^{n}$.
We let $\lambda$ be the representation of $\sigma(M)'=\mathbf{D}_{n}$
on $\sigma(M)'\otimes_{\Phi}\mathbb{C}^{n}$ is given by the formula
$\lambda(\underline{d})(\varepsilon_{i}\otimes e_{j})=d_{i}\varepsilon_{i}\otimes e_{j}$.
We also let $\iota$ be the identity representation of $\sigma(M)'=\mathbf{D}_{n}$
on $\mathbb{C}^{n}$. Then the Arveson-Stinespring correspondence
in this case is $E=\mathcal{I}(\iota,\lambda)$. An operator $X$
from $\mathbb{C}^{n}$ to $\sigma(M)'\otimes_{\Phi}\mathbb{C}^{n}$
is given by a matrix that we shall write $[X((i,j),k)]_{(i,j)\in G^{1}\, k\in\{1,2,\cdots,n\}}$.
The formula for $X((i,j),k)$ is, of course,\[
X((i,j),k)=\langle Xe_{k},a_{ji}^{-\frac{1}{2}}\varepsilon_{i}\otimes e_{j}\rangle.\]
 Since an $X$ in $E$ intertwines the identity representation of
$\sigma(M)'=\mathbf{D}_{n}$ on $\mathbb{C}^{n}$ and $\lambda$,
it follows from this equation that $X((i,j),k)$ is zero unless $i=k$,
when $X\in E$. Thus $E$ may be viewed as a space of functions supported
on $G^{1}$. Now for $X$ and $Y$ in $E,$ $X^{^{*}}Y(i,j)=\sum_{(k,l)\in G^{1}}\overline{X((k,l),i)}Y((k,l),j)$.
Since $X((k,l),i)=0$, unless $k=i$ and since $Y((k,l),j)=0$, unless
$k=j$, we see that $X^{*}Y(i,j)=0,$ unless $i=j$, in which case
we find that $X^{^{*}}Y(i,i)=\sum_{l=1}^{n}\overline{X((i,l),i)}Y((i,l),i)$.
So, if $X_{(i,j)}$, $(i,j)\in G^{1}$, is defined by the formula
\[
X_{(i,j)}((k,l),m)=a_{ji}^{-\frac{1}{2}},\]
 when $k=m=i$ and $l=j$, and zero otherwise, then $\{X_{(i,j)}\}_{(i,j)\in G^{1}}$
is an orthonormal basis for $E$. It follows, then, that $\{X_{(i,j)}\otimes e_{i}\}_{(i,j)\in G^{1}}$
is an orthonormal basis for $E\otimes_{\sigma(M)}\mathbb{C}^{n}$
(owing to the fact that $X\otimes\underline{d}h=X\cdot\underline{d}\otimes h$
for all $X\otimes h\in E\otimes_{\sigma(M)}\mathbb{C}^{n}$ and for
all $\underline{d}\in\mathbf{D}_{n}$.) The map $U:E\otimes_{\sigma(M)}\mathbb{C}^{n}\to\sigma(M)'\otimes_{\Phi}\mathbb{C}^{n}=\mathbf{D}_{n}\otimes_{\Phi}\mathbb{C}^{n}$
is given by the formula $U(X\otimes h)=Xh$ and so, at the level of
coordinates, we find that $U(X\otimes h)(i,j)=X((i,j),i)h(i)$, where
$(j,i)$ lies in $G^{1}$. In particular, we see that $U(X_{(i,j)}\otimes e_{i})(i,j)=X((i,j),i)=a_{ji}^{-\frac{1}{2}}$
so that $U(X_{(i,j)}\otimes e_{i})=a_{ji}^{-\frac{1}{2}}(\varepsilon_{i}\otimes e_{j})$.
The map $V:\mathbb{C}^{n}\to\sigma(M)'\otimes_{\Phi}\mathbb{C}^{n}=\mathbf{D}_{n}\otimes_{\Phi}\mathbb{C}^{n}$
is defined by the formula $Vh=1\otimes h$, where in this case, $1$
denotes the identity matrix. Recapitulating an earlier calculation
we see that \begin{eqnarray*}
\langle V^{*}(\varepsilon_{i}\otimes e_{j}),e_{k}\rangle & = & \langle\varepsilon_{i}\otimes e_{j},Ve_{k}\rangle\\
 & = & \langle\varepsilon_{i}\otimes e_{j},1\otimes e_{k}\rangle\\
 & = & \langle e_{j},\Phi(\varepsilon_{i}^{*})e_{k}\rangle\\
 & = & \langle e_{j},a_{ki}e_{k}\rangle.\end{eqnarray*}
 With all the pieces calculated, we see that the map $\widetilde{T}:E\otimes_{\sigma(M)}\mathbb{C}^{n}\to\mathbb{C}^{n}$
is defined on basis vectors for $E\otimes_{\sigma(M)}\mathbb{C}^{n}$
by the equation \begin{eqnarray}
\widetilde{T}(X_{(i,j)}\otimes e_{i}) & = & V^{*}U(X_{(i,j)}\otimes e_{i})\nonumber \\
 & = & V^{*}(a_{ji}^{-\frac{1}{2}}\varepsilon_{i}\otimes e_{j})\nonumber \\
 & = & a_{ji}a_{ji}^{-\frac{1}{2}}e_{j}\nonumber \\
 & = & a_{ji}^{\frac{1}{2}}e_{j}.\label{eq:tildeT}\end{eqnarray}
 We will use these calculations in later examples. \end{example}

\begin{definition}\label{def:superharmonic}Let $\Phi$ be a completely
positive operator on a $W^{*}$-algebra $M$. An element $Q\in M$
is called a \emph{superharmonic operator} in case $Q\geq0$ and \begin{equation}
\Phi(Q)\leq Q.\label{eq:superharmonic}\end{equation}
 If, in addition, the sequence $\{\Phi^{n}(Q)\}_{n\in\mathbb{N}}$
converges to zero strongly, then we say that $Q$ is a \emph{pure}
superharmonic operator. A superharmonic operator $Q$ such that $\Phi(Q)=Q$
is called \emph{harmonic}. \end{definition}If $M$ is $L^{\infty}(X,\mu)$
for some probability space $(X,\mu)$, then a superharmonic operator
is a superharmonic function in the sense of Markov processes. (See
\cite[Definition 2.1.1]{dR84}.)

\begin{remark}There is an analogue of the Riesz decomposition theorem
for superharmonic functions, viz: If $Q$ is a superharmonic operator
for $\Phi$, then $Q$ decomposes uniquely as $Q=Q_{p}+Q_{h}$, where
$Q_{p}$ is a pure superharmonic operator for $\Phi$ and $Q_{h}$
is a harmonic operator for $\Phi$. Indeed, simply set $Q_{h}:=Q-\lim\Phi^{n}(Q)$
and $Q_{p}:=Q-Q_{h}$.\end{remark}

Our next goal is to describe all the pure superharmonic operators
for a given completely positive map, $\Phi$, say. We will assume
that we are given some $W^{*}$-correspondence $E$ over $M$, a normal
representation $\sigma:M\to B(H)$ and an element $\eta\in E^{\sigma}$
so that $\Phi$ is realized as $\Phi_{\eta}$ acting on $\sigma(M)'$
through the formula\begin{equation}
\Phi_{\eta}(a):=\eta^{*}(I_{E}\otimes a)\eta,\qquad a\in\sigma(M)',\label{eq:Phi_eta_represent}\end{equation}
as in Theorem \ref{Theorem:Identity Representation}. The data $(E,\eta,\sigma)$
need not be the data constructed in that result; it can be quite arbitrary.
However, since we are not assuming that $\Phi_{\eta}$ and $\eta$
have norm at most one, some additional preparation is necessary. Since
$\eta^{*}$ is a bounded linear map from $E\otimes_{\sigma}H$ to
$H$ that satisfies the equation $\eta^{*}\sigma^{E}\circ\varphi=\sigma\eta^{*}$,
\cite[Lemma 3.5]{MS98b} implies that if we define $\widehat{\eta^{*}}$
by the formula\[
\widehat{\eta^{*}}(\xi)h=\eta^{*}(\xi\otimes h),\qquad\xi\otimes h\in E\otimes_{\sigma}H,\]
 then $\widehat{\eta^{*}}$ is a completely bounded bimodule map with
$cb$-norm $\Vert\eta^{*}\Vert=\Vert\eta\Vert$. We shall refer to
the pair $(\widehat{\eta^{*}},\sigma)$ as a \emph{completely bounded
covariant representation} of $(E,M)$. Although $\widehat{\eta^{*}}$
need not extend to a completely bounded representation of $\mathcal{T}_{+}(E)$,
as would be the case if $\Vert\eta\Vert\leq1$, we still can promote
$\widehat{\eta^{*}}$ to a map $\widehat{\eta_{n}^{*}}$ on each of
the tensor powers of $E$, $E^{\otimes n}$, via the formula\[
\widehat{\eta_{n}^{*}}(\xi_{1}\otimes\xi_{2}\otimes\cdots\otimes\xi_{n}):=\widehat{\eta^{*}}(\xi_{1})\widehat{\eta^{*}}(\xi_{2})\cdots\widehat{\eta^{*}}(\xi_{n}).\]
 When this is done, the map $\widehat{\eta_{n}^{*}}$ is a bimodule
map from $E^{\otimes n}$ to $B(H)$ whose associated linear map $\widetilde{\widehat{\eta_{n}^{*}}}$
from $E^{\otimes n}\otimes_{\sigma}H$ to $H$ is given by formula\begin{multline*}
\widetilde{\widehat{\eta_{n}^{*}}}(\xi_{1}\otimes\xi_{2}\otimes\cdots\otimes\xi_{n}\otimes h)=\eta^{*}(\xi_{1}\otimes\eta^{*}(\xi_{2}\otimes(\cdots\otimes\eta*(\xi_{n}\otimes h)\cdots)\\
=\eta^{*}(I_{E}\otimes\eta^{*})(I_{E^{\otimes2}}\otimes\eta^{*})\cdots(I_{E^{\otimes(n-1)}}\otimes\eta^{*})(\xi_{1}\otimes\xi_{2}\otimes\cdots\otimes\xi_{n}\otimes h),\end{multline*}
i.e., $\widetilde{\widehat{\eta_{n}^{*}}}=\eta^{*}(I_{E}\otimes\eta^{*})(I_{E^{\otimes2}}\otimes\eta^{*})\cdots(I_{E^{\otimes(n-1)}}\otimes\eta^{*})$.
To lighten the notation, we drop the {}``hat'' and {}``tilde'',
and simply write $\eta_{n}^{*}=\eta^{*}(I_{E}\otimes\eta^{*})(I_{E^{\otimes2}}\otimes\eta^{*})\cdots(I_{E^{\otimes(n-1)}}\otimes\eta^{*})$.
This is entirely consistent with what we used for $\widetilde{T_{n}}$
in paragraph \ref{subpar:CPMaps}. Further, we may then also define
$\eta_{n}:=(\eta_{n}^{*})^{*}$, which yields \[
\eta_{n}=(I_{E^{\otimes(n-1)}}\otimes\eta)(I_{E^{\otimes(n-2)}}\otimes\eta)\cdots(I_{E}\otimes\eta)\eta\]
 as expected. We let $\eta_{0}$ be the map from $H$ to $M\otimes_{\sigma}H$
that identifies $H$ with $M\otimes_{\sigma}H$ in the customary fashion.
With this notation, we find that\begin{equation}
\Phi_{\eta}^{n}(a)=\eta_{n}^{*}(I_{E^{\otimes n}}\otimes a)\eta_{n}\label{eq:Phi_eta_powers}\end{equation}
 for all $a\in\sigma(M)'$ and all $n\geq0$.

\begin{theorem}\label{Theorem:PureSuperHarm} Let $(\widehat{\eta^{*}},\sigma)$
be a completely bounded covariant representation of $(E,M)$ on the
Hilbert space $H$ and let $\Phi_{\eta}$ be the completely positive
map on $\sigma(M)'$ defined by \eqref{eq:Phi_eta_represent}. An
operator $Q$ in $\sigma(M)'$ is a pure superharmonic operator for
$\Phi_{\eta}$ if and only if $Q=CC^{*}$ for an operator $C\in\mathcal{I}((S_{0},\sigma_{0}),(\widehat{\eta^{*}},\sigma))$.
In this event, if $r=(Q-\Phi(Q))^{\frac{1}{2}},$ then $(I_{\mathcal{F}(E)}\otimes r)C(\eta)$
is a bounded linear operator defined on all of $H$, mapping $H$
to $\mathcal{F}(E)\otimes_{\sigma}H$, and $C^{*}$ may be written
as $C^{*}=(I_{\mathcal{F}(E)}\otimes v)(I_{\mathcal{F}(E)}\otimes r)C(\eta)$,
where $v$ is any partial isometry in $\mathcal{I}(\sigma,\pi)$ whose
initial projection contains the range projection of $r$.\end{theorem}
\begin{proof}Suppose $Q\in\sigma(M)'$ has the form $Q=CC^{*},$
$C\in\mathcal{I}((S_{0},\sigma_{0}),(\widehat{\eta^{*}},\sigma))$.
Then by equation \eqref{eq:Phi_eta_powers} we may write\begin{align*}
\Phi_{\eta}^{n}(Q) & =\eta_{n}^{*}(I_{E^{\otimes n}}\otimes Q)\eta_{n}\\
= & \eta_{n}^{*}(I_{E^{\otimes n}}\otimes C)(I_{E^{\otimes n}}\otimes C^{*})\eta_{n}\\
= & C(\widetilde{S_{0}})_{n}((\widetilde{S_{0}})_{n})^{*}C^{*}\\
= & CP_{n}C^{*},\end{align*}
 where here we use $P_{n}$ to denote the projection onto $\sum_{k\geq n}E^{\otimes k}\otimes_{\pi}K_{0}$.
Since the $P_{n}$ decrease strongly to zero, the operators $\Phi_{\eta}^{n}(Q)$
decrease strongly to zero, as $n\to\infty$. Thus $Q$ is pure superharmonic
for $\Phi_{\eta}$.

For the converse, suppose $Q\in\sigma(M)'$ is a given pure superharmonic
operator for $\Phi_{\eta}$ and write $r^{2}:=Q-\Phi_{\eta}(Q)$.
The {}``purity'' of $Q$ guarantees that $Q=\sum_{n\geq0}\Phi_{\eta}^{n}(r^{2})$,
where the series converges in the strong operator topology. Indeed,
the $n^{th}$ partial sum of the series is $Q-\Phi_{\eta}^{n+1}(Q)$.
Let $\mathcal{R}$ be the closure of the range of $r$. Since $r\in\sigma(M)'$,
$\mathcal{R}$ reduces $\sigma$ and so we get a new normal representation,
$\sigma_{\mathcal{R}}$, of $M$ by restricting $\sigma(\cdot)$ to
$\mathcal{R}$. Choose an isometry $v$ from $\mathcal{R}$ into $K_{0}$
that is in $\mathcal{I}(\sigma_{\mathcal{R}},\pi)$. (Such a choice
is possible by the definition of $\pi$.) Define $C^{*}:H\to\mathcal{F}(E)\otimes_{\pi}K_{0}$
by the formula\begin{align*}
C^{*}x: & =(I_{\mathcal{F}(E)}\otimes v)\sum_{n\geq0}(I_{E^{\otimes n}}\otimes r)\eta_{n}x\\
= & (I_{\mathcal{F}(E)}\otimes vr)C(\eta)x.\end{align*}
 A straightforward calculation shows that this series converges and
that the sum defines a bounded operator $C^{*}$ that satisfies the
equation\begin{eqnarray*}
CC^{*}x & = & \sum_{n\geq0}\eta_{n}^{*}(I_{E^{\otimes n}}\otimes r^{2})\eta_{n}x\\
 & = & \sum_{n\geq0}\Phi_{\eta}^{n}(r^{2})x\\
 & = & Qx.\end{eqnarray*}
 It is also clear that $C\in\mathcal{I}((S_{0},\sigma_{0}),(\widehat{\eta^{*}},\sigma))$.\end{proof}

Theorem \ref{Theorem:PureSuperHarm} has its roots in work of Kato
\cite{tK68}. Indeed, he might have called the operator $r$ a \emph{smooth
operator} with respect to $\eta^{*}$(See \cite[p. 545]{tK68}.) The
proof of the theorem that we presented is a minor modification of
Douglas's proof of Theorem 5 in \cite{rD69}. Popescu proved Theorem
\ref{Theorem:PureSuperHarm} in the setting of free semigroup algebras
as \cite[Theorem 3.7]{gP03} and developed a number of other important
features of $\Phi_{\eta}$ in that setting. Many of them extend to
our context, but we will not pursue all of them here.

Our primary objective is to prove the following theorem that identifies
$\mathcal{V}_{ac}(T,\sigma)$ for a completely contractive representation
$(T,\sigma)$ of $(E,M)$.

\begin{theorem}\label{Theorem:AbsContT} Let $(T,\sigma)$ be a completely
contractive representation of $(E,M)$ on the Hilbert space $H$,
let $(V,\rho)$ be the minimal isometric dilation of $(T,\sigma)$
acting on a Hilbert space $K$ containing $H$, and let $P$ denote
the projection of $K$ onto $H$. Then $K\ominus H$ is contained
in $\mathcal{V}_{ac}(V,\rho)$ and the following sets are equal.
\begin{enumerate}
\item [(1)] $\mathcal{V}_{ac}(T,\sigma)$.
\item [(2)]$H\cap\mathcal{V}_{ac}(V,\rho)$.
\item [(3)] $P\mathcal{V}_{ac}(V,\rho)$.
\item [(4)]$\bigcup\{Ran(C)\mid C\in\mathcal{I}((S_{0},\sigma_{0}),(T,\sigma))\}$.
\item [(5)]$\overline{span}\{Ran(Q)\mid Q\mbox{ is a pure superharmonic operator for }\Phi_{T}\}$.
\end{enumerate}
In particular, $(T,\sigma)$ is absolutely continuous if and only
if $(V,\rho)$ is absolutely continuous.

\end{theorem}

\begin{proof} First, observe that the orthogonal complement of $H$
in $K$, $H^{\perp}$, is $\mathcal{F}(E)\otimes_{\sigma_{1}}\mathcal{D}$,
where $\mathcal{D}$ is the closure of the range of $\Delta=(I_{E\otimes H}-\widetilde{T}^{*}\widetilde{T})^{\frac{1}{2}}$
and where $\sigma_{1}$ is the restriction of $\sigma^{E}\circ\varphi(\cdot)$
to $\mathcal{D}$. (See paragraph \ref{par:IsoDilate}.) The restriction
of $(V,\rho)$ to $H^{\perp}=\mathcal{F}(E)\otimes_{\sigma_{1}}\mathcal{D}$
is just the representation induced by $\sigma_{1}$ and, therefore,
is absolutely continuous. Thus $K\ominus H\subseteq\mathcal{V}_{ac}(V,\rho)$.
To see the equality of the indicated subspaces, begin by noting that
the coincidence of the two spaces $\mathcal{V}_{ac}(T,\sigma)$ and
$H\cap\mathcal{V}_{ac}(V,\rho)$ is an immediate consequence of the
fact that for every vector $x\in H$ the two functionals $\omega_{x}\circ(T\times\sigma)$
and $\omega_{x}\circ(V\times\rho)$ are equal. This, in turn, is clear
because for such an $x$, $Px=x$, where $P$ is the projection from
$K$ onto $H$. Consequently, $\omega_{x}\circ(V\times\rho)(\cdot)=\langle V\times\rho(\cdot)x,x\rangle=\langle PV\times\rho(\cdot)Px,x\rangle=\langle T\times\sigma(\cdot)x,x\rangle=\omega_{x}\circ(T\times\sigma)(\cdot)$.

Note in particular, by Theorem \ref{subspace}, the fact that $\mathcal{V}_{ac}(T,\sigma)=H\cap\mathcal{V}_{ac}(V,\rho)$
shows that $\mathcal{V}_{ac}(T,\sigma)$ is a closed subspace of $H$.

Clearly, $H\cap\mathcal{V}_{ac}(V,\rho)$ is contained in $P\mathcal{V}_{ac}(V,\rho)$.
On the other hand, if $x=Py$, with $y\in\mathcal{V}_{ac}(V,\rho)$,
then by Proposition \ref{X}, there is an $X\in\mathcal{I}((S_{0},\sigma_{0}),(V,\rho))$
and a $z\in\mathcal{F}(E)\otimes_{\pi}K_{0}$ such that $y=Xz$ .
Since $H^{\perp}$ is invariant under $V\times\rho$, we see that
$(T\times\sigma)PX=P(V\times\rho)PX=P(V\times\rho)X=PX(S_{0}\times\rho)$.
Thus $x=PXz$ lies in $\bigcup\{Ran(C)\mid C\in\mathcal{I}((S_{0},\sigma_{0}),(T,\sigma))\}$.
On the other hand, the commutant lifting theorem, \cite[Theorem 4.4]{MS98b},
implies that every operator $C\in\mathcal{I}((S_{0},\sigma_{0}),(T,\sigma))$
has the form $PX$ for an operator $X\in\mathcal{I}((S_{0},\sigma_{0}),(V,\rho))$.
Thus, $\bigcup\{Ran(C)\mid C\in\mathcal{I}((S_{0},\sigma_{0}),(T,\sigma))\}\subseteq P\bigvee\{Ran(X)\mid X\in\mathcal{I}((S_{0},\sigma_{0}),(V,\rho))\}=P\mathcal{V}_{ac}(V,\rho)$,
where the last equation is justified by Theorem \ref{subspace}. Thus
$\bigcup\{Ran(C)\mid C\in\mathcal{I}((S_{0},\sigma_{0}),(T,\sigma))\}=P\mathcal{V}_{ac}(V,\rho)$.
To see that $P\mathcal{V}_{ac}(V,\rho)=H\cap\mathcal{V}_{ac}(V,\rho)$,
note that we showed that $\mathcal{V}_{ac}(V,\rho)$ contains $H^{\perp}$
and so the projection onto $\mathcal{V}_{ac}(V,\rho)$ commutes with
$P$. Consequently, $H\cap\mathcal{V}_{ac}(V,\rho)=P\mathcal{V}_{ac}(V,\rho)$,
and so the first four sets (1)--(4) are equal.

From Theorem \ref{Theorem:PureSuperHarm}, we know that if $C\in\mathcal{I}((S_{0},\sigma_{0}),(T,\sigma))$,
then $CC^{*}$ is pure superharmonic for $\Phi_{T}$. Although the
range of $CC^{*}$ may be properly contained in the range of $C$
it is dense in the range of $C$, and so we see that $\bigcup\{Ran(C)\mid C\in\mathcal{I}((S_{0},\sigma_{0}),(T,\sigma))\}\subseteq\overline{span}\{Ran(Q)\mid Q\mbox{ is a pure superharmonic operator for }\Phi_{T}\}$.
The opposite inclusion is an immediate consequence of the opposite
implication in Theorem \ref{Theorem:PureSuperHarm}, which shows that
every pure superharmonic $Q$ for $\Phi_{T}$ has the form $Q=CC^{*}$
for a suitable $C\in\mathcal{I}((S_{0},\sigma_{0}),(T,\sigma))$,
and the fact, already proved, that $\bigcup\{Ran(C)\mid C\in\mathcal{I}((S_{0},\sigma_{0}),(T,\sigma))\}$
is the closed linear space $\mathcal{V}_{ac}(T,\sigma)$. \end{proof}

\begin{corollary}\label{Cor:AbsContT}Suppose $(T,\sigma)$ is a
completely contractive covariant representation of $(E,M)$ on $H$.
Then:
\begin{enumerate}
\item $\mathcal{V}_{ac}(T,\sigma)=0$ if and only if $\mathcal{I}((S_{0},\sigma_{0}),(T,\sigma))=\{0\}$.
\item If $\Vert\widetilde{T}\Vert<1$, then $(T,\sigma)$ is absolutely
continuous.
\end{enumerate}
\end{corollary}\begin{proof}The first assertion is immediate from
part ($4$) of Theorem \ref{Theorem:AbsContT}. The second assertion
is immediate from part ($5$) since when $\Vert\widetilde{T}\Vert<1$,
the identity is a pure superharmonic for operator $\Phi_{T}$ by equation
\ref{eq:Phi_T_n} and the fact that $\Vert\widetilde{T_{n}}\Vert\to0$
as $n\to\infty$.\end{proof}

\begin{remark}\label{rem:Nonselfadjoint_vageries} With Corollary
\ref{Cor:AbsContT} in hand, it is easy to pick up on a point raised
at the end of paragraph \ref{par:Intertwining}. Let $(T,\sigma)$
be completely contractive covariant representation of $(E,M)$ on
a Hilbert space $H$ and assume that $\Vert\widetilde{T}\Vert<1$.
Then $(T,\sigma)$ is absolutely continuous by Corollary \ref{Cor:AbsContT},
which means that $\mathcal{I}((S_{0},\sigma_{0}),(T,\sigma))$ is
quite large. On the other hand, it is easy to see that if $\Vert\widetilde{T}\Vert<1$,
then $\mathcal{I}((T,\sigma),(S_{0},\sigma_{0}))=0$. Indeed, if $C\in\mathcal{I}((T,\sigma),(S_{0},\sigma_{0}))$,
then $C\widetilde{T}=\widetilde{S_{0}}(I_{E}\otimes C)$. From this
it follows that $C\widetilde{T}_{n}=(\widetilde{S}_{0})_{n}(I_{E^{\otimes n}}\otimes C)$
for all $n$. Since each $(\widetilde{S)}_{n}$ is isometric, this
equation implies that \[
(\widetilde{S}_{0})_{n}^{*}C\widetilde{T}_{n}=I_{E^{\otimes n}}\otimes C.\]
 We conclude that $C=0$, since the left hand side of this equation
goes to zero in norm, while the right hand side has norm $\Vert C\Vert$
for all $n$.\end{remark}

We note in passing that when Theorem \ref{Theorem:AbsContT} is specialized
to the setting when $M=E=\mathbb{C}$, it yields an improvement of
\cite[Corollary 5.5]{rD69} in the following sense: If $W$ is a unitary
operator on a Hilbert space $H$, then its absolutely continuous subspace
is the closed span of the ranges of all the pure superharmonic operators
with respect to the automorphism of $B(H)$ induced by $W$; it is
also the union of the ranges of all the operators that intertwine
the unilateral shift of infinite multiplicity and $W$.

Recall that if $\mathcal{A}$ is an algebra of operators on a Hilbert
space $H$, then a subspace $\mathcal{M}$ of $H$ is called hyperinvariant
for $\mathcal{A}$ if and only if $\mathcal{M}$ is invariant under
every operator in $\mathcal{A}$ and every operator in the commutant
of $\mathcal{A}$. One important feature of this notion is that when
$\mathcal{A}$ is generated by single normal operator $T$, say, then
the hyperinvariant subspaces of $\mathcal{A}$ are precisely the spectral
subspaces of $T$. Thus in a sense, hyperinvariant subspaces for an
algebra should be viewed as analogues of spectral subspaces for an
operator. One needs to take this extended perspective with a grain
of salt, however, since spectral subspaces need not be central, i.e.,
the projection $P$ onto a hyperinvariant subspace need not lie in
the center of $\mathcal{A}$. Nevertheless, knowing that a subspace
is hyperinvariant for an algebra is useful information. Evidently,
if $(T,\sigma)$ is a completely contractive covariant representation
of $(E,M)$ then $\mathcal{V}_{ac}(T,\sigma)$ is invariant under
$T\times\sigma(\mathcal{T}_{+}(E))$ by part (4) of Theorem \ref{Theorem:AbsContT}.
Indeed, if $h\in\mathcal{V}_{ac}(T,\sigma)$, then there is a vector
$x\in\mathcal{F}(E)\otimes_{\pi}K_{0}$ and an operator $C\in\mathcal{I}((S_{0},\sigma_{0}),(T,\sigma))\}$
such that $h=Cx$. Then for $\xi\in E$ and $a\in M$, $T(\xi)h=T(\xi)Cx=CS_{0}(\xi)x$
and $\sigma(a)h=\sigma(a)Cx=C\pi^{\mathcal{F}(E)}x$ are in $\mathcal{V}_{ac}(T,\sigma).$
The next result, a consequence of Theorems \ref{Theorem:AbsContT}
and \ref{Theorem:PureSuperHarm}, shows that $\mathcal{V}_{ac}(T,\sigma)$
is hyperinvariant for $T\times\sigma(\mathcal{T}_{+}(E))$.

\begin{theorem}\label{Theorem:Spectral Subspace}For $i=1,2$, let
$(T_{i},\sigma_{i})$ be a completely contractive covariant representation
of $(E,M)$ on a Hilbert space $H_{i}$ and suppose that $R:H_{1}\to H_{2}$
intertwines $(T_{1},\sigma_{1})$ and $(T_{2},\sigma_{2})$. Then
$R\mathcal{V}_{ac}(T_{1},\sigma_{1})\subseteq\mathcal{V}_{ac}(T_{2},\sigma_{2})$.
In particular, the absolutely continuous subspace of a completely
contractive covariant representation of $(E,M)$ is hyperinvariant
for its image. \end{theorem}

\begin{proof}By Theorem \ref{Theorem:AbsContT}, $\mathcal{V}_{ac}(T_{i},\sigma_{i})=\bigcup\{Ran(C)\mid C\in\mathcal{I}((S_{0},\sigma_{0}),(T_{i},\sigma_{i}))\}$,
$i=1,2$. If $C\in\mathcal{I}((S_{0},\sigma_{0}),(T_{1},\sigma_{1}))$,
then $RC(S_{0}\times\sigma_{0})=R(T_{1}\times\sigma_{1})C=(T_{2}\times\sigma_{2})RC$,
which shows that $R\mathcal{I}((S_{0},\sigma_{0}),(T_{1},\sigma_{1}))\subseteq\mathcal{I}((S_{0},\sigma_{0}),(T_{2},\sigma_{2}))$.
Since $R(Ran(C))=Ran(RC)$, we conclude that $R\mathcal{V}_{ac}(T_{1},\sigma_{1})\subseteq\mathcal{V}_{ac}(T_{2},\sigma_{2})$.
\end{proof}

We come now to the main result of this section, which provides criteria
for deciding when a completely contractive covariant representation
of $\mathcal{T}_{+}(E)$ extends to an ultra-weakly continuous representation
of $H^{\infty}(E)$.

\begin{theorem}\label{Theorem:ExtensionThm} Let $(T,\sigma)$ be
a completely contractive covariant representation of $\mathcal{T}_{+}(E)$
on a Hilbert space $H$ and let $(V,\rho)$ be its minimal isometric
dilation acting on a Hilbert space $K$ containing $H$. Then the
following assertions are equivalent.
\begin{enumerate}
\item $T\times\sigma$ extends to an ultra-weakly continuous, completely
contractive representation of $H^{\infty}(E)$.
\item $(T,\sigma)$ is absolutely continuous.
\item $\overline{span}\{Ran(Q)\mid Q\mbox{ is a pure superharmonic operator for }\Phi_{T}\}=H$.
\item $V\times\rho$ extends to an ultra-weakly continuous, completely contractive
representation of $H^{\infty}(E)$.
\item $(V,\rho)$ is absolutely continuous.
\item $H$ is contained in $\mathcal{V}_{ac}(V,\rho)$.
\end{enumerate}
\end{theorem}

Of course, we could add a number of other conditions to this list.
However, these are the principal ones and more important, none refers
to the {}``external construct'' $(S_{0},\sigma_{0})$. That is to
say, all the conditions listed refer to \emph{intrinsic} features
of the representation $(T,\sigma)$.

\begin{proof}Of course, much of the proof amounts to assembling pieces
already proved. Thus, (2) and (3) are equivalent by virtue of Theorem
\ref{Theorem:AbsContT}. Likewise, (4) and (5) are equivalent by Theorem
\ref{abscontrep}.

The equivalence of (5) and (6) follows from the observation that $(V,\rho)$
is absolutely continuous if and only if $H\subseteq\mathcal{V}_{ac}(V,\rho)$.
(This is because $H^{\perp}\subseteq\mathcal{V}_{ac}(V,\rho)$, as
we noted in the proof of Theorem \ref{Theorem:AbsContT}.) Thus (5)
and (6) are equivalent. Conditions (2) and (6) are equivalent by virtue
of the equation $\mathcal{V}_{ac}(T,\sigma)=H\cap\mathcal{V}_{ac}(V,\rho)$
proved in Theorem \ref{Theorem:AbsContT}. Thus conditions (2) --
(6) are equivalent.

But (1) certainly implies (2). On the other hand, if (2) holds, so
does (4). If $\widetilde{V\times\rho}$ denotes the ultra-weakly continuous
extension of $V\times\rho$ to $H^{\infty}(E),$ then it is clear
that $P(\widetilde{V\times\rho})|H$ is an ultra-weakly continuous
extension of $T\times\sigma$ to $H^{\infty}(E)$. \end{proof}

\section{Further Corollaries and Examples\label{sec:Further-Corollaries}}

\subsection{Invariant States}

In this subsection we collect a number of results that show how the
notion of absolute continuity relates to the notion of an invariant
state for a completely positive map.

\begin{definition}Let $\Phi$ be a normal completely positive map
on a $W^{*}$-algebra $M$. We let $P_{ac}=P_{ac}(\Phi)$ denote the
smallest projection in $M$ that dominates the range projection of
each pure superharmonic element of $\Phi$ and we call $P_{ac}$ the
\emph{absolutely continuous projection} for $\Phi$. \end{definition}
If $M$ is represented faithfully on a Hilbert space $H$ by a normal
representation $\rho$, and if $\Phi$ is contractive, then the range
of $\rho(P_{ac})$ is the absolutely continuous subspace $\mathcal{V}_{ac}(T,\sigma)$
for the identity representation $(T,\sigma)$ of $(E,\rho(M)')$,
where $E$ is the Arveson-Stinespring correspondence determined by
$\Phi$ (and $\rho$), by Theorem \ref{Theorem:Identity Representation}
and Theorem \ref{Theorem:AbsContT}. So the terminology is consistent
with the developments in Sections 4 and 5, and it makes sense without
having to assume $\Phi$ is contractive.

\begin{definition}If $\Phi$ be a normal completely positive map
on a $W^{*}$-algebra $M$, then a normal state $\omega$ on $N$
is called \emph{periodic of period $k$} with respect to $\Phi$,
if $k$ is the least positive integer such that $\omega\circ\Phi^{k}=\omega$.
We denote the collection of all normal period states of period $k$
for $\Phi$ by $\mathcal{P}_{k}$ or by $\mathcal{P}_{k}(\Phi,M)$.\end{definition}

Recall that if $\omega$ is a normal state on a $W^{*}$-algebra $N$
then there is a largest projection $e\in N$ such that $\omega(e)=0$.
The projection $e^{\perp}:=1-e$ is called the \emph{support projection}
for $\omega$, which we shall denote by $supp(\omega)$. As is customary,
we often identify a projection with its range and we shall think of
$supp(\omega)$ as a subspace of whatever Hilbert space on which $N$
may be found to be acting. Our aim now is to prove

\begin{theorem}\label{Theorem:Periodic} Let $\Phi$ be a normal
completely positive map on the $W^{*}$-algebra $M$. If $\omega$
is a normal state on $M$ that is periodic for $\Phi,$ then its support
projection is orthogonal to $P_{ac}(\Phi)$. \end{theorem}

The proof is based on the following lemma.

\begin{lemma}\label{Lemma:Powers} For each $k\geq2$, $P_{ac}(\Phi^{k})=P_{ac}(\Phi)$.\end{lemma}

\begin{proof} It is clear that $P_{ac}(\Phi)\leq P_{ac}(\Phi^{k})$
since every superharmonic operator for $\Phi$ is superharmonic for
$\Phi^{k}$. The problem we face with trying to prove the reverse
inclusion is that in general a pure superharmonic operator for $\Phi^{k}$
is not evidently a pure superharmonic operator for $\Phi$. So what
we prove is that if $Q$ is a pure superharmonic operator for $\Phi^{k}$,
then there is a pure superharmonic operator $R$ for $\Phi$ such
that $Q\leq R$. This will show that $P_{ac}(\Phi^{k})\subseteq P_{ac}(\Phi)$.
Our choice for $R$ is $Q+\Phi(Q)+\cdots+\Phi^{k-1}(Q)$. Evidently,
this operator dominates $Q$. So it suffices to show that it is a
pure superharmonic operator for $\Phi$. Since $\Phi(Q+\Phi(Q)+\cdots+\Phi^{k-1}(Q))=\Phi(Q)+\Phi^{2}(Q)+\cdots+\Phi^{k}(Q)$
and $\Phi^{k}(Q)\leq Q$ by hypothesis, we see that \begin{eqnarray*}
\Phi(Q+\Phi(Q)+\cdots+\Phi^{k-1}(Q)) & = & \Phi(Q)+\Phi(Q)+\cdots+\Phi^{k}(Q)\\
 & \leq & \Phi(Q)+\Phi(Q)+\cdots+\Phi^{k-1}(Q)+Q\\
 & = & Q+\Phi(Q)+\cdots+\Phi^{k-1}(Q).\end{eqnarray*}
 Thus $Q+\Phi(Q)+\cdots+\Phi^{k-1}(Q)$ is superharmonic for $\Phi$.
This means that the sequence $\{\Phi^{n}(Q+\Phi(Q)+\cdots+\Phi^{k-1}(Q))\}_{n\geq0}$
is decreasing. Thus to show it tends strongly to zero, it suffices
to show that a subsequence tends to zero strongly. But the sequence
$\{\Phi^{nk}(Q+\Phi(Q)+\cdots+\Phi^{k-1}(Q))\}_{n\geq0}$ has this
property since \[
\Phi^{nk}(Q+\Phi(Q)+\cdots+\Phi^{k-1}(Q))=\Phi^{nk}(Q)+\Phi(\Phi^{nk}(Q))+\cdots+\Phi^{k-1}(\Phi^{nk}(Q))\]
 and each term on the right hand side tends to zero monotonically
as $n\to\infty$.\end{proof}\bigskip{}

\textbf{Proof of Theorem \ref{Theorem:Periodic}}\textbf{\emph{.}}\emph{
}Since $P_{ac}(\Phi^{k})=P_{ac}(\Phi)$ by Lemma \ref{Lemma:Powers},
it suffices to show that if $\omega$ is an invariant state for $\Phi$
then $supp(\omega)\perp P_{ac}(\Phi)$. But if $Q$ is a pure superharmonic
operator for $\Phi$, then the equation $\omega\circ\Phi=\omega$
implies that $\omega(Q)=0$. Thus the range projection of $Q$ is
orthogonal to the support of $\omega$, and so $supp(\omega)\perp P_{ac}(\Phi)$.
$\square$

%
{}

Recall from \cite[Page 404 ff.]{MSHardy} that if $(T,\sigma)$ is
an ultra-weakly continuous completely contractive representation of
$(E,M)$ acting on the Hilbert space $H$, then there is a largest
subspace $H_{1}$ of $H$ that is invariant under $(T\times\sigma)(\mathcal{T}_{+}(E))^{*}$
on which $\widetilde{T}^{*}$ acts isometrically. It is given by the
equation \[
H_{1}=\{h\in H\mid\Vert\widetilde{T_{n}}^{*}h\Vert=\Vert h\Vert,\mbox{\, for all}\, n\geq0\}.\]
 The representations $(T,\sigma)$ and $T\times\sigma$ are called
\emph{completely non coisometric} (abbreviated c.n.c.) in case $H_{1}=\{0\}$.
We record for reference the following theorem that is proved as part
of Theorem 7.3 in \cite{MSHardy}.

\begin{theorem}\label{Theorem:CNCExtension} If $(T,\sigma)$ is
a completely non coisometric, ultra-weakly continuous, completely
contractive representation of $(E,M)$ on a Hilbert space $H$, then
$T\times\sigma$ extends to an ultra-weakly continuous, completely
contractive representation of $H^{\infty}(E)$, and so $(T,\sigma)$
is absolutely continuous, by Theorem \ref{Theorem:ExtensionThm}.
\end{theorem} Of course, if $\Vert\widetilde{T}\Vert<1$, then $(T,\sigma)$
is completely non coisometric. Thus, Theorem \ref{Theorem:CNCExtension}
improves upon Corollary \ref{Cor:AbsContT}.

\begin{theorem}Let $(T,\sigma)$ be an ultra-weakly continuous, completely
contractive covariant representation of $(E,M)$ on a Hilbert space
$H$. If $\sigma(M)'$ is finite dimensional, then $T\times\sigma$
extends to an ultra-weakly continuous representation of $H^{\infty}(E)$
on $H$ if and only if $(T,\sigma)$ is completely non coisometric.
\end{theorem} \begin{proof}We know that $T\times\sigma$ extends
if $(T,\sigma)$ is completely non coisometric regardless of the dimension
of $\sigma(M)'$. So we attend to the reverse implication. If $(T,\sigma)$
has a nonzero coisometric part $H_{1}$ then we can compress $(T,\sigma)$
to $H_{1}$ to get a new representation $(T_{1},\sigma_{1})$ that
is fully coisometric, i.e., $\widetilde{T_{1}}\widetilde{T_{1}}^{*}=I_{H_{1}}$.
Of course $\sigma_{1}(M)'$ will be finite dimensional, too. So, it
suffices to assume from the outset that $(T,\sigma)$ is fully coisometric,
which is tantamount to assuming $\Phi_{T}$ is unital. But a unital
completely positive map on a finite dimensional $W^{*}$-algebra admits
an invariant normal state. The support of such a state, as we have
seen in Theorem \ref{Theorem:Periodic}, must be orthogonal to $\mathcal{V}_{ac}(T,\sigma)$
and so $(T,\sigma)$ cannot be absolutely continuous. By Theorem \ref{Theorem:ExtensionThm},
we conclude that $T\times\sigma$ does not admit an extension to an
ultra-weakly continuous representation of $H^{\infty}(E)$. \end{proof}

\begin{corollary}If $(T_{1},T_{2},\cdots,T_{d})$ is a row contraction
where the $T_{i}$ act on a finite dimensional Hilbert space, then
the map which takes the $i^{th}$ generator $S_{i}$ of $H^{\infty}(\mathbb{C}^{d})$
to $T_{i}$ extends to an ultra-weakly continuous representation of
$H^{\infty}(\mathbb{C}^{d})$ if and only if $(T_{1},T_{2},\cdots,T_{d})$
is completely non coisometric.\end{corollary}

\begin{example}As an extremely simple, yet somewhat surprising concrete
example, consider the case when $d=2$ and $T_{1}=\left(\begin{array}{cc}
0 & 1\\
0 & 0\end{array}\right)$ and $T_{2}=\left(\begin{array}{cc}
0 & 0\\
1 & 0\end{array}\right)$ are acting on $H=\mathbb{C}^{2}$. Then $(T_{1},T_{2})$ is a row
coisometry and so is not absolutely continuous. In fact, since $\Phi_{T}$
preserves the trace, as is easy to see, the absolutely continuous
subspace of $\mathbb{C}^{2}$ reduces to zero. Thus these matrices
do not come from an ultra-weakly continuous representation of $H^{\infty}(\mathbb{C}^{2})$
even though $T_{1}$and $T_{2}$ are both nilpotent. \end{example}

\begin{example}Suppose the $W^{*}$-algebra $M$ is given and that
$E$ comes from a unital endomorphism $\alpha$, i.e., suppose $E=_{\alpha}M$,
which is $M$ as a (right) Hilbert module over $M$, with the left
action of $M$ given by $\alpha$. If $(T,\sigma)$ is an ultra-weakly
continuous, completely contractive covariant representation of $(E,M)$
on a Hilbert space $H$, then the map $x\otimes h\to\sigma(x)h$ extends
to an isomorphism from $E\otimes_{\sigma}H$ to $H$ that allows us
to view $\widetilde{T}$ as an operator $T_{0}$ on $H$ that has
the property \[
T_{0}\sigma(\alpha(a))=\sigma(a)T_{0}\]
 for all $a\in M$. The map $\Phi_{T}$, then, is given by the formula
$\Phi_{T}(x)=T_{0}xT_{0}^{*}$. If $T_{0}$ is a unitary operator,
then \cite[Corollary 5.5]{rD69} tells us that $\mathcal{V}_{ac}(T,\sigma)$
is contained in the absolutely continuous subspace for $T_{0}$. It
is, however, quite possible for the two spaces to be distinct. Indeed,
let $M$ be $L^{\infty}(\mathbb{T})$ with $\sigma$ the multiplication
representation on $L^{2}(\mathbb{T})$. Let $\alpha$ be implemented
by an irrational rotation on $\mathbb{T}$, say $z\to e^{i\theta}z$.
Finally, let $T_{0}$ be the unitary operator on $L^{2}(\mathbb{T})$
defined by the equation \[
T_{0}\xi(z)=z\xi(e^{i\theta}z).\]
 Then $\sigma(M)'=L^{\infty}(\mathbb{T})$ and because $\Phi_{T}(\varphi)(z)=\varphi(e^{i\theta}z)$,
as may be easily calculated, we see that $\mathcal{V}_{ac}(T,\sigma)=\{0\}$
because integration against Lebesgue measure gives an invariant faithful
normal state on $M$. On the other hand, it is easy to see that $T_{0}$
is unitarily equivalent to the bilateral shift of multiplicity one.
Indeed, $T_{0}$ leaves $H^{2}(\mathbb{T})$ invariant and $H^{2}(\mathbb{T})\ominus T_{0}H^{2}(\mathbb{T})=H^{2}(\mathbb{T})\ominus zH^{2}(\mathbb{T})$,
where we have identified $z$ with multiplication by $z$. It is an
easy matter to check that $\mathcal{D}:=H^{2}(\mathbb{T})\ominus T_{0}H^{2}(\mathbb{T})$
is a complete wandering subspace for $T_{0}$, i.e., $L^{2}(\mathbb{T})=\sum_{k\in\mathbb{Z}}^{\oplus}T_{0}^{k}\mathcal{D}$.
Thus, we see that $T_{0}$ is an absolutely continuous unitary operator.\end{example}

\subsection{Markov Chains}

Our next example connects the theory we have been developing with
the theory of Markov chains. Recall from Example \ref{Example:MarkovChain1}
that a sub-Markov matrix is an $n\times n$ matrix $A$ with non-negative
entries $a_{ij}$ with the property that the row-sums $\sum_{j}a_{ij}\leq1$.
We think of $A$ as defining a completely positive map $\Phi$ on
the $W^{*}$-algebra of all diagonal $n\times n$ matrices, $D_{n}$.
If $d=diag\{d_{1},d_{2},\cdots,d_{n}\}$, then $\Phi(d):=diag\{\sum_{j}a_{1j}d_{j},\sum_{j}a_{2j}d_{j},\cdots,\sum_{j}a_{nj}d_{j}\}$.
We want to describe the absolutely continuous projection $\mathcal{P}_{ac}(\Phi)$.

Recall that the norm of $A$ as an operator on $\ell^{\infty}(\{1,2,\cdots,n\})$
is at most $1$ and if the norm of $A$ is $1$ then $1$ is an eigenvalue
for $A$. In this event, we can find an element $\underline{z}=(z_{1},z_{2},\cdots,z_{n})$
such that $z_{i}\geq0$ for all $i$ and such that $A\underline{z}=\underline{z}.$
This is a consequence of the Perron-Frobenius theory (see \cite[pp. 64,65]{fG59}.)
It applies as well to the transpose of $A$. We will call such an
eigenvector a \emph{non-negative left eigenvector} for $A$. We will
call the \emph{support of} $\underline{z}$ the set of $i\in\{1,2,\cdots,n\}$
such that $z_{i}\neq0$, and we will say that $\underline{z}$ \emph{has
full support} if $z_{i}\neq0$ for all $i$. The analysis in Section
III.4 of \cite{fG59} shows that after conjugating $A$ by a permutation
matrix, $A$ has the form described in the following lemma:

\begin{lemma}\label{lem: canonical form}If $A$ is a sub-Markov
matrix with spectral radius $1$, then there is a permutation matrix
$S$ so that $SAS^{-1}$ has the (block) lower triangular form\begin{equation}
\left[\begin{array}{ccccc}
A_{1}\\
0 & A_{2} &  & 0\\
0 &  & \ddots\\
0 & 0 & 0 & A_{k-1k-1}\\
 & U &  &  & A_{kk}\end{array}\right],\label{eq:canonical matrix}\end{equation}
 where:
\begin{enumerate}
\item For $i=1,2,\cdots,k-1$, each $A_{ii}$ has $1$ as an eigenvalue,
and the corresponding left eigenvector is nonnegative and has full
support.
\item $A_{kk}$ has spectral radius less than $1$.
\end{enumerate}
\end{lemma}

%
{}

\begin{definition}We call the matrix \eqref{eq:canonical matrix}
the \emph{canonical form} of the sub-Markov matrix $A$. The set of
indices $E_{i}$ such that the matrix entries of $A_{ii}$ are indexed
by $E_{i}\times E_{i}$ will be called the support of $A_{ii}$ as
well as the support of the left Perron-Frobenius eigenvector $z^{(i)}$
for $A_{ii}$, corresponding to $1$, $i=1,2,\ldots,k-1$. \end{definition}

\begin{lemma}Let $\Phi$ be a normal completely positive map on a
$W^{*}$-algebra $N$ and let $p$ be a central projection of $N$
that is invariant under $\Phi$ in the sense that $\Phi(p)\leq p$.
Then $pP_{ac}(\Phi)=P_{ac}(\Phi|_{pNp})$.\end{lemma} \begin{proof}
Suppose $Q$ is a pure superharmonic element of $N$ for $\Phi$.
Then the sequence $\{\Phi^{n}(Q)\}_{n\in\mathbb{N}}$ decreases to
zero. So $\Phi(pQp)=p\Phi(pQp)p\leq p\Phi(pQp+p^{\perp}Qp^{\perp})p=p\Phi(Q)p\leq pQp$.
Therefore, since $\Phi(Q)$ is also pure superharmonic for $\Phi$,
$\Phi^{2}(pQp)=\Phi(\Phi(pQp))=\Phi(p\Phi(pQp)p)\leq\Phi(p\Phi(Q)p)\leq p\Phi(\Phi(Q))p\leq\Phi^{2}(Q)$.
Continuing in this fashion, we see that $\Phi^{n}(pQp)\leq\Phi^{n}(Q)$
for all $n$, and so the range projection of $pQp$ is less than or
equal to $P_{ac}(\Phi\vert_{pNp})$. That is, $pP_{ac}(\Phi)\leq P_{ac}(\Phi\vert_{pNp})$.
The reverse inequality is clear, since if $pQp$ is pure superharmonic
for $\Phi\vert_{pNp}$, then it certainly is pure superharmonic for
$\Phi$.\end{proof}

\begin{theorem}\label{Theorem:MarkovChain2}Let $\Phi$ be the completely
positive map on $\ell^{\infty}(\{1,2,\cdots,n\})$ induced by a sub-Markov
matrix $A$ and assume $A$ is written in its canonical form \eqref{eq:canonical matrix}.
Then $P_{ac}(\Phi)$ is the support projection of $A_{kk}.$\end{theorem}
\begin{proof}If $E$ is the union of the support projections for
$A_{i},$ $i=1,2,\cdots,k-1,$ then the projection $1_{E}$ in $\ell^{\infty}(\{1,2,\ldots,n\})$
is coinvariant for $A$ and the sum of the vectors $z:=\sum_{i=1}^{k-1}z^{(i)}$
determines an invariant state $\omega_{z}$ for $\Phi$ via formula
$\omega_{z}(\underline{d})=\sum_{i=1}^{k-1}z^{(i)}\cdot\underline{d_{i}}$,
where $\underline{d_{i}}$ is the restriction of $\underline{d}$
to the support of $z^{(i)}$, $E_{i}$, and $z^{(i)}\cdot\underline{d_{i}}$
denotes the dot product of the two tuples. The state $\omega_{z}$
is faithful on $1_{E}\ell^{\infty}(\{1,2,\ldots,n\})$, and so $P_{ac}(\Phi)\leq1_{E_{k}}=1_{E^{c}}$,
by Theorem \ref{Theorem:Periodic}. On the other hand, since $1_{E_{k}}\ell^{\infty}(\{1,2,\ldots,n\})$
is invariant for $\Phi$, and since the spectral radius of $A_{k}$,
which is the matrix of the restriction of $\Phi$ to $1_{E_{k}}\ell^{\infty}(\{1,2,\ldots,n\})$,
is less than $1$, it is clear that $\Phi^{n}(1_{E_{k}})\to0$ in
norm. This implies that $1_{E_{k}}$ is pure superharmonic for $\Phi$
and, therefore, that $1_{E_{k}}\leq P_{ac}(\Phi)$. \end{proof}

\subsection{Similarity of representations}

Many of the results in this subsection are analogues of theorems in
Popescu's paper \cite{gP03} and many of his proofs work here, as
well. We focus on those features and proofs that take advantage of
our perspective that focuses on the connection between completely
positive maps and intertwiners. So a number of our arguments are different
from the ones given in \cite{gP03}. Suppose $\sigma:M\to B(H)$ is
a normal representation of $M$ on the Hilbert space $H$ and suppose
$\eta\in E^{\sigma}$. Then we may form the completely bounded bimodule
map $\widehat{\eta^{*}}:E\to B(H)$ discussed in the paragraph before
Theorem \ref{Theorem:PureSuperHarm}. Let $\mathcal{T}_{+0}(E)$ be
the \emph{linear span} of $\varphi_{\infty}(M)$ and the operators
$\{T_{\xi}\mid\xi\in E^{\otimes n},\, n\geq1\}$. Then $\mathcal{T}_{+0}(E)$
is the algebra generated by $\varphi_{\infty}(M)$ and $\{T_{\xi}\mid\xi\in E\}$,
and given $\sigma$ and $\widehat{\eta^{*}}$, we can define a representation
of $\mathcal{T}_{+0}(E)$ on $H$, denoted $\widehat{\eta^{*}}\times\sigma$
by the formulae \[
\widehat{\eta^{*}}\times\sigma(\varphi_{\infty}(a))=\sigma(a),\qquad a\in M,\]
and \[
\widehat{\eta^{*}}\times\sigma(T_{\xi_{1}}T_{\xi_{2}}\cdots T_{\xi_{n}})=\widehat{\eta^{*}}(\xi_{1})\widehat{\eta^{*}}(\xi_{2})\cdots\widehat{\eta^{*}}(\xi_{n}),\qquad\xi_{i}\in E,\, i=1,2,\cdots,n.\]
We are interested in understanding when $\widehat{\eta^{*}}\times\sigma$
extends to a completely bounded representation of $\mathcal{T}_{+}(E)$
in $B(H)$. Thanks to the famous theorem of Paulsen \cite{vP84},
this will happen if and only if $\widehat{\eta^{*}}\times\sigma$
is similar to a completely contractive representation of $\mathcal{T}_{+}(E)$.
By \cite[Theorem 3.10]{MS98b} (see also paragraph \eqref{par:CCReps}),
this will happen if and only if $\widehat{\eta^{*}}\times\sigma$
is similar to $\widehat{\zeta^{*}}\times\sigma_{1}$, for a $\zeta\in\overline{\mathbb{D}(E^{\sigma_{1}})^{*}}$
.

Thus, we are led to investigating the similarity properties of the
completely bounded maps $\widehat{\eta^{*}}\times\sigma$. For this
purpose, observe that if $\widehat{\eta_{1}^{*}}\times\sigma_{1}$
is similar to $\widehat{\eta_{2}^{*}}\times\sigma_{1}$ then there
is an invertible operator $S$ on $H$ such that $\widehat{\eta_{1}^{*}}\times\sigma_{1}(a)S=S\widehat{\eta_{2}^{*}}\times\sigma_{2}(a)$
for all $a\in\mathcal{T}_{+}(E)$. In particular, when $a\in\varphi_{\infty}(M)$,
$\sigma_{1}(a)S=S\sigma_{2}(a)$, i.e., $\sigma_{1}$ and $\sigma_{2}$
are similar. Since they are $*$-maps, $\sigma_{1}$ and $\sigma_{2}$
must be unitarily equivalent. But observe that if $U$ is a Hilbert
space isomorphism from the space of $\sigma_{1}$, $H_{1}$, to the
space of $\sigma_{2}$, then the map $\eta\to(I_{E}\otimes U)\eta U^{-1}$
is a complete isometric isomorphism between $E^{\sigma_{1}}$ and
$E^{\sigma_{2}}$ that is also a homeomorphism for the ultra-weak
topologies. Since we will be interested here in the norm properties
of $\eta$'s and related constructs the difference between the duals
of unitarily equivalent representations may safely be ignored. Thus,
in particular, when given a similarity $S$ between $\widehat{\eta_{1}^{*}}\times\sigma_{1}$
and $\widehat{\eta_{2}^{*}}\times\sigma_{2}$ we may identify $\sigma_{1}$
and $\sigma_{2}$ with one $\sigma$ and assume that $S$ lies in
$\sigma(M)'$. In this case, $\eta_{1}^{*}(I\otimes S)=S\eta_{2}^{*}$,
i.e. $S^{-1}\eta_{1}^{*}(I\otimes S)=\eta_{2}^{*}$. Conversely, any
invertible $S$ in the commutant of $\sigma(M)$ that satisfies the
equation $S^{-1}\eta_{1}^{*}(I\otimes S)=\eta_{2}^{*}$ implements
a similarity between $\widehat{\eta_{1}^{*}}\times\sigma(\cdot)$
and $\widehat{\eta_{2}^{*}}\times\sigma(\cdot)$.

\begin{definition}\label{def:Similarity-eta}We introduce several
terms we will use in the sequel.
\begin{enumerate}
\item If $\eta$ and $\zeta$ are two elements of $E^{\sigma}$, then we
shall say they are \emph{similar} if there is an invertible operator
$S$ in $\sigma(M)'$ such that $S^{-1}\cdot\eta\cdot S=(I_{E}\otimes S)^{-1}\eta S=\zeta$.
\item Let $\Phi_{1}$ and $\Phi_{2}$ be two completely positive maps defined
on a $W^{*}$-algebra $N$. We say they are \emph{similar} if and
only if there is an invertible operator $R\in N$ such that \[
\psi_{R}^{-1}\circ\Phi_{1}\circ\psi_{R}=\Phi_{2},\]
where $\psi_{R}$ is the complete positive map on $M$ defined by
the formula $\psi_{R}(a)=RaR^{*}.$
\item We say that an $\eta\in E^{\sigma}$ belongs to class $C_{\cdot0}$
in case the identity is a pure superharmonic operator for $\Phi_{\eta}$.
\end{enumerate}
\end{definition}

\begin{remark}\label{rem:On_similarities}Several points may be helpful.
\begin{enumerate}
\item We have taken the definition of similarity for completely positive
maps from Popescu \cite{gP03}.
\item It is easy to see that if elements $\eta,\zeta\in E^{\sigma}$ are
similar, then $\Phi_{\eta}$ and $\Phi_{\zeta}$ are similar. Indeed,
if $R$ is an invertible element of $\sigma(M)'$, then \begin{align*}
\psi_{R}^{-1}\circ\Phi_{\eta}\circ\psi_{R}(a) & =R^{-1}(\eta^{*}(I_{E}\otimes RaR^{*})\eta)R^{*-1}\\
= & R^{-1}\eta^{*}(I_{E}\otimes R)(I_{E}\otimes a)(I_{E}\otimes R^{*})\eta R^{*-1}\\
= & \Phi_{\zeta},\end{align*}
where $\zeta=(I_{E}\otimes R^{*})\eta R^{*-1}$. The converse assertion
is not true owing to the nonuniqueness of representing a completely
positive map $\Phi$ in the form $\Phi_{\eta}$.
\item In fact, it is convenient to use bimodule notation: $\eta$ and $\zeta$
are similar if and only if there is an invertible $r$ such that $r\cdot\eta\cdot r^{-1}=\zeta$.
Observe that in this case, $r\cdot\eta_{n}\cdot r^{-1}=\zeta_{n}$
for all $n$.
\item The notion of a $C_{\cdot0}$ element of $E^{\sigma}$ is borrowed
from \cite[Definition 7.14]{MSHardy}. There the norm of the element
was assumed to be at most $1$. It turns out that when $M=\mathbb{C}=E$,
so that $\eta$ is really an operator on Hilbert space, then the terminology
we have adopted agrees with that of Sz.-Nagy and Foia\c{s} in \cite{SzNF70}.
\item In the terminology of \cite{gP03}, a completely positive map $\Phi$
is called pure if and only if $I$ is a pure superharmonic operator
for $\Phi$.
\end{enumerate}
\end{remark}

Our first result gives a necessary and sufficient condition for $\widehat{\eta^{*}}\times\sigma$
to extend to a completely bounded representation on $\mathcal{T}_{+}(E)$.
It was inspired by \cite[Theorem 5.13]{gP03}. However, our proof
is somewhat different.

\begin{theorem}\label{Theorem:Similar_to_a_contraction}Let $\sigma:M\to B(H)$
be a normal representation and let $\eta\in E^{\sigma}$. Then the
following conditions are equivalent.
\begin{enumerate}
\item The representation $\widehat{\eta^{*}}\times\sigma$ extends to a
completely bounded representation of $\mathcal{T}_{+}(E)$.
\item $\eta$ is similar to a $\zeta\in\overline{\mathbb{D}(E^{\sigma})}$.
\item $\Phi_{\eta}$ admits an invertible superharmonic operator.
\end{enumerate}
\end{theorem}\begin{proof}If $\widehat{\eta^{*}}\times\sigma$ extends
to a completely bounded representation of $\mathcal{T}_{+}(E)$, then
$\widehat{\eta^{*}}\times\sigma$ must be similar to a completely
contractive representation $\rho$ of $\mathcal{T}_{+}(E)$ by Paulsen's
famous theorem \cite{vP84}. Since $\rho$ must be of the form $\widehat{\zeta^{*}}\times\sigma_{1}$
for some $\zeta$ of norm at most $1$ in $E^{\sigma_{1}}$, $\sigma$
and $\sigma_{1}$ are similar, and therefore unitarily equivalent.
As we noted above, we may identify $\sigma$ with $\sigma_{1}$ and
conclude that $\eta$ is similar to a point in $\overline{\mathbb{D}(E^{\sigma})^{*}}$.
Thus $1.$ implies $2.$ The converse is immediate, since as we noted
above, a similarity between two points $\eta_{1}$ and $\eta_{2}$
in $E^{\sigma}$ implements a similarity between $\widehat{\eta_{1}^{*}}\times\sigma$
and $\widehat{\eta_{2}^{*}}\times\sigma$. Suppose $\eta$ is similar
to a $\zeta$ in $\overline{\mathbb{D}(E^{\sigma})^{*}}$, say $r\cdot\eta\cdot r^{-1}=\zeta$
for some $r\in\sigma(M)'$. Then $r\cdot\eta=\zeta\cdot r$ and so
$\Phi_{\eta}(r^{*}r)=\Phi_{r\cdot\eta}(I)=\Phi_{\zeta\cdot r}(I)=r^{*}\Phi_{\zeta}(I)r\leq r^{*}r$,
since $\Vert\zeta\Vert\leq1$. Thus $r^{*}r$ is superharmonic for
$\Phi_{\eta}$ and since $r$ is invertible by assumption, 3. is proved.
Finally, suppose $\Phi_{\eta}$ admits an invertible super harmonic
operator, say $R$. If $r=R^{\frac{1}{2}}$ and if we set $\zeta=r\cdot\eta\cdot r^{-1}$,
then $\zeta\cdot r=r\cdot\eta$, $\Phi_{\zeta}(I)=r^{-1}\Phi_{\zeta\cdot r}(I)r^{-1}=r^{-1}\Phi_{r\cdot\eta}(I)r^{-1}=r^{-1}\Phi_{\eta}(r^{2})r^{-1}\leq r^{-1}r^{2}r^{-1}=I$,
since $r^{2}$ is superharmonic for $\Phi_{\eta}$. Thus $\Vert\zeta\Vert\leq1$.\end{proof}

The following theorem identifies when an $\eta\in E^{\sigma}$ is
similar to a $C_{\cdot0}$ element in $\overline{\mathbb{D}(E^{\sigma})}$.
It is was inspired by \cite[Theorem 5.11]{gP03}. Again, the proof
is somewhat different.

\begin{theorem}\label{Theorem:Similar_to_a_CzeroPt_contraction}Let
$\sigma$ be a normal representation of $M$ on a Hilbert space $H$
and let $\eta$ be an element of $E^{\sigma}$. Then the following
assertions about $\eta$ are equivalent.
\begin{enumerate}
\item $\eta$ is similar to a $C_{\cdot0}$ element of $\overline{\mathbb{D}(E^{\sigma})}$.
\item There is a positive element $r\in\sigma(M)'$ and positive numbers
$a$ and $b$ so that \begin{equation}
aI_{H}\leq\sum_{n=0}^{\infty}\Phi_{\eta}^{n}(r)\leq bI_{H}.\label{eq:Phi_eta_Invert}\end{equation}

\item There is an invertible pure superharmonic operator for $\Phi_{\eta}$.
\end{enumerate}
\end{theorem}

\begin{proof}The equivalence of $1.$ and 3. is an easy calculation
of the sort that we performed above. If $r\cdot\eta\cdot r^{-1}=\zeta$,
then $r\cdot\eta_{n}=\zeta_{n}\cdot r$ for all $n\geq0$ and we further
have \begin{equation}
\Phi_{\eta}^{n}(r^{*}ar)=\Phi_{r\cdot\eta_{n}}(a)=\Phi_{\zeta_{n}\cdot r}(a)=r^{*}\Phi_{\zeta_{n}}(a)r=r^{*}\Phi_{\zeta}^{n}(a)r,\label{eq:SimilarityPhi}\end{equation}
for all $a\in\sigma(M)'$. Now suppose that $1.$ holds, then with
$a=1$, we see that $r^{*}r$ is an invertible superharmonic operator
for $\Phi_{\eta}$ because $\Phi_{\eta}(r^{*}r)=r^{*}\Phi_{\zeta}(I)r\leq r^{*}r$,
since $\Vert\zeta\Vert\leq1$. On the other hand, because $\zeta$
is a $C_{\cdot0}$ element of $E^{\sigma}$, $r^{*}r$ is a pure superharmonic
operator: $\Phi_{\eta}^{n}(r^{*}r)=r^{*}\Phi_{\zeta}^{n}(I)r\to0$
weakly and, therefore, strongly. Thus, 3. is satisfied. The argument
is essentially reversible: Suppose 3. holds and let $a$ be a positive
invertible superharmonic operator for $\Phi_{\eta}$. If we let $r$
be the positive square root of $a$, then $r$ is invertible and we
may let $\zeta=r\cdot\eta\cdot r^{-1}$. Since $a$ is superharmonic,
we conclude that $r^{2}\geq\Phi_{\eta}(r^{2})=r\Phi_{\zeta}(I)r$,
which shows that $\Phi_{\zeta}(I)\leq I$, because $r$ is positive
and invertible, and this implies that $\zeta\in\overline{\mathbb{D}(E^{\sigma})}$.
On the other hand, we conclude from these calculations and the assumption
that $a=r^{2}$ is a pure superharmonic operator for $\Phi_{\eta}$,
that $\Phi_{\zeta}^{n}(I)=r^{-1}\Phi_{\eta}^{n}(a)r\to0$ weakly as
$n\to\infty$. Thus $\zeta$ is a $C_{\cdot0}$ element of $\overline{\mathbb{D}(E^{\sigma})}$,
as was required.

Suppose assertion $2.$ is satisfied and let $r$, $a$ and $b$ be
as in equation \eqref{eq:Phi_eta_Invert}. Then the series $\sum_{n=0}^{\infty}\Phi_{\eta}^{n}(r)$
converges strongly to an operator $R$ that is invertible in $\sigma(M)'$.
Now $\Phi_{\eta}(R)=\sum_{n=1}^{\infty}\Phi_{\eta}^{n}(r)=R-r\leq R$.
Thus $R$ is superharmonic for $\Phi_{\eta}$. But also $\Phi_{\eta}^{n}(R)=\sum_{k=n}^{\infty}\Phi_{\eta}^{k}(r)$
and this sequence of operators converges strongly to zero, since the
series $\sum_{n=0}^{\infty}\Phi_{\eta}^{n}(r)$ converges strongly.
Thus condition 3. is satisfied.

Suppose condition 3. is satisfied, let $R$ be an invertible pure
superharmonic operator for $\Phi_{\eta}$ and set $r:=R-\Phi_{\eta}(R)$.
Then $r$ is positive semidefinite and $R=\sum_{n=0}^{\infty}\Phi_{\eta}^{n}(r)$.
Since $R$ is assumed invertible, the inequality \eqref{eq:Phi_eta_Invert}
is satisfied for suitable $a$ and $b$. \end{proof}

The next theorem uses intertwiners to describe when an $\eta\in E^{\sigma}$
is similar to a $\zeta$ in the \emph{open} unit disc $\mathbb{D}(E^{\sigma})^{*}.$
It is similar in spirit to \cite[Theorem 5.9]{gP03}, but arguments
use different technology. Recall that $(S_{0},\sigma_{0})$ denotes
the universal isometric induced representation. Also, we let $P_{0}$
be the orthogonal projection of $\mathcal{F}(E)\otimes_{\pi}K_{0}$
onto the zero$^{\underline{th}}$ summand, $M\otimes_{\pi}K_{0}\simeq K_{0}$.

\begin{theorem}\label{Theorem:Open_Disc}Let $\sigma$ be a normal
representation of $M$ on a Hilbert space $H$ and let $\eta$ be
a point in $E^{\sigma}$. Then $\eta$ is similar to a point in the
open unit disc $\mathbb{D}(E^{\sigma})$ if and only if there is a
$C\in\mathcal{I}((S_{0},\sigma_{0}),(\eta^{*},\sigma))$ such that
$CP_{0}C^{*}$ is invertible.

\end{theorem}\begin{proof}Suppose there is an invertible $r\in\sigma(M)'$
such that $r\cdot\eta\cdot r^{-1}=\zeta$, with $\Vert\zeta\Vert<1$.
Then $r\cdot\eta_{n}=\zeta_{n}\cdot r$ for all $n$ and we see that
$\Phi_{\eta}^{n}(r^{*}r)=r^{*}\Phi_{\zeta}^{n}(I)r\leq r^{*}r\Vert\zeta^{*}\zeta\Vert^{n}$,
which shows both that $r^{*}r$ is a pure superharmonic operator for
$\Phi_{\eta}$. So by Theorem \ref{Theorem:PureSuperHarm} there is
a $C\in\mathcal{I}((S_{0},\sigma_{0}),(\eta^{*},\sigma))$ such that
$r^{*}r=CC^{*}$. But also,

\begin{multline*}
CP_{0}C^{*}=C(I-\widetilde{S_{0}}\widetilde{S_{0}^{*})}C^{*}=CC^{*}-\eta^{*}CC^{*}\eta\\
=r^{*}r-\Phi_{\eta}(r^{*}r)=r^{*}r-r^{*}\Phi_{\zeta}(I)r\geq(1-\Vert\zeta\Vert^{2})r^{*}r.\end{multline*}
Since $r$ is invertible, so is $CP_{0}C^{*}$.

Conversely, suppose that there is a $C\in\mathcal{I}((S_{0},\sigma_{0}),(\eta^{*},\sigma))$
such that $CP_{0}C^{*}$ is invertible and let $b\in\mathbb{R}$ satisfy
the inequality $CP_{0}C^{*}\geq bI>0$. Also let $t$ be a positive
number less than $\frac{b}{\Vert CC^{*}\Vert}$ ($<1$). Our objective
is to show that if $r=(CC^{*})^{\frac{1}{2}}$ and if $\zeta=r\cdot\eta\cdot r^{-1}$,
then $\Vert\zeta\Vert^{2}\leq1-t$. Recall that $CP_{0}C^{*}=CC^{*}-\Phi_{\eta}(CC^{*})$.
By definition of $t$, $CC^{*}\leq\Vert CC^{*}\Vert I\leq\frac{b}{t}I$.
Therefore $CC^{*}-bI\leq CC^{*}-tCC^{*}$. But then $(1-t)CC^{*}-\Phi_{\eta}(CC^{*})\geq[CC^{*}-bI]-\Phi_{\eta}(CC^{*})>0$.
Consequently, $\Phi_{\eta}(CC^{*})\leq(1-t)CC^{*}$. Now $CC^{*}=r^{2}$
and $\zeta$ is defined to be $r\cdot\eta\cdot r^{-1}$. We have $\Vert\zeta\Vert^{2}=\Vert\Phi_{\zeta}(I)\Vert=\Vert r^{-1}\Phi_{\eta}(r^{2})r^{-1}\|\leq\Vert r^{-1}((1-t)r^{2})r^{-1}\Vert=1-t$.
Thus, $2.$ implies $1$. \end{proof}

Our final theorem in this vein has no analogue in \cite{gP03}, but
it is in the spirit of that paper. The proof rests on the main results
proved to this point.

\begin{theorem}\label{Theorem:Similar_to_AC_contracton}Let $\sigma:M\to B(H)$
be a normal representation of $M$ on the Hilbert space $H$, and
let $\eta\in E^{\sigma}$ . Then the following assertions are equivalent.
\begin{enumerate}
\item $\eta$ is similar to an absolutely continuous $\zeta\in\overline{\mathbb{D}(E^{\sigma})}$.
\item $\Phi_{\eta}$ admits an invertible superharmonic operator and \[
H=\bigcup\{Ran(C)\mid C\in\mathcal{I}((S_{0},\sigma_{0}),(\eta^{*},\sigma))\}.\]

\item $\Phi_{\eta}$ admits an invertible superharmonic operator and \[
H=\bigvee\{Ran(Q)\mid\ensuremath{Q\in\sigma(M)'},\, Q-\mbox{pure superharmonic for }\Phi_{\eta}\}.\]

\end{enumerate}
\end{theorem}\begin{proof}Because of the hypotheses in 2. and 3.
that $\Phi_{\eta}$ admits an invertible superharmonic function, we
know from Theorem \ref{Theorem:Similar_to_a_contraction} that $\eta$
is similar to a contraction in each of the situations. The point of
2. is that if $\eta^{*}$ is similar to a point $\zeta^{*}\in\overline{\mathbb{D}(E^{\sigma})^{*}}$
then there is an invertible $r\in\sigma(M)'$ such that $r(\eta^{*}\times\sigma)r^{-1}=(\zeta^{*}\times\sigma)$,
and so a $C$ satisfies $C(S_{0}\times\sigma_{0})=(\eta^{*}\times\sigma)C$
if and only if $rC$ satisfies the equation $rC(S_{0}\times\sigma_{0})=r(\eta^{*}\times\sigma)r^{-1}(rC)=(\zeta^{*}\times\sigma)(rC$),
i.e., if and only $rC$ lies in $\mathcal{I}((S_{0},\sigma_{0}),(\zeta^{*},\sigma))$.
Thus if $\eta$ and $\zeta$ are similar, the spaces $\bigcup\{Ran(C)\mid C\in\mathcal{I}((S_{0},\sigma_{0}),(\eta^{*},\sigma))\}$
and $\bigcup\{Ran(C)\mid C\in\mathcal{I}((S_{0},\sigma_{0}),(\zeta^{*},\sigma))\}$
are identical. Similarly, if $\eta$ and $\zeta$ are similar, then
the spaces $\bigvee\{Ran(Q)\mid\ensuremath{Q\in\sigma(M)'},\, Q-\mbox{pure superharmonic for }\Phi_{\eta}\}$
and $\bigvee\{Ran(Q)\mid\ensuremath{Q\in\sigma(M)'},\, Q-\mbox{pure superharmonic for }\Phi_{\zeta}\}$
are identical. Thus the theorem is an immediate consequence of Theorem
\ref{Theorem:AbsContT}. \end{proof}

\section{Induced Representations and their Ranges\label{sec:Induced-Representations}}

In a sense, this section is an interlude that develops some ideas
that will be used in the next section on the structure theorem. However,
we believe the results in it are of sufficient interest in themselves
that we want to develop them separately.

Throughout this section, $\tau$ will be a normal representation of
our $W^{*}$-algebra $M$ on a Hilbert space $H$ and $\tau^{\mathcal{F}(E)}$
will be the induced representation of $\mathcal{L}(\mathcal{F}(E))$
acting on the Hilbert space $\mathcal{F}(E)\otimes_{\tau}H$. The
support projection of $\tau$ will be denoted $e$. This is a central
projection in $M$ and $e^{\perp}$ is the projection onto the kernel
of $\tau$, $\ker(\tau)$. The problem we want to address is this.

\begin{problem}\label{problem:Ultra-weak-closure}Determine when
the image of $H^{\infty}(E)$ under $\tau^{\mathcal{F}(E)}$ is ultra-weakly
closed.\emph{ }\end{problem}

Of course, $\tau^{\mathcal{F}(E)}$ is a normal representation of
$\mathcal{L}(\mathcal{F}(E))$ and so the image of $\mathcal{L}(\mathcal{F}(E))$
in $B(\mathcal{F}(E)\otimes_{\tau}H)$ is ultra-weakly closed, since
$\mathcal{L}(\mathcal{F}(E))$ is a $W^{*}$-algebra. Also, of course,
if $\tau$ is injective, then so is $\tau^{\mathcal{F}(E)}$ and,
consequently, $\tau^{\mathcal{F}(E)}$ is isometric and an ultra-weak
homeomorphism. In this event, $\tau^{\mathcal{F}(E)}(H^{\infty}(E))$
is an ultra-weakly closed subalgebra of $B(\mathcal{F}(E)\otimes_{\tau}H)$.
In particular, if $M$ is a factor, then $\tau^{\mathcal{F}(E)}(H^{\infty}(E))$
is ultra-weakly closed. The problem, then, is to determine what happens
when the kernel of $\tau$, $e^{\perp}M$, is non-trivial. In this
case, the projection onto the kernel of $\tau^{\mathcal{F}(E)}$ is
$I_{\mathcal{F}(E)}\otimes e^{\perp}$ and the problem is to see how
it interacts with $\tau^{\mathcal{F}(E)}(H^{\infty}(E))$. We have
no examples of representations $\tau$ where the image $\tau^{\mathcal{F}(E)}(H^{\infty}(E))$
fails to be ultra-weakly closed, but we are able to provide useful,
very general conditions on $e$ that guarantee that $\tau^{\mathcal{F}(E)}(H^{\infty}(E))$
is ultra-weakly closed.

We adopt the following terminology, which is suggested by \cite{FMR}.

\begin{definition}\label{reducing} A projection $e$ in the center
of $M$, $\mathfrak{Z}(M)$, that satisfies $\xi e=\varphi(e)\xi e$
for all $\xi\in E$ will be called an $E$-\emph{saturated} projection.
If $e$ also satisfies $\xi e=\varphi(e)\xi$ for all $\xi\in E$,
$e$ will be called an $E$-\emph{reducing} projection. \end{definition}

\begin{example}\label{endom} If $\alpha$ is an endomorphism of
$M$ and if $E$ is the correspondence $_{\alpha}M$, then a central
projection $e\in M$ is $E$-saturated if and only if $\alpha(e)ae=ae$
for all $a\in M$. That is, $e$ is $E$-saturated if and only if
$e\leq\alpha(e)$. Moreover, $e$ will be $E$-reducing if and only
if $e$ is fixed by $\alpha$, $e=\alpha(e)$. \end{example}

The meaning the $E$-saturation condition for the present discussion
may be further clarified by the following two lemmas.

\begin{lemma} \label{lem:Meaning_of_Saturated_1} A projection $e$
in the center of $M$, $\mathfrak{Z}(M)$, is an $E$-saturated projection
in $M$ if and only if $\varphi_{\infty}(e)$ is an invariant projection
for $H^{\infty}(E)$ in the sense that \begin{equation}
H^{\infty}(E)\varphi_{\infty}(e)=\varphi_{\infty}(e)H^{\infty}(E)\varphi_{\infty}(e).\label{eq:Invariant_e}\end{equation}
 \end{lemma}Thus, if $e$ is $E$-saturated, then in any completely
contractive representation $\rho$ of $H^{\infty}(E)$, the range
of $\rho(\varphi_{\infty}(e))$ is an invariant invariant subspace
$\rho(H^{\infty}(E))$.

\begin{proof}First, recall that for $a,b\in M$ and $\xi\in E$,
$T_{\varphi(a)\xi b}=\varphi_{\infty}(a)T_{\xi}\varphi_{\infty}(b)$.
Consequently, if $e$ is $E$ saturated, so that by definition $\xi e=\varphi(e)\xi e$
for all $\xi\in E$, it follows that for all $\xi\in E$, $T_{\xi}\varphi_{\infty}(e)=T_{\xi e}\varphi_{\infty}(e)=T_{\varphi(e)\xi e}\varphi_{\infty}(e)=\varphi_{\infty}(e)T_{\xi}\varphi_{\infty}(e)$.
Since $\varphi_{\infty}(e)$ obviously commutes with $\varphi_{\infty}(M)$,
equation \ref{eq:Invariant_e} is verified. For the converse assertion,
simply write out the matrices for $T_{\xi},$ $\xi\in E$, and $\varphi_{\infty}(e)$
with respect to the direct sum decomposition of $\mathcal{F}(E)=\sum_{n\geq0}E^{\otimes n}$
and compute what it means for the equation $T_{\xi}\varphi_{\infty}(e)=\varphi_{\infty}(e)T_{\xi}\varphi_{\infty}(e)$
to hold.\end{proof}

Note that the same argument shows $e$ is $E$-reducing if and only
if $\varphi_{\infty}(e)$ commutes with $H^{\infty}(E)$. In fact,
as we shall see in a moment, if $e$ is $E$-reducing, then $\varphi_{\infty}(e)$
lies in the center of $\mathcal{L}(\mathcal{F}(E))$.

\begin{lemma}\label{lem:Meaning_of_Saturation_2}If $e$ is $E$-saturated,
then the space $Ee$ becomes a $W^{*}$-correspon-\linebreak dence
over $Me$. Moreover, the left action of $M$ on $E$ restricts to
a unital left action of $Me$ on $Ee$, and if we define $\pi:Me\to\mathcal{L}(\mathcal{F}(E))$
and $V:Ee\to\mathcal{L}(\mathcal{F}(E))$ by the formulae\[
\pi(me):=\varphi_{\infty}(me)\]
 and \[
V(\xi e):=T_{\xi e},\]
 then the pair $(V,\pi)$ is an ultra-weakly continuous, isometric,
covariant representation of $(Ee,Me)$ in $\mathcal{L}(\mathcal{F}(E))$,
whose image is contained in $H^{\infty}(E)$. \end{lemma}

\begin{proof}The calculation,\begin{multline*}
V(\xi e)^{*}V(\eta e)=(T_{\xi}\varphi_{\infty}(e))^{*}T_{\eta}\varphi_{\infty}(e)=\varphi_{\infty}(e)T_{\xi}^{*}T_{\eta}\varphi_{\infty}(e)\\
=\varphi_{\infty}(e)\varphi_{\infty}(\langle\xi,\eta\rangle)\varphi_{\infty}(e)=\pi(\langle\xi e,\eta e\rangle_{Me}),\end{multline*}
shows that $V$ is isometric. The bimodule property is immediate.
The ultra-weak continuity is an immediate consequence of \cite[Lemma 2.5, Remark 2.6]{MSHardy}.
\end{proof}

\begin{remark}\label{rem:Inducing tool.}Strictly speaking, of course,
$(V,\pi)$ is an isometric representation of $(Ee,Me)$ into the abstract
$W^{*}$-algebra, $\mathcal{L}(\mathcal{F}(E))$, so to apply the
theory from \cite{MSHardy} here and elsewhere, one should compose
$(V,\pi)$ with a faithful normal representation of $\mathcal{L}(\mathcal{F}(E))$
on Hilbert space. The details are easy and may safely be omitted.
Later, however, it will prove useful to use that device. Anticipating
results to be proved shortly (Lemma \ref{L1}), we call $(V,\pi)$
or $V\times\pi$ the \emph{canonical embedding} of $\mathcal{T}(Ee)$
in $\mathcal{T}(E)$. We will see that $V\times\pi$ is faithful on
the Toeplitz algebra, $\mathcal{T}(Ee)$, and extends to a completely
isometric, ultra-weakly continuous representation of $H^{\infty}(Ee)$,
mapping it into $H^{\infty}(E)$. \end{remark}

The following lemma may be known, but we do not have a reference.
It will be helpful to have the details in hand.

\begin{lemma}\label{lem:Central_adjointable_ops}Let $F$ be a $C^{*}$-Hilbert
module over a $C^{*}$-algebra $N$. For $a\in N$, define $R_{a}:F\to F$
by the formula $R_{a}\xi=\xi a$, $\xi\in F$. Then $R_{a}$ is a
bounded $\mathbb{C}$-linear operator on $F$ with norm at most $\Vert a\Vert.$
If $a$ lies in the center of $N$, $\mathfrak{Z}(N)$, then $R_{a}$
is a bounded adjointable operator on $F$ that lies in the center
of $\mathcal{L}(F)$.\end{lemma}\begin{proof}For $\xi\in F$ and
$a\in N$, we have

\begin{align*}
\langle R_{a}\xi,R_{a}\xi\rangle & =\langle\xi a,\xi a\rangle\\
 & =a^{*}\langle\xi,\xi\rangle a\leq a^{*}\Vert\xi\Vert^{2}a\\
 & \leq\Vert a\Vert^{2}\Vert\xi\Vert^{2},\end{align*}
which shows that $R_{a}$ is a continuous $\mathbb{C}$-linear operator
with norm bounded by $\Vert a\Vert.$ To see that $R_{a}\in\mathcal{L}(F)$
when $a\in\mathfrak{Z}(N)$, simply observe that for $\xi$ and $\eta$
in $F$, \[
\langle R_{a}\xi,\eta\rangle=\langle\xi a,\eta\rangle=a^{*}\langle\xi,\eta\rangle=\langle\xi,\eta\rangle a^{*}=\langle\xi,\eta a^{*}\rangle=\langle\xi,R_{a^{*}}\eta\rangle.\]
 This shows that $R_{a}$ is adjointable, with adjoint $R_{a^{*}}$
and this, in turn, shows that $R_{a}$ is $N$-linear. Thus, $R_{a}\in\mathcal{L}(F)$.
(Of course, the fact that $a$ lies in $\mathfrak{Z}(N)$ also implies
directly that $R_{a}$ is $N$-linear.) However, since elements of
$\mathcal{L}(F)$ are $N$-module maps, i.e., $T(\xi b)=(T\xi)b$,
it is immediate that $R_{a}$ lies in the center of $\mathcal{L}(F)$.\end{proof}

Among other things, the following lemma solves Problem \ref{problem:Ultra-weak-closure}
under the hypothesis that the support projection of the representation
$\tau$ is $E$-reducing.

\begin{lemma}\label{ind} Let $\tau$ be a normal representation
of the $W^{*}$-algebra $M$ on a Hilbert space $K$ and let $e$
be its support projection. Let $q$ be the smallest projection in
$M$ such that $\varphi_{\infty}(q)R_{e}=R_{e}$. Then the following
assertions hold:
\begin{enumerate}
\item [(1)] $\ker(\tau^{\mathcal{F}(E)})=\{R\in\mathcal{L}(\mathcal{F}(E))\mid RR_{e}=0\}$.
\item [(2)]The ultra-weakly closed ideal \begin{eqnarray*}
H^{\infty}(E)\cap\ker(\tau^{\mathcal{F}(E)}) & = & \{R\in H^{\infty}(E):R\varphi_{\infty}(q)=0\}\\
 & = & \overline{H^{\infty}(E)\varphi_{\infty}(q^{\perp})H^{\infty}(E)}^{u-w},\end{eqnarray*}
 in $H^{\infty}(E)$ is generated by $\varphi_{\infty}(q^{\perp})$.
\item [(3)] The subspace $\varphi_{\infty}(q)\mathcal{F}(E)$ of $\mathcal{F}(E)$
is invariant for $\tau^{\mathcal{F}(E)}(H^{\infty}(E))$ and the map
$X\to\tau^{\mathcal{F}(E)}(X)\vert\varphi_{\infty}(q)\mathcal{F}(E)$
is an injective completely contractive representation of $H^{\infty}(E)$.
\item [(4)] If $e$ is $E$-saturated then $e=q$, so that $e$ is $\mathcal{F}(E)$-saturated,
i.e., $\varphi_{\infty}(e)\mathcal{F}(E)e=\mathcal{F}(E)e$.
\item [(5)] If $e$ is $E$-reducing, the three projections $R_{e}$,
$\varphi_{\infty}(e)$, and $\varphi_{\infty}(q)$ coincide and lie
in the center of $\mathcal{L}(\mathcal{F}(E))$. Consequently, $\tau^{\mathcal{F}(E)}(H^{\infty}(E))$
is ultra-weakly closed and the restriction of $\tau^{\mathcal{F}(E)}$
to $H^{\infty}(E)\varphi_{\infty}(e)$ is completely isometric.
\end{enumerate}
\end{lemma} \begin{proof} For $R\in\mathcal{L}(\mathcal{F}(E))$,
$R\otimes_{\tau}I_{K}=0$ if and only if for every $k\in K$ and $\eta\in\mathcal{F}(E)$,
$0=\langle R\eta\otimes k,R\eta\otimes k\rangle=\langle k,\tau(\langle R\eta,R\eta\rangle)k\rangle$,
that is, if and only if $\langle R\eta,R\eta\rangle\in Me^{\perp}$.
This happens if and only if $R=RR_{e^{\perp}}$. It follows that \[
\ker(\tau^{\mathcal{F}(E)})=\{R\in\mathcal{L}(\mathcal{F}(E))\mid RR_{e}=0\}\]
 and \begin{equation}
H^{\infty}(E)\cap\ker(\tau^{\mathcal{F}(E)})=\{R\in H^{\infty}(E)\mid RR_{e}=0\}.\label{ker}\end{equation}

Now choose a faithful normal representation, $\sigma,$ of $M$ on
a Hilbert space $H$. Then $\sigma^{\mathcal{F}(E)}$ is a $^{*}$-isomorphism
of $\mathcal{L}(\mathcal{F}(E))$ onto $\mathcal{L}(\mathcal{F}(E))\otimes I_{H}$
(and is therefore also a homeomorphism with respect to the ultra-weak
topologies.) Also note that $I_{\mathcal{F}(E)}\otimes\sigma(e)=R_{e}\otimes I_{H}=\sigma^{\mathcal{F}(E)}(R_{e})$
by Lemma \ref{lem:Central_adjointable_ops}. Set \begin{align*}
g & =\bigvee\{u(\sigma^{\mathcal{F}(E)}(R_{e}))u^{*}\mid u\in\sigma^{\mathcal{F}(E)}(\varphi_{\infty}(M))',\; u\emph{ {\rm is unitary}}\;\}.\end{align*}
 The range of the projection $u(\sigma^{\mathcal{F}(E)}(R_{e}))u^{*}$
is $u(R_{e}\mathcal{F}(E)\otimes_{\sigma}H)$ and, thus, $\varphi_{\infty}(q^{\perp})\otimes I_{H}=\sigma^{\mathcal{F}(E)}(\varphi_{\infty}(q^{\perp}))$
vanishes on it. It follows that $g\leq\sigma^{\mathcal{F}(E)}(\varphi_{\infty}(q))$.
By construction $g$ is a projection in the center of $\sigma^{\mathcal{F}(E)}(\varphi_{\infty}(M))$,
so we can write $g=\sigma^{\mathcal{F}(E)}(\varphi_{\infty}(z))$
for some projection $z\in\mathfrak{Z}(M)$. But then $\varphi_{\infty}(q-z)$
vanishes on $R_{e}\mathcal{F}(E)$, which implies that $q=z$, since
$q$ is the smallest projection in $M$ with this property. Consequently,
\begin{equation}
\sigma^{\mathcal{F}(E)}(\varphi_{\infty}(q))=\bigvee\{u(\sigma^{\mathcal{F}(E)}(R_{e}))u^{*}\mid u\in\sigma^{\mathcal{F}(E)}(\varphi_{\infty}(M))',\; u\mbox{ is unitary}\},\label{sup}\end{equation}
and $q\in\mathfrak{Z}(M)$.

Next we want to show that if $R\in H^{\infty}(E)$ and if $RR_{e}=0$,
then $R\varphi_{\infty}(q)=0.$ To this end, we use the gauge automorphism
group and the notation developed in paragraph \ref{par:GaugeAutos}.
Observe that $W_{t}$ commutes with $R_{e}$ for all $t\in\mathbb{T}$,
since $W_{t}\in\mathcal{L}(\mathcal{F}(E))$ and $R_{e}\in\mathfrak{Z}(\mathcal{L}(\mathcal{F}(E)))$.
Consequently, $\gamma_{t}(R)R_{e}=W_{t}RR_{e}W_{t}^{*}=0$, and so
$\Phi_{k}(R)R_{e}=0$ for all $k$, by equation (\ref{FourierOperators}).
Since $\Phi_{k}(R)^{*}\Phi_{k}(R)\in\varphi_{\infty}(M)$, equation
\eqref{sup} implies that $\Phi_{k}(R)\varphi_{\infty}(q)=0$ for
every $k$. Consequently each of the Cesaro sums $\Sigma_{n}(R):=\Sigma_{0\leq j<n}(1-\frac{j}{n})\Phi_{j}(R)$
satisfies the equation $\Sigma_{n}(R)\varphi_{\infty}(q)=0$, $n\geq0$.
Since $R$ is the ultra-weak limit of the Cesaro sums, $\sum_{n}(R)$,
we conclude that $R\varphi_{\infty}(q)=0$, as we wanted to show.

This argument shows, too, that the map sending $R\varphi_{\infty}(q)$
to $RR_{e}$, for $R\in H^{\infty}(E)$, is injective. It is also
completely contractive and ultra-weakly continuous.

Further, it is clear from equation (\ref{ker}) that \[
H^{\infty}(E)\cap\ker(\tau^{\mathcal{F}(E)})=\{R\in H^{\infty}(E):R\varphi_{\infty}(q)=0\}.\]
 Thus $\varphi_{\infty}(q^{\perp})$ lies in the kernel of $\tau^{\mathcal{F}(E)}$,
and the ultra-weakly closed ideal in $H^{\infty}(E)$ generated by
$\varphi_{\infty}(q^{\perp})$ lies in the kernel as well. On the
other hand, every $R$ in the kernel satisfies $R=R\varphi_{\infty}(q^{\perp})$
and, thus, is contained in the ideal generated by $\varphi_{\infty}(q^{\perp})$.
It follows that $\varphi_{\infty}(q^{\perp})H^{\infty}(E)\subseteq\ker(\tau^{\mathcal{F}(E)})\cap H^{\infty}(E)=\{R\in H^{\infty}(E):R\varphi_{\infty}(q)=0\}$
and, consequently, \[
\varphi_{\infty}(q^{\perp})H^{\infty}(E)\varphi_{\infty}(q)=\{0\},\]
that $e$ is $E$-saturated and is proved in i.e., $\varphi_{\infty}(q)$
is an invariant projection for $H^{\infty}(E)$.

The restriction of $\tau^{\mathcal{F}(E)}$ to $R_{e}\mathcal{L}(\mathcal{F}(E))R_{e}=\mathcal{L}(\mathcal{F}(E))R_{e}$
is an isomorphism of this von Neumann algebra onto the image of the
induced representation. It is, therefore, completely isometric. The
restriction of the induced representation to $\varphi_{\infty}(q)H^{\infty}(E)\varphi_{\infty}(q)$
is a composition of two completely contractive injective maps and
is, therefore, completely contractive and injective.

For assertion (5), note that, if $e$ is $E$-saturated then, for
every $n\leq1$ and every $\xi\in E^{\otimes n}$, $\xi e=\varphi_{n}(e)\xi e$.
Thus $\mathcal{F}(E)e=\varphi_{\infty}(e)\mathcal{F}(E)e$ and so
$q=e$.

For assertion (6), assume that $e$ is $E$-reducing. Then it is easy
to verify that $R_{e}=\varphi_{\infty}(e)$ and, in particular, $R_{e}$
lies in the center of $H^{\infty}(E)$. The map $\tau^{\mathcal{F}(E)}$,
restricted to the von Neumann algebra $\mathcal{L}(\mathcal{F}(E))R_{e}$,
is an injective ultra-weakly continuous representation and, thus,
is a homeomorphism with respect to the ultra-weak topologies onto
its image. Consequently, its restriction to $H^{\infty}(E)\varphi_{\infty}(e)$
is completely isometric and has a ultra-weakly closed image which
is $\tau^{\mathcal{F}(E)}(H^{\infty}(E))$. This proves (6). \end{proof}

The following Theorem solves Problem \ref{problem:Ultra-weak-closure}
under the hypothesis that the support projection is saturated.

\begin{theorem}\label{indsaturated} If the support projection $e$
of a normal representation $\tau$ of $M$ is $E$-saturated, then
$\tau^{\mathcal{F}(E)}(H^{\infty}(E))$ is ultra-weakly closed and
the restriction of $\tau^{\mathcal{F}(E)}$ to $H^{\infty}(E)\varphi_{\infty}(e)$
is completely isometric. \end{theorem}

We require two lemmas.

\begin{lemma}\label{L1} If $e$ is an $E$-saturated central projection
in $M$, then the canonical embedding $V\times\pi$ of $\mathcal{T}(Ee)$
into $\mathcal{T}(E)$, defined in Lemma \ref{lem:Meaning_of_Saturation_2},
is faithful. \end{lemma} \begin{proof}To show $V\times\pi$ is faithful,
it suffices to prove that if $\sigma_{0}$ is a faithful normal representation
of $M$ on the Hilbert space $H$ and if $\pi_{0}=\sigma_{0}^{\mathcal{F}(E)}\circ\pi$
and if $V_{0}=\sigma_{0}^{\mathcal{F}(E)}\circ V$, then $V_{0}\times\pi_{0}$
is a faithful representation of $\mathcal{T}(Ee)$ on $\mathcal{F}(E)\otimes_{\sigma_{0}}H$.
This is because $\sigma_{0}^{\mathcal{F}(E)}$ is a faithful normal
representation of $\mathcal{L}(\mathcal{F}(E))$. We use \cite[Theorem 2.1]{FR},
which asserts in our setting, that $V_{0}\times\pi_{0}$ will be faithful
if the representation of $Me$ obtained by restricting $\pi_{0}(Me)$
to ($\mathcal{F}(E)\otimes_{\sigma_{0}}H)\ominus(V_{0}(Ee)(\mathcal{F}(E)\otimes_{\sigma_{0}}H))$
is faithful. For this, observe that because $e$ is $E$-saturated\begin{align*}
V_{0}(Ee)(\mathcal{F}(E)\otimes_{\sigma_{0}}H) & =\sigma_{0}^{\mathcal{F}(E)}(V(Ee))(\mathcal{F}(E)\otimes_{\sigma_{0}}H)\\
= & Ee\otimes\mathcal{F}(E)\otimes_{\sigma_{0}}H=E\otimes\varphi_{\infty}(e)\mathcal{F}(E)\otimes_{\sigma_{0}}H\\
= & E\otimes\mathcal{F}(E)e\otimes_{\sigma_{0}}H=E\otimes\mathcal{F}(E)\otimes\sigma_{0}(e)H,\end{align*}
by point (4) of Lemma \ref{ind}. Since $E\otimes\mathcal{F}(E)=\sum_{n\geq1}E^{\otimes n}$,
we see that ($\mathcal{F}(E)\otimes_{\sigma_{0}}H)\ominus(V_{0}(Ee)(\mathcal{F}(E)\otimes_{\sigma_{0}}H))$
contains $M\otimes_{\sigma_{0}}\sigma_{0}(e)H$ as a summand, to which
$\pi_{0}(Me)$ restricts. But that restriction is obviously faithful
on $Me$ since it is just left multiplication on $M\otimes_{\sigma_{0}}\sigma_{0}(e)H=Me\otimes_{\sigma_{0}}H$
by elements of the form $me\otimes I_{H}$, $m\in M$, and $\sigma_{0}$
is faithful. \end{proof}

\begin{lemma}\label{L2} Suppose $\tau$ is a normal representation
of $M$, suppose its support projection $e$ is $E$-saturated, and
let $\tau_{e}$ be the restriction of $\tau$ to $Me$. Then the map
that sends $\xi e\otimes_{\tau_{e}}k$ in $\mathcal{F}(Ee)\otimes_{\tau_{e}}K$
to $\xi\otimes_{\tau}k$ in $\mathcal{F}(E)\otimes_{\tau}K$ extends
to a well-defined Hilbert space isomorphism $U$ mapping $\mathcal{F}(Ee)\otimes_{\tau_{e}}K$
onto $\mathcal{F}(E)\otimes_{\tau}K$ such that, for every $X\in\mathcal{T}_{+}(Ee)$,
\[
\tau_{e}^{\mathcal{F}(Ee)}(X)=U^{*}\tau^{\mathcal{F}(E)}((V\times\pi)(X))U,\]
where $(V,\pi)$ is the canonical embedding of $(Ee,Me)$ in $\mathcal{T}(E)$.
The image of $\mathcal{T}_{+}(Ee)$ under $V\times\pi$ is $\mathcal{T}_{+}(E)\varphi_{\infty}(e)=\varphi_{\infty}(e)\mathcal{T}_{+}(E)\varphi_{\infty}(e)$.
\end{lemma} \begin{proof} The fact that $U$ is well-defined and
isometric is an easy calculation. The fact that $U$ is surjective
is immediate. Observe that $\tau_{e}^{\mathcal{F}(Ee)}(T_{\theta e})(\xi e\otimes_{\tau_{e}}k)=\theta e\otimes\xi e\otimes_{\tau_{e}}k$.
By definition of $U$, this last expression is \begin{align*}
U^{*}(\theta e\otimes(\xi\otimes_{\tau}k)) & =U^{*}(\theta\otimes\varphi_{\infty}(e)\xi\otimes_{\tau}k)=U^{*}\tau^{\mathcal{F}(E)}(T_{\theta})(\varphi_{\infty}(e)\xi\otimes_{\tau}k)\\
= & U^{*}\tau^{\mathcal{F}(E)}(T_{\theta}\varphi_{\infty}(e))U(\xi e\otimes_{\tau_{e}}k)\\
= & U^{*}\tau^{\mathcal{F}(E)}((V\times\pi)(T_{\theta e}))U(\xi e\otimes_{\tau_{e}}k).\end{align*}
Similarly, $\tau_{e}^{\mathcal{F}(Ee)}(\varphi_{\infty}^{e}(me))=U^{*}\tau^{\mathcal{F}(E)}((V\times\pi)(\varphi_{\infty}(me))U$
for all $me\in Me$, where $\varphi_{\infty}^{e}$ denotes the action
of $Me$ on $\mathcal{F}(Ee)$. \end{proof}\bigskip{}

\textbf{\flushleft Proof of Theorem~\ref{indsaturated}.} Since
$\tau_{e}$ is a faithful representation of $Me$, $\tau_{e}^{\mathcal{F}(E)}$
is a completely isometric map of $H^{\infty}(Ee)$ and its image,
$\tau_{e}^{\mathcal{F}(E)}(H^{\infty}(Ee))$, is ultra-weakly closed.
Since both $\tau^{\mathcal{F}(E)}$ and $\tau_{e}^{\mathcal{F}(E)}$
are ultra-weakly continuous, we have \[
\overline{\tau^{\mathcal{F}(E)}(\mathcal{T}_{+}(E))}^{u-w}=\overline{\tau^{\mathcal{F}(E)}(H^{\infty}(E))}^{u-w}\]
 and \[
\overline{\tau_{e}^{\mathcal{F}(Ee)}(\mathcal{T}_{+}(Ee))}^{u-w}=\overline{\tau_{e}^{\mathcal{F}(Ee)}(H^{\infty}(Ee))}^{u-w}=\tau_{e}^{\mathcal{F}(Ee)}(H^{\infty}(Ee)).\]
 Now $\overline{\tau^{\mathcal{F}(E)}(H^{\infty}(E))}^{u-w}=\overline{\tau^{\mathcal{F}(E)}(H^{\infty}(E)\varphi_{\infty}(e))}^{u-w}$
since $\varphi_{\infty}(e^{\perp})$ is the projection onto the kernel
of $\tau^{\mathcal{F}(E)}$. Also, Lemma \ref{L2} implies that $(V\times\pi)(\mathcal{T}_{+}(Ee))=\mathcal{T}_{+}(E)\varphi_{\infty}(e)$.
So $\overline{(V\times\pi)(\mathcal{T}_{+}(Ee))}^{u-w}=H^{\infty}(E)\varphi_{\infty}(e)$.
Consequently, from Lemma \ref{L2}, we conclude that $\overline{\tau^{\mathcal{F}(E)}(H^{\infty}(E)\varphi_{\infty}(e)}^{u-w}=U\,\tau_{e}^{\mathcal{F}(Ee)}(H^{\infty}(Ee))\, U^{*}.$

Thus, given $Z\in\overline{\tau^{\mathcal{F}(E)}(H^{\infty}(E))}^{u-w}$,
there is a net $\{X_{n}\}$ of elements of $\mathcal{T}_{+}(Ee)$
that converges ultra-weakly to $X\in H^{\infty}(Ee)$ such that $\norm{X_{n}}\leq\norm{X}$
for all $n$ and $Z=U\tau_{e}^{\mathcal{F}(Ee)}(X)U^{*}$. The net
$\{(V\times\pi)(X_{n})\}$ is bounded in $\mathcal{T}_{+}(E)\varphi_{\infty}(e)$
and so has a subnet $\{(V\times\pi)(X_{n_{\alpha}})\}$ that converges
ultra-weakly to some $Y\in H^{\infty}(E)\varphi_{\infty}(e)$, with
$\norm{Y}\leq\norm{X}$. However, from the ultra-weak continuity of
the maps $\tau^{\mathcal{F}(E)}$ and $\tau_{e}^{\mathcal{F}(Ee)}$,
we conclude that \[
Z=U\tau_{e}^{\mathcal{F}(Ee)}(X)U^{*}=\lim U\tau_{e}^{\mathcal{F}(Ee)}(X_{n_{\alpha}})U^{*}=\lim\tau^{\mathcal{F}(E)}(X_{n_{\alpha}})\]
 \[
=\tau^{\mathcal{F}(E)}(Y)\]
 belongs to $\tau^{\mathcal{F}(E)}(H^{\infty}(E)\varphi_{\infty}(e)).$
This shows that $\tau^{\mathcal{F}(E)}(H^{\infty}(E)$) is ultra-weakly
closed. We also have $\norm{Y}\leq\norm{X}=\norm{Z}\leq\norm{\tau^{\mathcal{F}(E)}(Y)}\leq\norm{Y}$.
Thus $\norm{Y}=\norm{\tau^{\mathcal{F}(E)}(Y)}$. If we started with
a given $Y\in H^{\infty}(E)\varphi_{\infty}(e)$ and $Z=\tau^{\mathcal{F}(E)}(Y)$,
we would be able to conclude that $\norm{Y}=\norm{Z}$ so that the
map $\tau^{\mathcal{F}(E)}$, restricted to $H^{\infty}(E)\varphi_{\infty}(e)$,
is isometric. One can argue similarly to show that it is completely
isometric. $\Box$

\section{The Structure Theorem\label{sec:The-Structure-Theorem}}

We have two goals in this section. The first is Theorem \ref{structure},
which is a generalization of \cite[Theorem 2.6]{DKP}, that Davidson,
Katsoulis and Pitts call the \emph{Structure Theorem}.

Suppose, then, that $(S,\sigma)$ is an isometric representation of
$(E,M)$ acting on a Hilbert space $H$. We shall write $\mathcal{S}$
for the ultra-weak closure of $(S\times\sigma)(\mathcal{T}_{+}(E))$.
It is thus the ultra-weakly closed algebra generated by the operators
$\{\sigma(a),S(\xi)\mid\xi\in E,\; a\in M\}$. Also, we shall write
$\mathcal{S}_{0}$ for the ultra-weakly closed algebra generated by
$\{S(\xi)\mid\xi\in E\}$. Evidently, $\mathcal{S}_{0}$ is an ultra-weakly
closed, two-sided ideal in $\mathcal{S}$. The decomposition that
Davidson, Katsoulis and Pitts advanced centers on understanding the
position of $\mathcal{S}_{0}$ in $\mathcal{S}$. In particular, it
is important to know when $\mathcal{S}_{0}$ is a proper ideal of
$\mathcal{S}$. Our analysis follows a similar route, but it is made
more complicated by the presence of $M$. Observe that since $\mathcal{S}_{0}$
is an ultra-weakly closed $2$-sided ideal in $\mathcal{S}$, $\sigma^{-1}(\mathcal{S}_{0})$
is an ultra-weakly closed two sided ideal in $M$ and hence is of
the form $pM$ for a suitable projection in the center of $M$.

\begin{definition}The projection $e$ in $M$ with the property that
$\sigma^{-1}(\mathcal{S}_{0})=e^{\perp}M$ is called the \emph{model
projection} for $\mathcal{S}$ (or for $(S,\sigma)).$\end{definition}

The reason for terminology will become clear through a brief outline
for our analysis that we hope will be helpful when following our arguments.
As in \cite{DKP}, we let $N$ be the von Neumann algebra generated
by $\mathcal{S}$. We will see in Proposition \ref{ideal} that $\bigcap_{k=1}^{\infty}\mathcal{S}_{0}^{k}$
is a left ideal in $N$. Consequently, there is a projection $P\in\mathcal{S}\cap\sigma(M)'$
such that $\bigcap_{k=1}^{\infty}\mathcal{S}_{0}^{k}=NP=SP$. Following
\cite{DKP} we call this projection \emph{the structure projection}
for $\mathcal{S}$. We will show in Theorem \ref{structure} that
$(I-P)H=P^{\perp}H$ is invariant under $\mathcal{S}$ and that $\mathcal{S}=NP+P^{\perp}\mathcal{S}P^{\perp}$.
That is, in matrix form,\[
\mathcal{S}=\left[\begin{array}{cc}
PNP & 0\\
P^{\perp}NP & P^{\perp}\mathcal{S}P^{\perp}\end{array}\right].\]
We will see in Lemma \ref{p} that $e$ is $E$-saturated and that
it is the support projection for an induced representation $\tau^{\mathcal{F}(E)}$
that we will soon describe. We will see in Lemma \ref{pe} that $\sigma(e)$
is the central support of $P^{\perp}$ in $\sigma(M)$. In part (3)
of Theorem \ref{structure} we will see that $P^{\perp}\mathcal{S}P^{\perp}$
is completely isometrically isomorphic and ultra-weakly homeomorphic
to $\tau^{\mathcal{F}(E)}(H^{\infty}(E))$. Thus $\tau^{\mathcal{F}(E)}(H^{\infty}(E))$
serves as a model for $P^{\perp}\mathcal{S}P^{\perp}$. Further, it
will follow from Theorem \ref{Q} that if $\sigma(e^{\perp})=P=0$,
then the representation $(S,\sigma)$ is absolutely continuous.

Concerning this last statement, one might be inclined at first to
believe that once one knows that $\mathcal{S}$ is completely isometrically
isomorphic and ultra-weakly homeomorphic to the ultra-weak closure
of the range of an induced representation, then $(S,\sigma)$, itself,
must be an induced representation, but examples constructed by Davidson,
Katsoulis and Pitts in \cite[Section 3]{DKP} show this is not the
case. Thus our Theorem \ref{Q} leads to a much more inclusive result
and one may wonder about a converse: Does every absolutely continuous,
isometric representation $(S,\sigma)$ have a vanishing structure
projection? There answer, in general, is {}``no''. Indeed, even
in the case when $A=\mathbb{C}=E$, the answer is ``no'' for classical
reasons. In this situation, it follows from Szeg\"{o}'s theorem that
if $S(1)$ is an absolutely continuous unitary operator whose spectrum
does not cover the circle, then $P=I$. It is of interest to determine
more precisely when this may happen in the general setting.

To proceed with the details, we fix the isometric representation $(S,\sigma)$
and we begin by analyzing the model projection $e$. For this purpose
we observe that since we are working with ultra-weakly closed algebras,
ultra-weakly closed ideals and related constructs involving the ultra-weak
topology, it will be convenient to employ an infinite multiple $(S^{(\infty)},\sigma^{(\infty)})$
of $(S,\sigma)$, which acts on the countably infinite direct sum
of copies of $H$, $H^{(\infty)}$. In general, for $x\in B(H)$,
we write $x^{(\infty)}$ for its infinite ampliation in $B(H^{(\infty)})$.
Likewise for a space of operators $X\subseteq B(H)$, we write $X^{(\infty)}$
for $\{x^{(\infty)}\mid x\in X\}$. Note that for $x\in H^{\infty}(E)$,
we have $S^{(\infty)}\times\sigma^{(\infty)}(x)=(S\times\sigma(x))^{\infty}$.
In particular, $S^{(\infty)}(\xi)=S(\xi)^{(\infty)}$, for all $\xi\in E$.

We are especially interested in the set $\mathfrak{M}$ that consists
of all subspaces $\mathcal{M}$ of $H^{(\infty)}$ that reduce $\sigma^{(\infty)}(M)$
and are \emph{wandering subspaces} for $S^{(\infty)}$ in the sense
that $\mathcal{M}$, $[S^{(\infty)}(E)\mathcal{M}]$, $[S^{(\infty)}(E)S^{(\infty)}(E)\mathcal{M}]$,
etc. are all orthogonal, and their direct sum is a subspace of $H^{(\infty)}$
that is invariant for the algebra $S^{(\infty)}\times\sigma^{(\infty)}(H^{\infty}(E))$.
This subspace must be invariant for the ultra-weak closure of $S^{(\infty)}\times\sigma^{(\infty)}(H^{\infty}(E))$,
which is the same as $\mathcal{S}^{(\infty)}$. It is easily seen,
then, that $\mathcal{M}\oplus[S^{(\infty)}(E)\mathcal{M}]\oplus[S^{(\infty)}(E)S^{(\infty)}(E)\mathcal{M}]\oplus\cdots=[\mathcal{S}^{\infty}(\mathcal{M})]$.
For $\mathcal{M}$ in $\mathfrak{M}$, we denote the restriction of
$(S^{(\infty)},\sigma^{(\infty)})$ to $[\mathcal{S}^{(\infty)}(\mathcal{M})]$
by $(S_{\mathcal{M}}^{(\infty)},\sigma_{\mathcal{M}}^{(\infty)})$.
Evidently, $(S_{\mathcal{M}}^{(\infty)},\sigma_{\mathcal{M}}^{(\infty)})$
is unitarily equivalent to the induced representation $\tau_{\mathcal{M}}^{\mathcal{F}(E)}$
where $\tau_{\mathcal{M}}$ is the restriction of $\sigma^{(\infty)}$
to $\mathcal{M}$. In more detail, since $\mathcal{M}\oplus[S^{(\infty)}(E)\mathcal{M}]\oplus[S^{(\infty)}(E)S^{(\infty)}(E)\mathcal{M}]\oplus\cdots=[\mathcal{S}^{\infty}(\mathcal{M})]$,
the map which takes $h_{0}\oplus S^{(\infty)}(\xi_{1})h_{1}\oplus S^{(\infty)}(\xi_{21})S^{(\infty)}(\xi_{22})h_{2}\oplus\cdots\oplus S^{(\infty)}(\xi_{n1})S^{(\infty)}(\xi_{n2})\cdots S^{(\infty)}(\xi_{nn})h_{n}$
in $[\mathcal{S}^{\infty}(\mathcal{M})]$ to $h_{0}\oplus\xi_{1}\otimes h_{1}\oplus\xi_{21}\otimes\xi_{22}\otimes h_{2}\oplus\cdots\oplus\xi_{n1}\otimes\xi_{n2}\otimes\cdots\xi_{nn}\otimes h_{n}$
in $\mathcal{F}(E)\otimes_{\tau_{\mathcal{M}}}\mathcal{M}$ is well
defined and extends to a Hilbert space isomorphism $U_{\mathcal{M}}:[\mathcal{S}^{(\infty)}(\cM)]\rightarrow\cF(E)\otimes_{\tau_{\cM}}\cM$
such that \[
U_{\cM}^{*}(\varphi_{\infty}(a)\otimes I_{\cM})U_{\cM}=\sigma^{(\infty)}(a)|[\mathcal{S}^{(\infty)}\cM]\;,\;\; a\in M,\]
 and \[
U_{\cM}^{*}(T_{\xi}\otimes I_{\cM})U_{\cM}=S^{(\infty)}(\xi)|[\mathcal{S}^{(\infty)}\cM]\;,\;\;\xi\in E.\]
 (See \cite[Remark 7.7]{MSHardy}).

Although it leads to non-separable spaces, in general, and introduces
a lot of redundancy into our analysis, it is convenient to let $\mathcal{E}$
be the external direct sum of all the spaces $\mathcal{M}$ in $\mathfrak{M}$,
writing \[
\mathcal{E}:=\sum_{\mathcal{M}\in\mathfrak{M}}^{\oplus}\mathcal{M},\]
 and setting \begin{equation}
\tau:=\sum_{\mathcal{M}\in\mathfrak{M}}^{\oplus}\tau_{\mathcal{M}}.\label{tau}\end{equation}
 Then $\tau^{\mathcal{F}(E)}$ gives representations of $\mathcal{T}_{+}(E)$
and $H^{\infty}(E)$ on $\mathcal{F}(E)\otimes_{\tau}\mathcal{E}$.
Further, if we write $U:=\sum\oplus U_{\cM}$, then $U$ implements
a unitary equivalence between $\tau^{\mathcal{F}(E)}$ and $\sum_{\mathcal{M}\in\mathfrak{M}}\oplus(S_{\mathcal{M}}^{(\infty)}\times\sigma_{\mathcal{M}}^{(\infty)})=\sum_{\mathcal{M}\in\mathfrak{M}}\oplus(S^{(\infty)}\times\sigma^{(\infty)})|[\mathcal{S}^{(\infty)}\cM]$.
The map

\[
\Phi:B(H)\rightarrow B(\cF(E)\otimes_{\tau}\cE)\]
 defined by the formula \begin{equation}
\Phi(x)=U(\sum\oplus P_{[\mathcal{S}^{(\infty)}\cM]}x^{(\infty)}|[\mathcal{S}^{(\infty)}\cM])U^{*},\label{Phi}\end{equation}
where $P_{[\mathcal{S}^{(\infty)}\cM]}$ is the projection of $H^{(\infty)}$
onto $[\mathcal{S}^{(\infty)}\cM]$, is a completely positive map
that is continuous with respect to the ultra-weak topologies. When
restricted to $\mathcal{S}$, \begin{equation}
\Phi(x)=U(\sum\oplus x^{(\infty)}|[\mathcal{S}^{(\infty)}\cM])U^{*}\;,\;\; x\in\mathcal{S},\label{PhiS}\end{equation}
 so this restriction, also denoted $\Phi$, is a homomorphism that
satisfies $\Phi(\sigma(a))=\varphi_{\infty}(a)$ for $a\in M$ and
$\Phi(S(\xi))=T_{\xi}\otimes I_{\cE}$ for $\xi\in E$. Thus $\Phi(\mathcal{S})$
is ultra-weakly dense in the ultra-weak closure of $\tau^{\mathcal{F}(E)}(\mathcal{T}_{+}(E))$.
(Note that each element in this ultra-weak closure can be approximated
by Cesaro means calculated with respect to the gauge automorphism
group, and each such mean lies in $\Phi(\mathcal{S})$).

\begin{lemma}\label{p} The model projection $e$ for $\mathcal{S}$
is the support projection of $\tau$ and is $E$-saturated. \end{lemma}
\begin{proof} Recall that by definition, $e^{\perp}$ is the central
cover of $\{a\in M\mid\sigma(a)\in\mathcal{S}_{0}\}$, which is an
ultra-weakly closed ideal in $M$. Since $\ker(\tau)$ is also an
ultra-weakly closed ideal in $M$, it suffices to show that each contains
the same projections $p$. If $\sigma(p)\in\mathcal{S}_{0}$, then
for $\mathcal{M}\in\mathfrak{M}$, $\tau_{\mathcal{M}}(p)\mathcal{M}=\sigma^{(\infty)}(p)\mathcal{M}\subseteq\mathcal{S}_{0}(\mathcal{M})\subseteq\mathcal{M}^{\perp}$.
But $\tau_{\mathcal{M}}(p)\mathcal{M}\subseteq\mathcal{M}$. Thus
$\tau_{\mathcal{M}}(p)=0$. Since $\tau$ is defined to be $\sum_{\mathcal{M}\in\mathfrak{M}}\tau_{\mathcal{M}}$,
$p\in\ker(\tau)$. On the other hand, if $\sigma(p)$ is not in $\mathcal{S}_{0}$,
then there is an ultra-weakly continuous linear functional $f$ such
that $f(\sigma(p))\neq0$ and $f(a)=0$ for $a\in\mathcal{S}_{0}$.
But then we can write $f(x)=\langle x^{(\infty)}\xi,\eta\rangle$
for suitable vectors $\xi,\eta\in H^{(\infty)}$, and find that $\eta$
is orthogonal to the $\sigma^{(\infty)}(M)$-invariant subspace $[(\mathcal{S}_{0})^{(\infty)}\xi]$,
on the one hand, but is not orthogonal to $\sigma^{(\infty)}(p)\xi$,
on the other. Consequently, the space $\mathcal{M}:=[\sigma^{(\infty)}(p)\mathcal{S}^{(\infty)}\xi]\ominus[\sigma^{(\infty)}(p)\mathcal{S}_{0}^{(\infty)}\xi]$
is non zero, and by construction, $\mathcal{M}$ is a $\sigma^{(\infty)}(M)$-invariant
wandering subspace. That is, $\mathcal{M}$ lies in $\mathfrak{M}$,
and evidently, $\tau_{\mathcal{M}}(p)=I_{\mathcal{M}}\neq0$. Thus
$\ker(\tau)=\{a\in M\mid\sigma(a)\in\mathcal{S}_{0}\}$.

To see that $e$ is $E$-saturated, i.e., to see that $\varphi(e^{\perp})\xi e=0$
for all $\xi\in E$, note that $\sigma(e^{\perp})$ lies in $\mathcal{S}_{0}$.
Thus, for every $\xi,\eta\in E$, $S(\eta)^{*}\sigma(e^{\perp})S(\xi)\in\mathcal{S}_{0}$.
But $S(\eta)^{*}\sigma(e^{\perp})S(\xi)=\sigma(\langle\eta,\varphi(e^{\perp})\xi\rangle)$
and so it follows from the definition of $e$ that $\langle\eta,\varphi(e^{\perp})\xi e\rangle=\langle\eta,\varphi(e^{\perp})\xi\rangle e=0$
for all $\xi,\eta\in E$. Thus $\varphi(e^{\perp})\xi e=0$ for all
$\xi\in E$. \end{proof}

The representation $\tau$ acts on a non-separable space, so it may
be comforting to know that for all intents and purposes we have in
mind, it is possible to replace $\tau$ with a representation on a
separable space. This an immediate consequence of the following proposition.

\begin{proposition}\label{prop:Iso_induced} Suppose $\tau_{1}$
and $\tau_{2}$ are two normal representations of $M$ on Hilbert
spaces $H_{1}$ and $H_{2}$, respectively, and suppose that $\ker(\tau_{1})=\ker(\tau_{2})$.
Then the map $\Psi$ that sends $\tau_{1}^{\mathcal{F}(E)}(R)$ to
$\tau_{2}^{\mathcal{F}(E)}(R)$, $R\in\mathcal{L}(\mathcal{F}(E))$,
is a normal $*$-isomorphism from the the von Neumann algebra $\cL(\cF(E))\otimes_{\tau_{1}}I_{H_{1}}$
onto the von Neumann algebra $\cL(\cF(E))\otimes_{\tau_{2}}I_{H_{2}}$.
Further, the restriction of $\Psi$ to the ultra-weak closure of $\tau_{1}^{\mathcal{F}(E)}(H^{\infty}(E))$
is a completely isometric, ultra-weak homeomorphism from $\tau_{1}^{\mathcal{F}(E)}(H^{\infty}(E))$
onto the ultra-weak closure of $\tau_{2}^{\mathcal{F}(E)}(H^{\infty}(E))$.\end{proposition}
\begin{proof} For $R\in\mathcal{L}(\mathcal{F}(E))$, $\xi\in\mathcal{F}(E)$
and $h\in H_{i}$, $i=1,2$, $\langle R\xi\otimes h,R\xi\otimes h\rangle=\langle h,\tau_{i}(\langle R\xi,R\xi\rangle)h\rangle$.
Thus $R\otimes_{\tau_{i}}I_{H_{i}}=0$ if and only if $\langle R\xi,R\xi\rangle\subseteq\ker(\tau_{i})$
for all $\xi\in E$. Since we assume that $\ker(\tau_{1})=\ker(\tau_{2})$,
the map $\Psi$ is a well defined, injective map. It is clearly a
$^{*}$-homomorphism and normality is also easy to check.\end{proof}

{}

We turn next to the problem of identifying the structure projection
$P$ for $\mathcal{S}$.

\begin{proposition}\label{kerphi} If $\Phi$ is the map defined
by equation (\ref{Phi}), then \[
\mathcal{S}\cap\ker(\Phi)=\bigcap_{k=1}^{\infty}\mathcal{S}_{0}^{k}.\]
 \end{proposition} \begin{proof} For $x\in\mathcal{S}_{0}^{k}$
and $\cM\in\mathfrak{M}$, $x^{(\infty)}[\mathcal{S}^{(\infty)}\cM]\subseteq[(\mathcal{S}_{0}^{k})^{(\infty)}\cM]$.
Since $\bigcap_{k}[(\mathcal{S}_{0}^{k})^{(\infty)}\cM]=\{0\}$, we
find that $\bigcap_{k=1}^{\infty}\mathcal{S}_{0}^{k}\subseteq\ker(\Phi)$.
On the other hand, if $x\in\mathcal{S}$ is not in $\bigcap_{k=1}^{\infty}\mathcal{S}_{0}^{k}$,
then there is a $k\geq1$ such that $x$ is not in $\mathcal{S}_{0}^{k}$.
Consequently, there is an ultra-weakly continuous linear functional
$f$ such that $f$ vanishes on $\mathcal{S}_{0}^{k}$ but $f(x)\neq0$.
Since $f$ is ultra-weakly continuous, we may find vectors $\xi,\eta$
in $H^{(\infty)}$ such that $f(y)=\langle y^{(\infty)}\xi,\eta\rangle$
for all $y\in B(H)$. It follows that $x^{(\infty)}\xi$ is not in
$[(\mathcal{S}_{0}^{k})^{(\infty)}]$, proving that $\mathcal{N}:=[\mathcal{S}^{(\infty)}\xi]\ominus[(\mathcal{S}_{0}^{k})^{(\infty)}\xi]\neq\{0\}$.
In fact, we find that \[
x^{(\infty)}\mathcal{N}\nsubseteqq[(\mathcal{S}_{0}^{k})^{(\infty)}].\]
 We can write $\mathcal{N}$ as a direct sum of wandering spaces in
$\mathfrak{M}$: \[
\mathcal{N}=([\mathcal{S}^{(\infty)}\xi]\ominus[(\mathcal{S}_{0})^{(\infty)}\xi])\oplus([(\mathcal{S}_{0})^{(\infty)}\xi]\ominus[(\mathcal{S}_{0}^{2})^{(\infty)}])\oplus\cdots\]
 \[
\oplus([(\mathcal{S}_{0}^{k-1})^{(\infty)}\xi]\ominus[(\mathcal{S}_{0}^{k})^{(\infty)}]),\]
 which shows that $x$ is not in the kernel of $\Phi$. \end{proof}

\begin{proposition}\label{ideal} The space $\bigcap_{k=1}^{\infty}\mathcal{S}_{0}^{k}$
is an ultra-weakly closed left ideal in the von Neumann algebra $N$
generated by $\mathcal{S}$. Thus \begin{equation}
\bigcap_{k=1}^{\infty}\mathcal{S}_{0}^{k}=NP=\mathcal{S}P\label{eq:corner}\end{equation}
 for some projection $P\in\mathcal{S}\cap\sigma(M)'$. \end{proposition}
Recall that $P$ is called the structure projection for $\mathcal{S}$
(and for $(S,\sigma)$.)

\begin{proof} We first claim that the map $\Phi$ defined in equation
(\ref{Phi}) satisfies the equation \[
\Phi(S(\xi)^{*}R)=(T_{\xi}^{*}\otimes I_{\mathcal{E}})\Phi(R)\;,\;\;\xi\in E,\; R\in\mathcal{S}.\]
 Since $\Phi$ is an ultra-weakly continuous linear map, it suffices
to prove the claim for $R=\sigma(a),\; a\in M$, and for $R=S(\eta_{1})S(\eta_{2})\cdots S(\eta_{k})$
for $\eta_{1},\eta_{2},\ldots\eta_{k}\in E$. For $R=\sigma(a)$,
$\Phi(\sigma(a))=\varphi_{\infty}(a)\otimes I_{\mathcal{E}}$ and,
thus both sides of the equation are equal to $0$. For $R=S(\eta_{1})S(\eta_{2})\cdots S(\eta_{k})$,
we have \begin{eqnarray*}
\Phi(S(\xi)^{*}R) & = & \Phi(\sigma(\langle\xi,\eta_{1}\rangle)S(\eta_{2})\cdots S(\eta_{k}))\\
 & = & (\varphi_{\infty}(\langle\xi,\eta_{1}\rangle)\otimes I_{\mathcal{E}})\Phi(S(\eta_{2})\cdots S(\eta_{k}))\\
 & = & (T_{\xi}^{*}\otimes I_{\mathcal{E}})(T_{\eta_{1}}\otimes I_{\mathcal{E}})\Phi(S(\eta_{2})\cdots S(\eta_{k}))\\
 & = & (T_{\xi}^{*}\otimes I_{\mathcal{E}})\Phi(R),\end{eqnarray*}
 which proves the claim. Consequently, for $R\in\bigcap_{k=1}^{\infty}\mathcal{S}_{0}^{k}$
($=\ker\Phi\cap\mathcal{S}$), $\Phi(S(\xi)^{*}R)=0$ and, thus, $S(\xi)^{*}R\in\bigcap_{k=1}^{\infty}\mathcal{S}_{0}^{k}$.
Clearly $S(\xi)R\in\bigcap_{k=1}^{\infty}\mathcal{S}_{0}^{k}$ for
such $R$ and, therefore, $\cap_{k=1}^{\infty}\mathcal{S}_{0}^{k}$
is an ideal in $N$. Since $\ker(\Phi)\cap\mathcal{S}$ is an ideal
in $\mathcal{S}$, it follows that $P\in\sigma(M)'$. \end{proof}

\begin{lemma}\label{pe} The structure projection $P$ and the model
projection $e$ for the isometric representation $(S,\sigma)$ are
related: $e$ is the smallest projection in $M$ satisfying \[
\sigma(e)=\bigvee\{v(I-P)v^{*}:\; v\;\emph{is a unitary in}\;\sigma(M)'\}.\]
In particular, $\sigma(e)$ is the central support in $\sigma(M)$
of $P^{\perp}$. So, if $P=0$, then $e$ is the support projection
of $\sigma$; thus if $P=0$ and the restriction of $S\times\sigma$
to $M$ is faithful, then $e=I$. \end{lemma} \begin{proof} If $z$
is a projection in $M$ with $\sigma(z)\in\mathcal{S}_{0}$ then,
for $k\geq1$, $\sigma(z)=\sigma(z)^{k}\in\mathcal{S}_{0}^{k}$ and,
thus, $\sigma(z)\in NP$ and $\sigma(z)\leq P$. The converse also
holds since if $\sigma(z)=\sigma(z)P$, then $\sigma(z)\in\mathcal{S}_{0}$.
Thus it follows from Lemma~\ref{p} that $e^{\perp}$ is the largest
projection $z$ in $M$ such $\sigma(z)\leq P$. Equivalently, $e$
is the smallest projection $p$ in $M$ such that $\sigma(p)\geq I-P$.
The statement of the lemma is now immediate. \end{proof}

{}

\begin{theorem}\label{thm1.1} Let $(S,\sigma)$ be an isometric
representation of $(E,M)$, write $\mathcal{S}$ for the ultra-weak
closure of $(S\times\sigma)(\mathcal{T}_{+}(E))$ and let $e$ be
the model projection for $\mathcal{S}$. Suppose $\tau$ is a normal
representation of $M$ on a Hilbert space $K$ and that the support
projection of $\tau$ is $e$. Suppose, also, that there is given
an ultra-weakly continuous, completely contractive homomorphism $\psi$
from $\mathcal{S}$ to $\tau^{\mathcal{F}(E)}(H^{\infty}(E))$ such
that $\psi\circ(S\times\sigma)=\tau^{\mathcal{F}(E)}$ on $\mathcal{T}_{+}(E)$.
Then $\psi$ is surjective and the map $\widetilde{\psi}:\mathcal{S}/\ker(\psi)\to\tau^{\mathcal{F}(E)}(H^{\infty}(E))$
obtained by passing to the quotient is a complete isometry and a homeomorphism
with respect to the ultra-weak topologies on $\mathcal{S}/\ker(\psi)$
and $\tau^{\mathcal{F}(E)}(H^{\infty}(E)).$ \end{theorem}

\begin{proof} Recall the unitary group $\{W_{t}\}$ in $\mathcal{L}(\mathcal{F}(E))$,
and the gauge automorphism group $\{\gamma_{t}\}$, $\gamma_{t}=AdW_{t}$,
discussed in paragraph \ref{par:GaugeAutos}. Recall in particular
that \[
\Sigma_{k}(a):=\frac{1}{2\pi}\int_{0}^{2\pi}(\sum_{|j|<k}(1-\frac{|j|}{k})e^{-int})\gamma_{t}(a)dt\]
 for $a\in\mathcal{L}(\mathcal{F}(E))$. We write $W_{t}^{(n)}$ for
$diag(W_{t},\ldots,W_{t})$, $\gamma_{t}^{(n)}$ for $AdW_{t}^{(n)}$
and \[
\Sigma_{k}^{(n)}(A):=\frac{1}{2\pi}\int_{0}^{2\pi}(\sum_{|j|<k}(1-\frac{|j|}{k})e^{-ijt})\gamma_{t}^{(n)}(A)dt,\]
 for $A\in M_{n}(\mathcal{L}(\mathcal{F}(E)))$. We find that $\Sigma_{k}^{(n)}$
is a contractive map and $\Sigma_{k}^{(n)}(A)\rightarrow A$ in the
ultra-weak topology as $k\rightarrow\infty$.

Now fix $A\in M_{n}(\varphi_{\infty}(e)H^{\infty}(E)\varphi_{\infty}(e))$
and write $X_{k}:=(S\times\sigma)^{(n)}(\Sigma_{k}^{(n)}(A))$. Note
that $\Sigma_{k}^{(n)}(A)\in M_{n}(\mathcal{T}_{+}(E))$ so that the
last expression is well defined and $\psi^{(n)}(X_{k})=(\tau^{\mathcal{F}(E)})^{(n)}(\Sigma_{k}^{(n)}(A))$.
Thus, $\norm{X_{k}}=\norm{(S\times\sigma)^{(n)}(\Sigma_{k}^{(n)}(A))}\leq\norm{\Sigma_{k}^{(n)}(A)}\leq\norm{A}$
(here we used the fact that $(S\times\sigma)^{(n)}$ is contractive
on $\varphi_{\infty}(e)\mathcal{T}_{+}(E)\varphi_{\infty}(e)$ and
$\Sigma_{k}^{(n)}(A)$ lies in $M_{n}(\varphi_{\infty}(e)\mathcal{T}_{+}(E)\varphi_{\infty}(e))$
). Now, the sequence $X_{k}$ has a subnet $X_{k_{\alpha}}$ that
is ultra-weakly convergent to some $X\in M_{n}(\mathcal{S})$ with
$\norm{X}\leq\norm{A}$. Since $\psi$ is ultra-weakly continuous,
it follows that \[
\psi^{(n)}(X)=\lim\psi^{(n)}(X_{k_{\alpha}})=\lim(\tau^{\mathcal{F}(E)})^{(n)}(\Sigma_{k_{\alpha}}^{(n)}(A))=(\tau^{\mathcal{F}(E)})^{(n)}(A),\]
 where the limits are all taken in the ultra-weak topology. Thus $\psi^{(n)}$
is surjective for every $n\geq1$.

Since $e$ is $E$-saturated, Theorem~\ref{indsaturated} implies
that $\tau^{\mathcal{F}(E)}$, restricted to $H^{\infty}(E)\varphi_{\infty}(e)$,
is a complete isometry and, thus, \[
\norm{A}\geq\norm{X}\geq\norm{\psi^{(n)}(X)}=\norm{(\tau^{\mathcal{F}(E)})^{(n)}(A)}.\]
 Therefore, in this case we get \[
\norm{X}=\norm{(\tau^{\mathcal{F}(E)})^{(n)}(A)}.\]
 This shows that the map $\widetilde{\psi}$ from $\mathcal{S}/\ker(\psi)$
onto $\tau^{\mathcal{F}(E)}(H^{\infty}(E))$, induced by $\psi$,
is a complete isometry.

The fact that this map is an ultra-weak homeomorphism follows from
the ultra-weak continuity of $\psi$ as in the proof of \cite[Theorem 1.1]{DKP}.
To be more precise, the map $\widetilde{\psi}$ is the dual of a map
$\psi_{*}$ from the predual of $\tau^{\mathcal{F}(E)}(H^{\infty}(E))$
to the predual of $\mathcal{S}/\ker(\psi)$. This map is injective
(as $\widetilde{\psi}$ is surjective ) and contractive and, thus,
has a bounded inverse $(\psi_{*})^{-1}$. The dual of this inverse
is the inverse of $\widetilde{\psi}$. Thus $\widetilde{\psi}$ is
an ultra-weak homeomorphism. \end{proof}

The following theorem is our analogue of \cite[Theorem 2.6]{DKP}.

\begin{theorem}\label{structure} Let $(S,\sigma)$ be an isometric
representation of $(E,M)$ on $H$, let $\mathcal{S}$ be the ultra-weak
closure of $(S\times\sigma)(\mathcal{T}_{+}(E))$, let $N$ be the
von Neumann algebra generated by $\mathcal{S}$, and let $P$ be the
structure projection for $\mathcal{S}$. Then in addition to equation
\eqref{eq:corner}, the following assertions hold:
\begin{enumerate}
\item [(1)] $P^{\perp}H$ is invariant for $\mathcal{S}$.
\item [(2)] $\mathcal{S}P^{\perp}=P^{\perp}\mathcal{S}P^{\perp}$.
\item [(3)] $\mathcal{S}=NP+P^{\perp}\mathcal{S}P^{\perp}$.
\item [(4)] If $\tau$ is any normal representation of $M$ whose support
projection is the model projection for $\mathcal{S}$, then the map
$\Psi$ from $(S\times\sigma)(\mathcal{T}_{+}(E))P^{\perp}$ to $\tau^{\mathcal{F}(E)}(H^{\infty}(E))$
defined by the equation \[
\Psi((S\times\sigma)(a)P^{\perp})=\tau^{\mathcal{F}(E)}(a),\qquad a\in\mathcal{T}_{+}(E),\]
 extends to a complete isometric isomorphism from $\mathcal{S}P^{\perp}$
onto $\tau^{\mathcal{F}(E)}(H^{\infty}(E))$ that is a homeomorphism
with respect to the ultra-weak topologies on $\mathcal{S}P^{\perp}$
and $\tau^{\mathcal{F}(E)}(H^{\infty}(E))$.
\end{enumerate}
\end{theorem}

\begin{proof} Suppose first that $\tau$ is the representation defined
in (\ref{tau}) and that $\Phi$ is the map defined in (\ref{PhiS}).
Since $\ker(\Phi)$ is an ideal in $\mathcal{S}$ and the structure
projection, $P$, belongs to it, $P\mathcal{S}P^{\perp}\subseteq\ker(\Phi)=NP$.
But then $P\mathcal{S}P^{\perp}=0$. Thus $P^{\perp}H$ is invariant
for $\mathcal{S}$ and $\mathcal{S}P^{\perp}=P^{\perp}\mathcal{S}P^{\perp}$.

Now, $\Phi$ is a completely contractive, ultra-weakly continuous,
homomorphism mapping $\mathcal{S}$ into the $\sigma$-weak closure
of $\tau^{\mathcal{F}(E)}(\mathcal{T}_{+}(E))$. So we may apply Theorem
\ref{thm1.1} to $\Phi$ (which plays the role of $\psi$ in the statement
of Theorem \ref{thm1.1}) to conclude that the map induced by $\Phi$
on $\mathcal{S}P^{\perp}\simeq\mathcal{S}/\ker(\Phi)$ is completely
isometric and an ultra-weak homeomorphism onto the ultra-weak closure
of $\tau^{\mathcal{F}(E)}(\mathcal{T}_{+}(E))$ - which is $\tau^{\mathcal{F}(E)}(H^{\infty}(E))$)
by Theorem \ref{indsaturated} and the fact that the model projection
for $\mathcal{S}$ is $E$-saturated, Lemma \ref{p}. But once we
have proven the result with the special representation $\tau$ from
\eqref{tau}, we can replace it with any representation with the same
support, viz. the model projection of $\mathcal{S}$, thanks to Proposition
\ref{prop:Iso_induced}. \end{proof}

To connect the structure projection with the absolutely continuous
subspace of an isometric representation, we employ the universal isometric
representation $(S_{0},\sigma_{0})$ from paragraph \ref{par:UnivIIRs}
and use results developed in Section \ref{sec:Absolute-Continuity}.
Recall that $(S_{0},\sigma_{0})$ is the representation induced by
a representation $\pi$ on a Hilbert space $K_{0}$, and so $(S_{0},\sigma_{0})$
acts on the Hilbert space $\mathcal{F}(E)\otimes_{\pi}K_{0}$.

\begin{theorem}\label{Q} Let $(S,\sigma)$ be an isometric representation
of $(E,M)$ on a Hilbert space $H$, let $P$ be its structure projection,
and let $Q$ be the structure projection for the representation $(S\times\sigma)\oplus(S_{0}\times\sigma_{0})$,
acting on $H\oplus\mathcal{F}(E)\otimes_{\pi}K_{0}$. Then
\begin{enumerate}
\item [(1)] $\mathcal{V}_{ac}(S,\sigma)=H\ominus Q(H\oplus(\mathcal{F}(E)\otimes_{\pi}K_{0}))$,
\item [(2)] $P^{\perp}H\subseteq\mathcal{V}_{ac}(S,\sigma)$, and
\item [(3)] the range of $Q$ is contained in the range of $P$, viewed
as a subspace of $H\oplus(\mathcal{F}(E)\otimes_{\pi}K_{0})$.
\end{enumerate}
\end{theorem} \begin{proof} We shall write $\mathcal{S}$ for the
ultra-weak closure of $S\times\sigma(\mathcal{T}_{+}(E))$. To simplify
the notation, we shall write $\rho$ for $(S\times\sigma)\oplus(S_{0}\times\sigma_{0})$
and we shall write $\mathcal{S}_{\rho}$ for the ultra-weak closure
of $\rho(\mathcal{T}_{+}(E))$. We first want to prove that the range
of $Q$ is contained in $H$. For this purpose and for later use,
we shall write $P_{H}$ for the projection of $H\oplus(\mathcal{F}(E)\otimes_{\pi}K_{0})$
onto $H$. Since $\rho(a)$ commutes with $P_{H}$, for every $a\in\mathcal{T}_{+}(E)$,
the von Neumann algebra generated by $\mathcal{S}_{\rho}$, $N_{\rho}$,
commutes with $P_{H}$. In particular, $Q$ commutes with $P_{H}$.
So it suffices to show that if $\xi$ is a vector in $H\oplus(\mathcal{F}(E)\otimes_{\pi}K_{0})$
such that $\xi=Q\xi$ lies in $H^{\perp}$, then $\xi=0$. So fix
such a vector $\xi$. Then $\xi\in\mathcal{F}(E)\otimes_{\pi}K_{0}$
and, for every $\eta\in E^{\otimes m}$, $\rho(T_{\eta})\xi\in(E^{\otimes m}\oplus E^{\otimes(m+1)}\oplus\cdots)\otimes_{\pi}K_{0}$.
Thus, for every $X\in(\mathcal{S}_{\rho})_{0}^{m}$, $X\xi\in(E^{\otimes m}\oplus E^{\otimes(m+1)}\oplus\cdots)\otimes_{\pi}K_{0}$.
It follows that, if $X\in\bigcap_{m}(\mathcal{S}_{\rho})_{0}^{m}=N_{\rho}Q$,
then $X\xi=0$. In particular, $Q\xi=0$. This shows that the range
of $Q$ is contained in $H$.

Next observe that the wandering vectors of \emph{any} isometric representation
are orthogonal to the range of its structure projection. To prove
this, we will work with $(S,\sigma)$, but the argument is quite general.
So suppose $\xi\in H$ is a wandering vector for $(S,\sigma)$. Then
$\mathcal{M}_{0}:=[\sigma(M)\xi]$ is a $\sigma(M)$-invariant wandering
subspace for $(S,\sigma)$, $\mathcal{M}:=\{\eta^{(\infty)}:\eta\in\mathcal{M}_{0}\}$
is in $\mathfrak{M}$, and $\xi^{(\infty)}\in[\mathcal{S}^{(\infty)}\mathcal{M}]$.
Since the structure projection $P$ of $(S,\sigma)$ lies in the kernel
of $\Phi$ by Proposition \ref{kerphi}, $P^{(\infty)}|[\mathcal{S}^{(\infty)}\mathcal{M}]=0$
and, thus, $P\xi=0$. Thus the span of the wandering vectors for $\rho$
is orthogonal to the range of $Q$. But by Corollary \ref{wander},
the span of the wandering vectors for $\rho$ is $\mathcal{V}_{ac}(S,\sigma)\cap H$.
Thus $\mathcal{V}_{ac}(S,\sigma)\subseteq P_{H}Q^{\perp}(H\oplus(\mathcal{F}(E)\otimes K))$.

To show the reverse inclusion, suppose $\xi=P_{H}Q^{\perp}\xi$. Then
$\xi=Q^{\perp}\xi$, as $Q\leq P_{H}$. Let $f$ be the linear functional
on $\mathcal{T}_{+}(E)$ defined by the equation $f(a)=\langle\rho(a)\xi,\xi\rangle=\langle\rho(a)Q^{\perp}\xi,Q^{\perp}\xi\rangle$,
$a\in\mathcal{T}_{+}(E)$. We want to use Theorem~\ref{structure}
to show that $\xi$ lies in $\mathcal{V}_{ac}(S,\sigma)$ by showing
that $f$ extends to an ultra-weakly continuous linear functional
on $H^{\infty}(E)$. In the notation of that theorem, replace $S\times\sigma$
with $\rho$ and let $e_{\rho}$ be the model projection for $\rho$.
Also, let $\tau$ be any normal representation of $M$ with support
equal $e_{\rho}$. Then by part (4) of Theorem \ref{structure} there
is a completely isometric isomorphism $\Psi$ from $\mathcal{S}_{\rho}Q^{\perp}$
onto $\tau^{\mathcal{F}(E)}(H^{\infty}(E))$ that is an ultra-weak
homeomorphism such that $\Psi(\rho(a)Q^{\perp})=\tau^{\mathcal{F}(E)}(a)$
for all $a\in H^{\infty}(E)$. We conclude, then, that $f(a)=\langle\rho(a)Q^{\perp}\xi,Q^{\perp}\xi\rangle=\langle\Psi^{-1}\circ\tau^{\mathcal{F}(E)}(a)\xi,\xi\rangle$,
$a\in\mathcal{T}_{+}(E)$, extends to an ultra-weakly continuous linear
functional on $H^{\infty}(E)$. By Remark~\ref{ac}, $\xi\in\mathcal{V}_{ac}(S,\sigma)$.
Thus we conclude that $\mathcal{V}_{ac}(S,\sigma)=H\ominus Q(H\oplus(\mathcal{F}(E)\otimes_{\pi}K_{0}))$.
This proves point (1).

The argument needed to prove that $P^{\perp}H\subseteq\mathcal{V}_{ac}(S,\sigma)$
is similar. Let $\xi\in P^{\perp}H$ and consider the functional $g(a):=\langle(S\times\sigma)(a)\xi,\xi\rangle=\langle(S\times\sigma)(a)P^{\perp}\xi,P^{\perp}\xi\rangle$
for $a\in\mathcal{T}_{+}(E)$. Using part (4) of Theorem~\ref{structure}
(and the notation there), we may write $g(a)=\langle\Psi^{-1}(\tau_{e}^{\mathcal{F}(E)}(a))\xi,\xi\rangle$
for any normal representation $\tau_{e}$ of $M$ whose support projection
is the model projection for $S\times\sigma$. This shows that $g$
is ultra-weakly continuous and so $\xi\in\mathcal{V}_{ac}(S,\sigma)$.
Thus $P^{\perp}H$ is contained in $\mathcal{V}_{ac}(S,\sigma)$,
proving point (2).

Since $P^{\perp}H\subseteq\mathcal{V}_{ac}(S,\sigma)=H\ominus Q(H\oplus(\mathcal{F}(E)\otimes_{\pi}K_{0}))$,
we see that the range of $P_{H}Q$ is contained in the range of $P$.
Since, however, the range of $Q$ is contained in $H$, as we proved
at the outset, we conclude that the range of $Q$ is contained in
the range of $P$, proving point (3). \end{proof}

As a corollary we conclude with following theorem, which complements
Theorem \ref{abscontrep} in the sense that if $(S,\sigma)$ is an
ultra-weakly continuous, isometric representation of $(E,M)$, then
$(S,\sigma)$ is absolutely continuous if and only if $S\times\sigma\oplus S_{0}\times\sigma_{0}$
is {}``completely isometrically equivalent'' to an induced representation
of $(E,M)$. Thus, $(S,\sigma)$ is absolutely continuous if and only
if it {}``equivalent to a subrepresentation of an induced representation''.

\begin{theorem}\label{thm:Complement_of_abscontrep} Let $(S,\sigma)$
be an ultra-weakly continuous, isometric covariant representation
of $(E,M)$ on a Hilbert space and denote by $\rho$ the representation
$S\times\sigma\oplus S_{0}\times\sigma_{0}$. Then the following assertions
are equivalent:
\begin{enumerate}
\item [(1)] $(S,\sigma)$ is absolutely continuous.
\item [(2)] The structure projection for $\rho$ is zero.
\item [(3)] If $\tau$ is any representation of $M$ whose support projection
is the model projection of $S\times\sigma$, then there is an ultra-weakly
homeomorphic, completely isometric isomorphism $\Psi$ from $\tau^{\mathcal{F}(E)}(H^{\infty}(E))$
onto the ultra-weak closure $\mathcal{S}_{\rho}$ of $\rho(\mathcal{T}_{+}(E))$
such that $\Psi\circ\tau^{\mathcal{F}(E)}(a)=\rho(a)$, for all $a\in\mathcal{T}_{+}(E)$.
\end{enumerate}
\end{theorem} \begin{proof} Let $P$ be the structure projection
for $S\times\sigma$. If $(S,\sigma)$ is absolutely continuous then
$\mathcal{V}_{ac}(S,\sigma)=H$ and it follows from Theorem~\ref{Q}
(1) that the range of $Q$ is orthogonal to $H$. Since, however,
the range of $Q$ is contained in the range of $P$ by part (3) of
that lemma, we conclude that $Q=0$. This proves that (1) implies
(2). The implication, (2) $\Rightarrow$ (3), follows from Theorem~\ref{structure}
(4). Now assume that (3) holds. Then $\Psi^{-1}(\mathcal{S}_{\rho})=\tau^{\mathcal{F}(E)}(H^{\infty}(E)),$
and for $m\geq0$, $\Psi^{-1}((\mathcal{S}_{\rho})_{0}^{m})$ is the
image under $\tau^{\mathcal{F}(E)}$ of the space $H_{m}^{\infty}(E)$,
which consists of all operators $X\in H^{\infty}(E)$ with the property
that $\Phi_{k}(X)=0$, $k\leq m$, where $\Phi_{k}$ is the $k^{th}$
Fourier operator \eqref{FourierOperators}. Then $\Psi^{-1}(\bigcap_{m}(\mathcal{S}_{\rho})_{0}^{m})=\bigcap_{m}\Psi^{-1}((\mathcal{S}_{\rho})_{0}^{m})=\bigcap_{m}\tau^{\mathcal{F}(E)}(H_{m}^{\infty}(E))=\tau^{\mathcal{F}(E)}(\bigcap_{m}H_{m}^{\infty}(E))$.
Since $\bigcap_{m}H_{m}^{\infty}(E)=\{0\}$, $\bigcap_{m}(\mathcal{S}_{\rho})_{0}^{m}=\{0\}$
and so $Q=0$. This proves that (3) implies (2). Since Theorem~\ref{Q}
(1) shows that (2) implies (1), the proof is complete. \end{proof}

\begin{remark}\label{rem:Kennedy_improvement}In \cite{mKp2010},
M. Kennedy deftly uses technology from the theory of dual operator
algebras to show that in the setting when $M=\mathbb{C}$ and $E=\mathbb{C}^{d}$,
the ultra-weak closure of the image of an isometric representation
of $\mathcal{T}_{+}(\mathbb{C}^{d})$ is ultraweakly homeomorphic
and completely isometrically isomorphic to $H^{\infty}(\mathbb{C}^{d})$
if and only if the representation is absolutely continuous. Whether
his result can be generalized to the setting of this paper remains
to be seen. \end{remark}

\end{document}